\documentclass[11pt]{amsart}

\usepackage{amsmath, amssymb, amsthm}
\usepackage{mathtools}
\usepackage{csquotes}
\usepackage[english]{babel}
\usepackage{xcolor}
\usepackage{tikz}
\usepackage{tikz-cd}
\usepackage{comment}
\usepackage{hyperref}


\newcounter{genum}
\newenvironment{Glist}{
\begin{enumerate}
\setcounter{enumi}{\value{genum}}

}{
\setcounter{genum}{\value{enumi}}
\end{enumerate}
}

\newtheorem{theorem}{Theorem}[section]
\newtheorem*{theorem*}{Theorem}
\newtheorem{lemma}[theorem]{Lemma}
\newtheorem{corollary}[theorem]{Corollary}
\newtheorem{claim}[theorem]{Claim}
\newtheorem{proposition}[theorem]{Proposition}

\theoremstyle{definition}
\newtheorem{definition}[theorem]{Definition}
\newtheorem{example}[theorem]{Example}

\theoremstyle{remark}
\newtheorem{remark}[theorem]{Remark}

\numberwithin{equation}{section}

\DeclareMathOperator{\Spec}{Spec}

\DeclareMathOperator{\Hom}{Hom}

\DeclareMathOperator{\Aut}{Aut}

\DeclareMathOperator{\Sym}{Sym}

\DeclareMathOperator{\comp}{\widehat{\cO}}

\newcommand{\Gnum}[1]{\textbf{(G#1)}}
\newcommand{\Pnum}[1]{\textbf{(P#1)}}

\newcommand\cE{\mathcal{E}}
\newcommand\cF{\mathcal{F}}

\newcommand{\cJ}{\mathcal{J}}

\newcommand\cO{\mathcal{O}}
\newcommand\cP{\mathcal{P}}

\newcommand\cS{\mathcal{S}}

\newcommand\cW{\mathcal{W}}

\renewcommand\AA{\mathbb{A}}

\newcommand\CC{\mathbb{C}}

\newcommand\ZZ{\mathbb{Z}}

\newcommand{\frakm}{\mathfrak{m}}
\newcommand{\frakn}{\mathfrak{n}}

\newcommand{\frakp}{\mathfrak{p}}

\begin{document}

\title{An equivariant version of Gabber's lemma}

%    author one information
\author{Filippo Belfiori}
\address{Institut Galil\'ee\\
Universit\'e Sorbonne Paris Nord\\    
93430, Villetaneuse, France}
\curraddr{}
\email{belfiori@math.univ-paris13.fr} 
\thanks{}

%%% \subjclass[2010]{Primary }


\date{}

\dedicatory{}

\begin{abstract}
Let $G$ be a finite abelian group, and let $k$ be an infinite perfect field containing a primitive root of unity of order the exponent of $G$. We prove that the $G$-equivariant version of Gabber's presentation lemma proposed by Bachmann holds over $k$ if and only if $G$ is cyclic of $p$-power order, for some prime $p$.
\end{abstract}

\maketitle

\tableofcontents

\section{Introduction} \label{section: introduction}

This paper is devoted to establishing the equivariant version of Gabber's presentation lemma proposed by Bachmann \cite{Bachmann} for every finite abelian group $G$ that is a cyclic group of order the power of a prime. On the other hand, if $G$ is not of that form, we show that such an equivariant Gabber's presentation lemma cannot hold for $G$. 

\subsection{Gabber's presentation lemma}
Let $k$ be a field. Gabber's geometric presentation lemma \cite[Lemma 3.1]{Gabber} is the following fundamental result in algebraic geometry.

\begin{theorem}[Gabber] \label{theorem: gabber original}
Let $X$ be an affine, irreducible smooth $k$-scheme of dimension $d>0$; let $Z \subset X$ be a closed subscheme of positive codimension and let $z \in Z$. Then there exist:
\begin{itemize}
	\item an open neighborhood $U \subset X$ of $z$;
	\item an open subset $W \subset \AA^{d-1}_k$;
	\item an étale morphism $\varphi  = (\psi, \nu) \colon U \to \AA^1_W = W \times_k \AA^1_k$ such that $\varphi_{|Z \cap U}$ is a closed immersion, the composite
	\[
	\psi_{|Z \cap U} \colon Z \cap U \longrightarrow \AA^1_W \to W
	\]
	is finite and $\varphi$ restricts to an isomorphism
	\[
	\varphi^{-1}(\varphi(Z \cap U)) \xlongrightarrow{\sim} \varphi(Z \cap U).
	\]
\end{itemize}
\end{theorem}

In other words, étale locally we may give a presentation of the pair $(X,Z)$ as $(\AA^1_W, Z \cap U)$, and one could say that $\varphi$ induces an analytic isomorphism along $Z \cap U$. Theorem \ref{theorem: gabber original} has been proved by Gabber over an infinite field, and it has subsequently been extended to finite fields \cite{Hogadi}, and also over more general base schemes \cite{Hogadi2}. Gabber used this preparation lemma in order to establish the exactness of a class of Gersten complexes; see, for example, \cite{Colliot-Thelene} for a complete treatment on the topic.

With the work of Morel \cite{Morel}, Gabber's lemma has nowadays become a cornerstone in motivic homotopy theory over fields. For example, Morel used Gabber's lemma to prove the exactness of the Rost--Schmid complex \cite[Chapter 5]{Morel}, and to prove the strongly (resp.\ strictly) $\AA^1$-invariance of the homotopy sheaf $\underline{\pi}_1X$ (resp.\ $\underline{\pi}_iX$ for $i \geq 2$) of a motivic space $X$, see \cite[Theorem 1.9]{Morel}. Indeed, as a corollary of Gabber's lemma, one obtains a distinguished square for the Nisnevich topology, and such a square becomes a pushout when viewed in the category $\cS pc(k)$ of motivic spaces over a field, so that there exists a motivic equivalence $U / U \setminus U \cap Z \simeq \AA^1_W / \AA^1_W \setminus Z \cap U$. 

\subsection{The equivariant setting}
Equivariant motivic homotopy theory has been studied in \cite{MotivicHomotopyTheoryOfGroupSchemesActions} and in \cite{Hoyois}, based on the introduction of an equivariant version of the Nisnevich topology. 
Moreover, a $C_2$-equivariant version of Gabber's lemma has been established by Bachmann in \cite{Bachmann}, leading to applications in motivic homotopy theory for schemes equipped with a $C_2$-action. In particular, the author proves a $C_2$-equivariant version of Gersten’s injectivity and suggests that extending Gabber’s presentation lemma to more general group actions would provide a foundation for analogous applications in equivariant motivic homotopy theory in a broader setting. The main difference between the $C_2$-equivariant version and the original version of Gabber's lemma is that the equivariant result is Nisnevich-local on $X$ instead of being Zariski-local. More precisely, we refer to \cite{MotivicHomotopyTheoryOfGroupSchemesActions} for the following notion of equivariant Nisnevich neighborhood.

\begin{definition}
Let $G$ be a finite abelian group. Let $X$ be a $G$-scheme over $k$ and let $x \in X$ be a point. A \emph{$G$-equivariant Nisnevich neighborhood of $Gx$} is the datum of a $G$-scheme $Y$ over $k$, a $G$-equivariant étale morphism $\varphi \colon Y \to X$ and a $G$-equivariant morphism $s \colon G \times^{S_x} \Spec(\kappa(x)) \to Y$ such that the triangle
% https://q.uiver.app/#q=WzAsMyxbMCwxLCJHIFxcdGltZXNee1NfeH0gXFxTcGVjKFxca2FwcGEoeCkpIl0sWzEsMCwiWSJdLFsxLDEsIlgiXSxbMSwyXSxbMCwyXSxbMCwxXV0=
\[\begin{tikzcd}
	& Y \\
	{G \times^{S_x} \Spec(\kappa(x))} & X
	\arrow["\varphi", from=1-2, to=2-2]
	\arrow["s", from=2-1, to=1-2]
	\arrow[from=2-1, to=2-2]
\end{tikzcd}\]
commutes, where the horizontal map is given by passing the action morphism on $x$ to the quotient. We will say that $\varphi \colon Y \to X$ is a $G$-equivariant Nisnevich neighborhood of $Gx$, tacitly assuming that a section $s$ exists when it is clear from the context, and the points of $Y$ that are in the image of the section $s$ will be called \emph{distinguished points}. Note that if $x$ is set-theoretically fixed, then there is a unique distinguished point in $Y$ lying over $x$. A $G$-equivariant Nisnevich neighborhood of $Gx$ is said to be \emph{affine} (resp.\ \emph{irreducible}) whenever $Y$ is affine (resp.\ irreducible).
\end{definition}

In order to state the equivariant version of Gabber's presentation lemma, we also introduce the following terminology.

\begin{definition} \label{definition: gabber presentation}
Let $G$ be a finite abelian group. A \emph{$G$-pair over $k$} (or simply a $G$-pair if $k$ is understood) is an ordered pair $(X,Z)$ where $X$ is a separated smooth $k$-scheme of positive dimension endowed with a $G$-action and $Z \subseteq X$ is a closed $G$-invariant subscheme of positive codimension. We say that a $G$-pair $(X,Z)$ admits a \emph{$G$-equivariant Gabber's presentation over $k$} if for any point $z \in Z$ the following data exist:

\begin{Glist}
	\item a $G$-equivariant Nisnevich neighborhood $X' \to X$ of $Gz$;
	\item a smooth $k$-scheme $W$ endowed with a $G$-action;
	\item a one-dimensional $G$-representation $V$ over $k$;
	\item a $G$-equivariant morphism $\varphi = (\psi, \nu) \colon X' \to  W \times_k \AA(V)$ such that 
	\begin{enumerate}
		\renewcommand{\labelenumii}{\textbf{(P\arabic{enumii})}}
		\item $\varphi$ is étale;
		\item $\varphi_{|Z_{X'}}$ is a closed immersion;
		\item the composite $\psi_{|Z_{X'}} \colon Z_{X'} \to W \times_k \AA(V) \to W$ is finite;
		\item $\varphi$ restricts to an isomorphism $\varphi^{-1}(\varphi(Z_{X'})) \xrightarrow{\sim} \varphi(Z_{X'})$.
	\end{enumerate}
\end{Glist}
In the case $G= \left\{e\right\}$ we simply say $(X,Z)$ is a \emph{pair} and that the above data provide a \emph{Gabber's presentation}. If a $G$-pair $(X,Z)$ admits a $G$-equivariant Gabber's presentation, for a point $z \in Z$ we will say that the associated data $(X', W, V, \varphi)$ form a \emph{$G$-equivariant Gabber's presentation with respect to the point $z$}.
\end{definition}

The property of a $G$-pair $(X,Z)$ of having a $G$-equivariant Gabber's presentation depends also on the base field $k$. We will simply say that the $G$-pair admits a $G$-equivariant Gabber's presentation when the base field $k$ is clear from the context.

Suppose that $G= \left\{e\right\}$. Then Theorem \ref{theorem: gabber original}, i.e.\ non-equivariant Gabber's lemma, provides a Gabber's presentation for any pair $(X,Z)$ with $X$ affine and irreducible. Indeed, given a point $z \in Z$, any open neighborhood $U \subset X$ of $z$ is a $G$-equivariant Nisnevich neighborhood of $z$, giving $\Gnum{1}$. The smooth $k$-scheme $W$ of $\Gnum{2}$ is given by the open $W \subset \AA^{d-1}_k$. The trivial one-dimensional $k$-representation gives $\Gnum{3}$. Finally, the existence of a morphism $\varphi$ as in $\Gnum{4}$ having properties $\Pnum{1}, \Pnum{2}, \Pnum{3}$ and $\Pnum{4}$ is precisely the conclusion of Theorem \ref{theorem: gabber original}.

\subsection{Main theorem}

We can rephrase the $C_2$-equivariant version of Gabber's lemma proposed by Bachmann \cite{Bachmann} by saying that every $C_2$-pair $(X,Z)$ over an infinite field of characteristic different from $2$, admits a $C_2$-equivariant Gabber's presentation.
In this paper, we show that the equivariant version of Gabber's lemma proposed by Bachmann holds for a finite abelian group $G$ \emph{if and only if} $G$ is a cyclic $p$-group for some prime $p$.
More precisely, our first result is a $C_{p^r}$-equivariant version of Gabber's lemma, generalizing the $C_2$-equivariant version proposed by Bachmann, over infinite perfect fields $k$ containing a primitive $p^r$th root of unity.

\begin{theorem}[Theorem \ref{theorem: gabber}] \label{theorem: equivariant gabber intro}
Let $k$ be an infinite perfect field. Let $p$ be a prime and let $r \geq 1$ be an integer such that $k$ contains a primitive $p^r$th root of unity. Let $(X,Z)$ be a $C_{p^r}$-pair and let $z \in Z$ be a point. Then there exists a $C_{p^r}$-equivariant Gabber's presentation $(X', W, V, \varphi = (\psi, \nu))$ of $(X,Z)$ with respect to the point $z$ (see Definition \ref{definition: gabber presentation}). Moreover, if $V$ is a nontrivial $C_{p^r}$-representation, then $\psi \colon X' \to W$ admits a $C_{p^r}$-equivariant section.
\end{theorem}

It follows that, under hypotheses of Theorem \ref{theorem: equivariant gabber intro}, any $C_{p^r}$-pair $(X,Z)$ admits a $C_{p^r}$-equivariant Gabber's presentation. Note that Theorem \ref{theorem: equivariant gabber intro} also provides a $C_{p^r}$-equivariant section of the morphism $\psi$ every time that $V$ is a nontrivial one-dimensional $C_{p^r}$-representation, generalizing the result of Bachmann. We believe that the existence of such a $C_{p^r}$-equivariant section will be important for the applications of Theorem \ref{theorem: equivariant gabber intro}, as pointed out in \cite{Bachmann}.

In contrast, if a finite abelian group $G$ decomposes into more than one factor under the structure theorem for finite abelian groups, then we show that the analogous conclusions of Theorem \ref{theorem: equivariant gabber intro} do not hold.

\begin{theorem}[Theorem \ref{theorem: counterexample}] \label{theorem: counterexample intro}
Let $G$ be a finite abelian group. Suppose given primes $p_1$ and $p_2$ such that $G \simeq C_{p_1^{r_1}} \times C_{p_2^{r_2}} \times H$, for some abelian group $H$. Suppose that $k$ contains a primitive $p_i$th root of unity for $i=1,2$. Then there exists a $G$-pair $(X,Z)$ that does not admit a $G$-equivariant Gabber's presentation over $k$.
\end{theorem}

By combining Theorem \ref{theorem: equivariant gabber intro} and Theorem \ref{theorem: counterexample intro} we get the following:

\begin{corollary}
Let $k$ be an infinite perfect field. Let $G$ be a finite abelian group such that $k$ contains a primitive $p^r$th root of unity for each cyclic summand $C_{p^r}$
appearing in the decomposition of $G$. Then every $G$-pair  admits a $G$-equivariant Gabber's presentation over $k$ if and only if $G$ is a cyclic $p$-group.
\end{corollary}

The reason why Theorem \ref{theorem: equivariant gabber intro} does not hold for more general groups is that it is not possible in general to exhibit a finite equivariant morphism from $Z$ to the base $W$, see also Remark \ref{remark: why not possible to produce quasifinite equivariant morphisms} and Example \ref{example: failure quasi-finiteness}.

\subsection{Proof-sketch of the main theorem}

The proof of Theorem \ref{theorem: equivariant gabber intro} goes by ``shrinking" $X$ in various steps to equivariant Nisnevich neighborhoods of $Gz$, each step providing one of the required properties of Definition \ref{definition: gabber presentation}. In particular, we will first provide conclusions over the henselization $W^{h}_{Gz}$ of a smooth $k$-scheme $W$ and then pass to an actual model using a limit-descending argument. We just mention the fact that if $r=1$, then Theorem \ref{theorem: equivariant gabber intro} still holds over an infinite imperfect field $k$, the proof being analogous to Bachmann's $C_2$-equivariant case \cite{Bachmann}; however we will not work in this generality. For coefficients in which the characteristic exponent of $k$ is invertible, passage to the perfect closure is harmless for Voevodsky motives: pullback along $k \subseteq k^{\textnormal{perf}}$ induces an equivalence $\textnormal{DM}(k; R) \simeq \textnormal{DM}(k^{\text{perf}}; R)$; see \cite[Proposition 8.1]{CisinskiDeglise}.

The main difficulty in generalizing Bachmann's result to more general finite abelian groups is to find equivariant morphisms that are étale and whose restriction to the closed subvariety $Z$ gives a finite morphism over the base $W$. In particular, for a morphism $\varphi \colon X \to \AA^n_k$, the property of being étale at some point $x \in X$ requires the study of the tangent space $T_{x}X$. If a group $G$ acts on $X$, and if $x \in X^G$ is fixed by the $G$-action, then $T_{x}X$ is a $G$-representation over the residue field $\kappa(x)$, and it is possible to find a $G$-equivariant morphism $\varphi \colon X \to \AA(T_{x}X)$ which is étale at $x$. However, if the geometric stabilizer $G_{x}$ of $x$ is strictly smaller than $G$, a priori we can only find a $G_{x}$-equivariant morphism $\varphi \colon X \to \AA(T_{x}X)$ which is étale at $x$. In order to actually have a $G$-equivariant morphism, we will require an additional property on the $G_x$-representation $T_{x}X$, and find an étale morphism as in Proposition \ref{proposition: étale morphism as in 8.10 of equivariant cycles}, using \cite[Lemma 8.10]{Equivariantcycles}.

On the other hand, in order to produce equivariant finite morphisms $Z \to W$ we mimic the same strategy as \cite{Bachmann}. Namely, we first produce equivariant quasi-finite morphisms and then produce a finite morphism by shrinking to an open neighborhood, using Lemma \ref{lemma: pass from quasi-finite to finite}. However, note that in order to produce equivariant quasi-finite morphisms it will be essential that the group $G$ is of the form $C_{p^r}$. The result is Proposition \ref{proposition: quasi-finite morphisms construction}, and this is the reason why Theorem \ref{theorem: equivariant gabber intro} fails for finite abelian groups which are not of the form $C_{p^r}$.

We believe that our result can be used to explore applications in equivariant motivic homotopy theory.
We will study applications of Theorem \ref{theorem: equivariant gabber intro} in future work.

\subsection{Outline of the paper}
The paper is organized as follows. In $\S$\ref{section: counterexamples} we provide a counterexample to the equivariant version of Gabber's lemma for finite abelian groups which are not cyclic $p$-groups for some prime $p$. In $\S$\ref{section: representations} we recall some basics on semilinear representations, which naturally arise when working with a group scheme action over a scheme. In $\S$\ref{section: constructing equivariant étale morphisms} and $\S$\ref{section: constructing equivariant quasifinite morphisms} we prove the key fact that under suitable hypotheses, equivariant étale and quasi-finite morphisms exist from a fixed variety to some representation. We proceed in $\S$\ref{section: improving morphisms} by showing that we may improve some properties of an equivariant morphism of schemes by passing to some smaller equivariant Nisnevich neighborhood; these results will be applied throughout the proof of the main result (Theorem \ref{theorem: gabber}), which is an equivariant version of Gabber's lemma for cyclic $p$-groups; indeed we will achieve all the required properties of the morphism by passing step by step to a smaller equivariant Nisnevich neighborhood of the orbit of a fixed point. Finally in $\S$\ref{section: proof of theorem}, we prove the main result Theorem \ref{theorem: gabber}. The reader may first read $\S$\ref{section: counterexamples}, $\S$\ref{section: proof of theorem} and then consult the proof of Theorem \ref{theorem: gabber} to gain an overview of the argument, and refer to the preceding sections for the technical details. 

\subsection{Notation and conventions} Throughout the paper, $k$ denotes a fixed field, which will be assumed to be infinite or perfect when necessary. Every $k$-scheme is assumed to be separated over $k$ and morphisms of schemes will be over $k$ unless otherwise stated. Moreover, $G$ will always denote a finite abelian abstract group, which is viewed as a group scheme over $k$ as $G \coloneq \coprod_{g \in G} \Spec(k)$. The \emph{exponent} of $G$ is the least common multiple of the orders of its elements. If $X$ is a scheme over $k$, a $G$-action on $X$ over $k$ is equivalent to a group homomorphism $G \to \Aut_k (X)$. In other words, every element $g \in G$ defines an automorphism $g \colon X \to X$ of $X$ over $k$. If $x \in X$ is a point, we denote the \emph{set-theoretic stabilizer of $x$} by $S_x \coloneq \left\{ g \in G \mid gx = x \right\}$, the \emph{geometric stabilizer of $x$} by $G_x \coloneq \ker (S_x \to \Aut_k(\kappa(x)))$ and the \emph{set-theoretic orbit of $x$} by $Gx \coloneq \left\{gx \mid g \in G \right\}$. Note that $S_x$ and $G_x$ are subgroups of $G$, and $G_x \subseteq S_x$. Since $G$ is a finite abstract group, $Gx$ is a finite set of points of the underlying topological space of $X$. We view $Gx$ as the set of points of the underlying topological space of the scheme $G \times^{S_x} \Spec(\kappa(x)) = (G \times_k \Spec(\kappa(x)))/S_x$. We denote by $X^G$ the closed subscheme of geometrically $G$-fixed points of $X$.

If $(V, \rho)$ is a $G$-representation over $k$, we will, by abuse of notation, simply refer to $V$ as a $G$-representation whenever $\rho$ is understood or irrelevant. If $(V,\rho)$ is a $G$-representation over $k$, then $\AA(V)$ denotes the affine scheme $\Spec(\Sym(V^{\vee}))$ associated to $V$ with linear $G$-action determined by $\rho$. Note that if $\dim_k V=n$, up to fixing a basis of $V$, there is an isomorphism $\AA(V) \xrightarrow{\sim} \AA^n_k$. We also say that $\AA(V)$ (or $\AA^n_k$) is a $G$-representation over $k$, meaning that it is the affine space associated to a $G$-representation $(V, \rho)$ over $k$ and endowed with the linear $G$-action determined by $\rho$.

Let $X$ be a scheme, $Z \subseteq X$ a subscheme and $X' \to X$ a morphism. We write $Z_{X'} \coloneq X' \times_X Z$ for the base change of $Z$ along $X' \to X$.

\section{Failure outside the cyclic prime-power case} \label{section: counterexamples}

In this section we prove the failure of the equivariant version of Gabber's lemma proposed by Bachmann for finite abelian groups which decompose into more than one factor. 

\begin{theorem} \label{theorem: counterexample}
Let $G$ be a finite abelian group. Suppose given primes $p_1$ and $p_2$ such that $G \simeq C_{p_1^{r_1}} \times C_{p_2^{r_2}} \times H$, for some abelian group $H$. Suppose that $k$ contains a primitive $p_i$th root of unity for $i=1,2$. Then there exists a $G$-pair $(X,Z)$ that does not admit a $G$-equivariant Gabber's presentation over $k$.
\end{theorem}	

Theorem \ref{theorem: counterexample} will be a direct consequence of Proposition \ref{proposition: counterexample}. In order to prove the latter, we will need the following computational lemmas.

\begin{lemma} \label{proposition: formal power serie computation}
Let $m \geq 2$ be an integer such that $\textnormal{char}(k) \nmid m$ and let $g \colon k[[t]] \to k[[t]]$ be a $k$-algebra automorphism of the ring $k[[t]]$ of formal power series in the variable $t$ over $k$, satisfying $g^m = \textnormal{id}$. Write $g(t)= \sum_{n \geq 0} a_n t^n$ for some $a_n \in k$.
Then:
\begin{enumerate}
	\item $a_1 \in k$ is an $m$th root of unity;
	\item if $a_1=1$, then $a_n =0 $ for all $n \geq 2$.
\end{enumerate}
\end{lemma}
\begin{proof}
First note that $a_0=0$ because $g$ is a local morphism. Hence we may write $g(t) = a_1 t + a_2 t^2 + O(t^3)$.
Then we compute
\begin{align*}
	g^2(t) = g(g(t)) &= a_1 g(t) +a_2 g(t)^2 + O(t^3) = \\
	&=a_1(a_1t + a_2t^2+ O(t^3))+ a_2(a_1 t+O(t^2))(a_1t +O(t^2)) = \\
	&= a_1^2 t+a_2(a_1+a_1^2)t^2+ O(t^3).
\end{align*}
More generally we have that
\begin{align*}
	g^l(t) = a_1^lt+ a_2\bigg (\sum_{j=0}^{l-1} a_1^{l-1+j}\bigg )t^2 + O(t^3).
\end{align*}
Since $g^m(t)=t$, we see that $a_1^m = 1$ and thus $a_1 \in k$ is a $m$th root of unity, proving (1).

We prove (2) by induction on the index $n \geq 2$. In order to show that $a_2 = 0$ we observe that the coefficient of $t^2$ in $g^m(t)$ is equal to
\[
a_2\bigg (\sum_{j=0}^{m-1}a_1^{m-1+j} \bigg) = a_2 \bigg(\sum_{j=0}^{m-1}1 \bigg) = ma_2.
\]
On the other hand, $g^m(t)=t$, hence $ma_2=0$. Since $\textnormal{char}(k) \nmid m$, we see that $a_2 = 0$.
Suppose now that $a_1=1$ and $a_i=0$ for all $1<i \leq n$. We need to show that $a_{n+1}=0$. Write $g(t) = t + a_{n+1}t^{n+1} + O(t^{n+2})$.	As before, we compute $g^m(t)= t + ma_{n+1}t^{n+1} + O(t^{n+1})$. By comparing the coefficients again we find $ma_{n+1}=0$ and thus $a_{n+1}=0$.
\end{proof}

\begin{lemma} \label{lemma: first coefficient map local rings}
Let $G$ be an abstract group acting on the local rings $k[[t]]$ and $k[[x]]$ by $k$-algebra automorphisms. Let $\alpha \colon k[[t]] \to k[[x]]$ be a local $G$-equivariant $k$-algebra homomorphism such that $\alpha(t) \neq 0$. Let $g \in G$ be an element which acts trivially on $k[[x]]$ and write $g(t) = \sum_{n \geq 1} a_n t^n$ for some $a_n \in k$. Then $a_1 =1$.
\end{lemma}
\begin{proof}
Let $\alpha(t)  =\sum_{m \geq 0} \lambda_m x^m$ be the image of $t$ in $k[[x]]$ for some $\lambda_m \in k$. Since $\alpha$ is a local morphism we have $\lambda_0 = 0$. Let $m' \geq 1$ be the least integer $m$ such that $\lambda_{m'} \neq 0$. The equivariance of the map $\alpha$ gives $\alpha(g(t)) = g(\alpha(t)) = \alpha(t)$ where the last equality follows from the fact that $g$ acts trivially on $k[[x]]$. On one hand we have $\alpha(t)=\sum_{m \geq m'} \lambda_m x^m$; on the other hand we have
\[
\alpha(g(t)) = \alpha \bigg (\sum_{n \geq 1} a_n t^n \bigg) = a_1 \alpha(t) + \sum_{n \geq 2} a_n \alpha(t)^n = a_1 \bigg (\sum_{m \geq m'} \lambda_m x^m \bigg ) + \sum_{n \geq 2} a_n  \bigg (\sum_{m \geq m'} \lambda_{m} x^m \bigg)^n.
\]
Comparing the coefficient of $x^{m'}$ in $\alpha(g(t))$ and in $\alpha(t)$ we find $\lambda_{m'} = \lambda_{m'} a_1$. Since $\lambda_{m'} \neq 0$, it follows that $a_1 =1$.
\end{proof}

\begin{proposition} \label{proposition: counterexample}
Let $p_1,p_2$ be two not necessarily distinct primes such that $k$ contains a primitive $p_i$th root of unity $\zeta_{p_i}$ for $i=1,2$. Let $r_1,r_2 \geq 1$ be two integers, set $G \coloneq C_{p_1^{r_1}} \times C_{p_2^{r_2}}$ and let $\sigma_i$ be a generator of $C_{p_i^{r_i}}$ for $i=1,2$. Let $X = \AA^2_k = \Spec(k[x_1,x_2])$ with $G$-action given by 
\[
\sigma_i(x_j) = 
\begin{cases}
	\zeta_{p_i}x_i & \text{ if } i = j;  \\
	x_j & \text{ if } i \neq j.
\end{cases}
\]
Let $Z=V(x_1x_2) \subseteq X$ be the closed invariant subscheme given by the union of the coordinate axes and let $z=(0,0)$ be the origin. Then there does not exist a $G$-equivariant Gabber's presentation of $(X,Z)$ with respect to $z$.
\end{proposition}

\begin{proof}
Suppose by contradiction that there exists a $G$-equivariant Gabber's presentation $(X',W, V, \varphi)$ of $(X,Z)$ with respect to $z$. For $\{i,j\} = \{1,2\}$, let $Z_i \coloneq V(x_j)$; thus $Z_i = Z^{\langle \sigma_j \rangle}$. These are closed $G$-invariant subschemes of $Z$, and the completed local ring of $Z_i$ at $z$ is $k[[x_i]]$.

Let $z' \in X'$ be the distinguished point in $X'$ over $z \in X$.
Since $X$ is an irreducible scheme of dimension two and $z$ is fixed by the $G$-action, after replacing $X'$ with the irreducible component containing $z'$, we may assume that $X'$ is irreducible and that $\dim(X')=2$. After replacing $W$ with its irreducible component containing the image of the composite morphism $\psi= \textnormal{pr}_W \circ \varphi$, we may also assume that $W$ is irreducible, so that $\AA^1_W$ is also irreducible. Since $\varphi \colon X' \to \AA^1_W$ is étale, we see that necessarily $\dim(\AA^1_W)=2$ and $\dim(W)=1$.

Let $Z_i'$ be the base change of $Z_i$ to $X'$. Consider the following diagram for $i=1,2$:

% https://q.uiver.app/#q=WzAsOCxbMSwxLCJaIl0sWzEsMCwiWl97WCd9Il0sWzMsMCwiWCciXSxbMiwxLCJYIl0sWzQsMSwiXFxBQV4xX2sgXFx0aW1lc19rIFciXSxbNCwyLCJXIl0sWzAsMSwiWl9pIl0sWzAsMCwiWl9pJyJdLFsxLDIsIiIsMCx7InN0eWxlIjp7InRhaWwiOnsibmFtZSI6Imhvb2siLCJzaWRlIjoidG9wIn19fV0sWzIsMywiw6l0Il0sWzEsMF0sWzAsMywiIiwwLHsic3R5bGUiOnsidGFpbCI6eyJuYW1lIjoiaG9vayIsInNpZGUiOiJ0b3AifX19XSxbMiw0LCLDqXQiLDJdLFs0LDUsIlxcdGV4dG5vcm1hbHtwcn1fMiJdLFs2LDAsIiIsMCx7InN0eWxlIjp7InRhaWwiOnsibmFtZSI6Imhvb2siLCJzaWRlIjoidG9wIn19fV0sWzcsMSwiIiwwLHsic3R5bGUiOnsidGFpbCI6eyJuYW1lIjoiaG9vayIsInNpZGUiOiJ0b3AifX19XSxbNyw2XV0=
\begin{equation}\label{equation: diagram counterexample}
	\begin{tikzcd} 
		{Z_i'} & {Z_{X'}} && {X'} & \\
		{Z_i} & Z & X && {\AA^1_k \times_k W} \\
		&&&& W.
		\arrow[hook, from=1-1, to=1-2]
		\arrow[from=1-1, to=2-1]
		\arrow[hook, from=1-2, to=1-4]
		\arrow[from=1-2, to=2-2]
		\arrow["\textnormal{ét}", from=1-4, to=2-3]
		\arrow["\textnormal{ét}"', from=1-4, to=2-5]
		\arrow[hook, from=2-1, to=2-2]
		\arrow[hook, from=2-2, to=2-3]
		\arrow["{\textnormal{pr}_2}", from=2-5, to=3-5]
	\end{tikzcd}
\end{equation}

Since $Z_i'$ are closed subschemes of $Z_{X'}$ and $Z_{X'}$ is finite over $W$ by property $\Pnum{3}$, the composite morphisms $Z_i' \to W$ are still finite for $i=1,2$. Let $\varphi(z')=(\lambda, w)$ be the image of $z'$ in $\AA^1_W$. Note that $z' \in (X')^{G}$ is a fixed point, because $X' \to X$ is an equivariant Nisnevich neighborhood of $z \in X^{G}$. Then $w \in W^{G}$ is also a fixed point, being the image of $z'$ under the $G$-equivariant morphism $\textnormal{pr}_2 \circ \varphi$.

The point $z'$ is $k$-rational, and the closed immersion $Z_{X'} \to \AA^1_W$ therefore makes $\varphi(z')$ and its projection $w$ $k$-rational. Since $W$ is a smooth irreducible $k$-scheme of dimension one, the Cohen structure theorem (see, for example, \cite[\href{https://stacks.math.columbia.edu/tag/032A}{Tag 032A}]{stacks-project}) gives $\comp_{W,w} \simeq k[[t]]$. Moreover diagram \ref{equation: diagram counterexample} shows that we have isomorphisms
\[
\comp_{\AA^1_W, \varphi(z')} \xrightarrow{\sim} \comp_{X',z'} \xleftarrow{\sim} \comp_{X,z} \simeq k[[x_1, x_2]]
\]
because the maps $\varphi$ and $X' \to X$ are étale and we use \cite[Corollary 18.66]{GortzWedhorn2}. It follows that diagram \ref{equation: diagram counterexample} induces the following diagram on completed local rings for $i=1,2$:

% https://q.uiver.app/#q=WzAsOCxbMSwxLCJcXGZyYWN7a1tbeF8xLHhfMl1dfXsoeF8xeF8yKX0iXSxbMSwwLCJcXHdpZGVoYXR7XFxjT31fe1pfe1gnfSwgen0iXSxbMywwLCJcXHdpZGVoYXR7XFxjT31fe1gnLCB6J30iXSxbMiwxLCJrW1t4XzEseF8yXV0iXSxbNCwxLCJcXHdpZGVoYXR7XFxjT31fe1xcQUFeMV9XLCBcXHZhcnBoaSh6Jyl9Il0sWzQsMiwia1tbdF1dIl0sWzAsMSwia1tbeF9pXV0iXSxbMCwwLCJcXHdpZGVoYXR7XFxjT31fe1pfaScsIHp9Il0sWzIsMV0sWzMsMiwiXFxzaW0iLDJdLFswLDEsIlxcc2ltIl0sWzMsMF0sWzQsMiwiXFxzaW0iXSxbNSw0XSxbMCw2XSxbNiw3LCJhX2kiXSxbMSw3XV0=
\begin{equation} \label{equation: diagram on completions}
	\begin{tikzcd}
		{\widehat{\cO}_{Z_i', z'}} & {\widehat{\cO}_{Z_{X'}, z'}} && {\widehat{\cO}_{X', z'}} & \\
		{k[[x_i]]} & {\frac{k[[x_1,x_2]]}{(x_1x_2)}} & {k[[x_1,x_2]]} && {\widehat{\cO}_{\AA^1_W, \varphi(z')}} \\
		&&&& {k[[t]].}
		\arrow[from=1-2, to=1-1]
		\arrow[from=1-4, to=1-2]
		\arrow["{a_i}", from=2-1, to=1-1]
		\arrow["\sim", from=2-2, to=1-2]
		\arrow[from=2-2, to=2-1]
		\arrow["\sim"', from=2-3, to=1-4]
		\arrow[from=2-3, to=2-2]
		\arrow["\sim", from=2-5, to=1-4]
		\arrow[from=3-5, to=2-5]
	\end{tikzcd}
\end{equation}

By the functoriality of completion, all of the complete local rings appearing in diagram \ref{equation: diagram on completions} carry a $G$-action, because the points $z,z'$ and $w$ are geometrically fixed points. In particular, $\sigma_i$ acts on $k[[x_1,x_2]]$ as $\sigma_i(x_i)=\zeta_{p_i}x_i$ and $\sigma_i(x_j)=x_j$ for $i,j =1,2$. Moreover, again by the functoriality of the completion, all the morphisms in diagram \ref{equation: diagram on completions} are $G$-equivariant. 

Since the composite morphism $Z_i' \to W$ is finite for $i=1,2$, then it induces a finite morphism of completed local rings $k[[t]] \to \comp_{Z_i',z'}$ for $i=1,2$ by \cite[\href{https://stacks.math.columbia.edu/tag/07N9}{Tag 07N9}]{stacks-project}.

Moreover, the map $a_i \colon k[[x_i]] \to \widehat{\cO}_{Z_i',z'}$ is an isomorphism for $i =1,2$, because $Z_i' \to Z_i$ is étale and induces an isomorphism on residue fields at the distinguished point. Composing the map $k[[t]] \to \widehat{\cO}_{Z_i',z'}$ with the inverse $a_i^{-1}$ of $a_i$, we get finite $G$-equivariant morphisms of rings $\alpha_i \colon k[[t]] \rightarrow  k[[x_i]]$ for $i=1,2$. Note that $\alpha_i(t) \neq 0$ for $i=1,2$, because otherwise $k[[x_i]]$ would be a finite $k$-algebra. For $i=1,2$, let $m_i$ be the order of the automorphism of $k[[t]]$ induced by $\sigma_i$. Then $m_i$ divides $p_i^{r_i}$, for $i=1,2$. If $m_i =1$, the action of $\sigma_i$ on $k[[t]]$ is already trivial. Otherwise, $\textnormal{char}(k) \nmid m_i$, and we may write
\[
\sigma_i(t) = a_{i,1} t + \sum_{q \geq 2} a_{i,q}t^q.
\]
By Lemma \ref{proposition: formal power serie computation}, $a_{i,1}$ is an $m_i$th root of unity. Applying Lemma \ref{lemma: first coefficient map local rings} to $\alpha_j \colon k[[t]] \to k[[x_j]]$ and $\sigma_i$, where $i \neq j$, gives $a_{i,1} = 1$, because $\sigma_i$ acts trivially on $k[[x_j]]$. Lemma \ref{proposition: formal power serie computation} then implies that $a_{i,q} = 0$ for every $q \geq 2$. Thus $\sigma_i(t) = t$, and $\sigma_i$ acts trivially on $k[[t]]$. It follows that the $G$-action on $\comp_{W,w}$ is trivial.

The completion morphism $\cO_{W,w} \to \comp_{W,w}$ is injective by the Krull intersection theorem, because the ring $\cO_{W,w}$ is noetherian. It follows that the $G$-action on $\cO_{W,w}$ is also trivial, and in particular the $G$-action on $T_wW$ is trivial. This yields a contradiction as follows. The diagram 
% https://q.uiver.app/#q=WzAsMyxbMSwwLCJYJyJdLFswLDEsIlgiXSxbMiwxLCJcXEFBXjFfayBcXHRpbWVzX2sgVyJdLFswLDEsIsOpdCJdLFswLDIsIsOpdCIsMl1d
\[\begin{tikzcd}
	& {X'} & \\
	X && {\AA^1_k \times_k W}
	\arrow["\text{ét}", from=1-2, to=2-1]
	\arrow["\text{ét}"', from=1-2, to=2-3]
\end{tikzcd}\]

gives $G$-equivariant isomorphisms on tangent spaces
% https://q.uiver.app/#q=WzAsMyxbMSwwLCJUX3t4J31YIl0sWzAsMSwiVF97eH1YIl0sWzIsMSwiVF97XFx2YXJwaGkodyl9XFxBQV4xX1ciXSxbMCwxLCJcXHNpbSIsMl0sWzAsMiwiXFxzaW0iXV0=
\[\begin{tikzcd}
	& {T_{z'}X'} & \\
	{T_{z}X} && {T_{\varphi(z')}\AA^1_W \simeq T_{\lambda}\AA^1_k \oplus T_{w}W}
	\arrow["\sim"', from=1-2, to=2-1]
	\arrow["\sim", from=1-2, to=2-3]
\end{tikzcd}\]
where the isomorphism in the lower-right corner is \cite[Proposition 6.9]{GortzWedhorn1}. Then we have a contradiction because $(T_zX)^{G} = 0$, while $T_wW$ is trivial.
\end{proof}

We are now ready to prove Theorem \ref{theorem: counterexample}.

\begin{proof}[Proof of Theorem \ref{theorem: counterexample}]
Write $G = C_{p_1^{r_1}} \times C_{p_2^{r_2}} \times H$ for some primes $p_1, p_2$ and some finite abelian group $H$. We let $(X,Z)$ be the $G$-pair defined in Proposition \ref{proposition: counterexample} by letting $H$ act trivially on $X$. Let $G' = C_{p_1^{r_1}} \times C_{p_2^{r_2}} \subseteq G$. If a $G$-equivariant Gabber's presentation of $(X,Z)$ with respect to the origin $z \in X$ existed, then it would also be a $G'$-equivariant Gabber's presentation of $(X,Z)$ with respect to $z$. However, Proposition \ref{proposition: counterexample} shows that such data cannot exist.
\end{proof}

\section{Preliminaries on semilinear representations} \label{section: representations}

In this subsection we recall some results about semilinear $G$-representations. These naturally arise in the proof of Theorem \ref{theorem: gabber}, because if $X$ is a scheme with an action of a finite group $G$ and $x \in X$ is a point, then the tangent space $T_{x}X$ is a semilinear $S_{x}$-representation over $\kappa(x)$.

\vspace{0.2cm}
We begin by recalling some terminology from the theory of linear representations. Let $G$ be a group. We denote by $\textnormal{Rep}_k(G)$ the category of $G$-representations over $k$. An object of $\textnormal{Rep}_k(G)$ is a pair $(V, \rho)$ where $V$ is a finite-dimensional $k$-vector space and $\rho \colon G \to \Aut_k(V)$ is a group homomorphism; morphisms in $\textnormal{Rep}_k(G)$ are morphisms of $G$-representations. We will also say, by abuse of notation, that $V$ is a $G$-representation, tacitly assuming that a group homomorphism $\rho \colon G \to \Aut_k(V)$ exists, when $\rho$ is understood or irrelevant. For example, we say that $V$ is a trivial $G$-representation to mean that $\rho$ is trivial. Let $(V, \rho)$ be a $G$-representation over $k$. A $k$-vector subspace $W \subseteq V$ is a \emph{subrepresentation} if $\rho(g)_{|W}(W) \subseteq W$ for all $g \in G$. A $G$-representation $(V,\rho)$ is said to be \emph{irreducible} if it is nonzero and has no proper nontrivial subrepresentations. Let $H \subseteq G$ be a subgroup. Recall that there is a \emph{restriction functor} $\textnormal{Res}^G_H(-) \colon \textnormal{Rep}_k(G) \to \textnormal{Rep}_k(H)$ and an \emph{induction functor} $\textnormal{Ind}^H_G(-) \colon \textnormal{Rep}_k(H) \to \textnormal{Rep}_k(G)$ going in the opposite direction, which form a pair of adjoint functors $\textnormal{Ind}^H_G \dashv \textnormal{Res}^G_H$. We will need the following result.

\begin{lemma} \label{lemma: lemma on linear representations}
Let $G$ be a finite abelian group such that $k$ contains a primitive $n$th root of unity, where $n$ is the exponent of $G$. Then:
\begin{enumerate}
	\item all irreducible $G$-representations over $k$ are one-dimensional;
	\item the group $G$ is noncanonically isomorphic to its character group $\widehat{G}$ over $k$;
	\item for any subgroup $H \subseteq G$, the restriction functor $\textnormal{Res}^G_H(-) \colon \textnormal{Rep}_k(G) \to \textnormal{Rep}_k(H)$ is essentially surjective.
\end{enumerate}
\end{lemma}

\begin{proof}
(1) is \cite[Lemma 5.2]{Equivariantcycles}. Write $G = C_{n_1} \times \ldots \times C_{n_k}$. Since $k$ contains a primitive $n$th root of unity, we have $\widehat{G} = \Hom(G, k^{\times}) \simeq \prod_{i=1}^k \Hom(C_{n_i}, k^{\times}) \simeq G$, giving (2). By Maschke's theorem, the categories $\textnormal{Rep}_k(G)$ and $\textnormal{Rep}_k(H)$ are semisimple. Then, in order to show (3), it is sufficient to show that any irreducible $H$-representation over $k$ is isomorphic to $\textnormal{Res}^G_H(W)$ for some $G$-representation $W$ over $k$. Let $V$ be an irreducible $H$-representation over $k$. Let $\textnormal{Ind}^H_G(V)$ be the induced $G$-representation and let $W \subseteq \textnormal{Ind}^H_G(V)$ be any irreducible $G$-subrepresentation. By (1), it follows that $\dim_k(W)=1$. Then the projection $\textnormal{Ind}^H_G(V) \to W$ is a morphism of $G$-representations that corresponds to a nonzero morphism of $H$-representations $\varphi \colon V \to \textnormal{Res}^G_H(W)$ under the adjunction $\textnormal{Ind}^H_G \dashv \textnormal{Res}^G_H$. By Schur's lemma, $\varphi$ is necessarily an isomorphism, because both $V$ and $\textnormal{Res}^G_H(W)$ are irreducible $H$-representations.
\end{proof}

We now recall some basics about semilinear representations.

\begin{definition}
Let $G$ be a group, let $K$ be a field and let $\sigma \colon G \to \Aut(K)$ be a group homomorphism. A \emph{$\sigma$-semilinear $G$-representation over $K$} is a pair $(V, \rho)$ where $V$ is a finite-dimensional $K$-vector space and $\rho \colon G \to \Aut_{\ZZ}(V)$ is a group homomorphism such that 
\[
\rho(g) (\lambda v) = \sigma(g)(\lambda) \rho(g)(v) \qquad \forall g \in G, \lambda \in K, v \in V.
\]
A \emph{morphism of semilinear $G$-representations over $K$} from $(V, \rho)$ to $(W, \tau)$ is a $K$-linear morphism $\phi \colon V \to W$ such that $\tau (g)\circ \phi = \phi \circ \rho(g)$ for all $g \in G$. We denote the resulting category of $\sigma$-semilinear $G$-representations over $K$ by $\textnormal{Rep}_K^{\sigma}(G)$.
\end{definition}

To simplify the notation, we will also write $g(\lambda)$ for $\sigma(g)(\lambda)$ and $g(v)$ for $\rho(g)(v)$.

\begin{remark}
Note that if $(V, \rho)$ is a semilinear $G$-representation over $K$, the group homomorphism $\rho$ maps to $\Aut_{\ZZ}(V)$ which is the group of automorphisms of the underlying abelian group of $V$, and not $\Aut_K(V)$, otherwise each morphism $\rho(g)$ would be $K$-linear.
\end{remark}

\begin{example}
If $\sigma \colon G \to \Aut(K)$ is the trivial group homomorphism, then the category $\textnormal{Rep}_K^{\sigma}(G)$ is simply $\textnormal{Rep}_K(G)$.
\end{example}

\begin{definition}
Let $(V, \rho)$ be a $\sigma$-semilinear $G$-representation over $K$. A \emph{$\sigma$-semilinear $G$-subrepresentation over $K$} is a $K$-vector subspace $W \subseteq V$ such that $\rho(g)_{|W}(W) \subseteq W$ for all $g \in G$. We say that $(V, \rho)$ is \emph{irreducible} if it is nonzero and the only semilinear $G$-subrepresentations are $0$ and $V$.
\end{definition}

The category $\textnormal{Rep}^{\sigma}_{K}(G)$ is \emph{semisimple} if every $\sigma$-semilinear $G$-representation is a direct sum of irreducible $\sigma$-semilinear $G$-representations.

\begin{lemma} \label{lemma: semisimple}
Let $G$ be a finite group, let $K$ be a field and let $\sigma \colon G \to \Aut(K)$ be a group homomorphism. Then the category $\textnormal{Rep}^{\sigma}_{K}(G)$ is semisimple if and only if $|\ker(\sigma)|$ is invertible in $K$.
\end{lemma}

\begin{proof}
See \cite[Lemma 1.3]{Kunzer}.
\end{proof}

We can also define semilinear $G$-representations as the full subcategory of the category of left modules over a suitable ring. This is needed in the proof of Proposition \ref{proposition: étale morphism as in 8.10 of equivariant cycles}, because we will work with $K$-vector spaces carrying a semilinear $G$-action but which are not finite-dimensional $K$-vector spaces.

\begin{definition}
Let $G$ be a finite abelian group and let $K$ be a field. The \emph{group algebra} $K[G]$ is the associative $K$-algebra whose underlying $K$-vector space has basis $\left\{e_g\right\}_{g \in G}$; multiplication is defined by $\lambda_1 e_{g_1} \cdot \lambda_2 e_{g_2} = \lambda_1 \lambda_2 e_{g_1g_2}$ for all $\lambda_1, \lambda_2 \in K$, $g_1, g_2 \in G$ and extended $K$-bilinearly. 

Let $\sigma \colon G \to \Aut(K)$ be a group homomorphism. The \emph{skew group ring} $K^{\sigma}[G]$ is the associative ring whose underlying left $K$-vector space is the same as $K[G]$, but multiplication is defined by $\lambda_1 e_{g_1} \cdot \lambda_2 e_{g_2} = \lambda_1 \sigma(g_1)(\lambda_2) e_{g_1g_2}$ for all $\lambda_1, \lambda_2 \in K$, $g_1, g_2 \in G$ and extended additively. Unless $\sigma$ is trivial, the natural copy of $K$ in $K^{\sigma}[G]$ is not central, so $K^{\sigma}[G]$ is generally not a $K$-algebra in the usual sense.
\end{definition}

If $(V, \rho) \in \textnormal{Rep}^{\sigma}_K(G)$, we can associate to it a left $K^{\sigma}[G]$-module whose underlying abelian group is $V$ and the ring homomorphism $K^{\sigma}[G] \to \Aut_{\ZZ}(V)$ defining the $K^{\sigma}[G]$-module structure sends an element $\lambda e_{g} \in K^{\sigma}[G]$ to the endomorphism of the underlying abelian group of $V$ defined by $\lambda e_{g}(v) \coloneq \lambda \rho(g) (v)$ for all $g \in G$ and $v \in V$. This association identifies the category $\textnormal{Rep}^{\sigma}_K(G)$ with the full subcategory of left $K^{\sigma}[G]$-modules whose underlying $K$-vector space has finite dimension (see, for example, \cite[Proposition 2.10]{taylor}).

Suppose now that $G$ is a group and $\sigma \colon G \to \Aut_k(L)$ is a group homomorphism with target the group of field automorphisms of $L$ fixing $k$. Then we can define a functor $- \otimes_k L \colon \textnormal{Rep}_k(G) \to \textnormal{Rep}_L^{\sigma}(G)$ which sends an object $(V, \rho) \in \textnormal{Rep}_k(G)$ to $(V \otimes_k L, \psi)$, where $\psi \colon G \to \Aut_{\ZZ}(V \otimes_k L)$ is defined by $\psi(g)(v \otimes \lambda) = \rho(g)(v) \otimes \sigma(g)(\lambda)$ for all $g \in G$, $v \in V$ and $\lambda \in L$. To simplify notation, we write $g(v \otimes \lambda) = g(v) \otimes g(\lambda)$. The functor $- \otimes_k L$ is neither fully faithful nor essentially surjective in general. However, under suitable hypotheses on the group $G$ and the ground field $k$ we have the following result.

\begin{proposition} \label{lemma: essentially surjectivity}
Let $L/k$ be a finite field extension. Let $G$ be a finite abelian group  such that $k$ contains a primitive $n$th root of unity, where $n$ is the exponent of $G$. Let $\sigma \colon G \to \Aut_k(L)$ be a group homomorphism. Then the functor $ - \otimes_k L \colon \textnormal{Rep}_k(G) \to \textnormal{Rep}^{\sigma}_L(G)$ is essentially surjective and the number of distinct irreducible semilinear representations up to isomorphism in $\textnormal{Rep}^{\sigma}_L(G)$ equals $|\ker{\sigma}|$.
\end{proposition}
\begin{proof}
By Lemma \ref{lemma: semisimple}, the category $\textnormal{Rep}^{\sigma}_L(G)$ is semisimple. Thus, in order to show that the functor $- \otimes_k L$ is essentially surjective,  it is sufficient to show that every irreducible $V \in \textnormal{Rep}^{\sigma}_L(G)$ is isomorphic to $W \otimes_k L$ for some $W \in \textnormal{Rep}_k(G)$. We will show that there exists a complete set of irreducible semilinear $G$-representations over $L$ which are of the form $W \otimes_k L$ for some $W \in \textnormal{Rep}_k(G)$. Set $K \coloneq L^G$, so that $L/K$ is Galois and $G \to \textnormal{Gal}(L/K)$ is surjective. Let $I \coloneq \ker ( G \to \Aut_k(L)) = \ker ( G \to \textnormal{Gal}(L/K))$ and set $d \coloneq |I|$. Since $G \to \textnormal{Gal}(L/K)$ is surjective, we have $G/I \simeq \textnormal{Gal}(L/K)$. By Lemma \ref{lemma: lemma on linear representations}.(1),(2), there exists a set $\{\tilde{V}_1, \ldots, \tilde{V}_d\}$, with $\tilde{V}_i \in \textnormal{Rep}_k(I)$, which form a complete set of irreducible $I$-representations over $k$. By Lemma \ref{lemma: lemma on linear representations}.(3), for all $i=1, \ldots, d$ there exists some $V_i \in \textnormal{Rep}_k(G)$ such that $\textnormal{Res}^G_I(V_i) \simeq \tilde{V}_i$. Note that $\dim_k(V_i) = 1$ necessarily. Set $V_i' \coloneq V_i \otimes_k L$ for all $i=1, \ldots, d$. We claim that $\left\{V'_1, \ldots, V'_d\right\}$ form a complete set of irreducible semilinear representations in $\textnormal{Rep}^{\sigma}_L(G)$. First, note that each $V_i' \in \textnormal{Rep}^{\sigma}_L(G)$ is irreducible because $\dim_L(V_i')= 1$. Suppose that $V_i' \simeq V_j'$ are isomorphic in $\textnormal{Rep}^{\sigma}_L(G)$. Then by \cite[Theorem A]{taylor} we have that $\textnormal{Res}^G_I(V_i') \simeq \textnormal{Res}^G_I(V_j')$ in $\textnormal{Rep}_L(I)$. In particular, $\textnormal{Res}^G_I(V_i) \otimes_k L \simeq \textnormal{Res}^G_I(V_j) \otimes_k L$ and thus, since $\dim_k(V_i) = \dim_k(V_j) = 1$, also $\textnormal{Res}^G_I(V_i) \simeq \textnormal{Res}^G_I(V_j)$ in $\textnormal{Rep}_k(I)$ which gives a contradiction, unless $i=j$. It follows that the set $\left\{V'_1, \ldots, V'_d\right\}$ is made of $d$ distinct irreducible semilinear representations in $\text{Rep}^{\sigma}_L(G)$. To show that they are the only possible irreducible semilinear representations up to isomorphism, we use the following dimensional argument. Set $D_i \coloneq \textnormal{End}_{L^{\sigma}[G]}(V_i')$. Then each $D_i$ is a division ring and we claim that $D_i \simeq K$ for all $1 \leq i \leq d$. Indeed, suppose that $\psi\colon V_i' \to V_i'$ is a $L^{\sigma}[G]$-module endomorphism of $V_i'$. In particular, $\psi \in \textnormal{End}_{L}(V_i')$. Since $\dim_L(V_i')=1$, we may regard $\psi$ as multiplication by some scalar $\lambda \in L$. Moreover, since $\psi$ is a $L^{\sigma}[G]$-module endomorphism, then $\psi(e_g(v))= e_g(\psi(v))$ for all $g \in G$ and $v \in V_i'$. The left hand side of the latter is $\psi(e_g(v)) = \lambda e_g(v)$; while the right hand side is $e_g(\psi(v))= e_g(\lambda v)= e_g(\lambda) e_g(v)$. It follows that $\lambda e_g(v)= e_g(\lambda) e_g(v)$ for all $v \in V'_i$, $g \in G$. Since $e_g$ is an automorphism of $V_i'$, for any $v \neq 0$ we have $e_g(v) \neq 0$. Then $\lambda = e_g(\lambda)$ necessarily. Since this holds for all $g \in G$, it follows that $\lambda \in L^G = K$, and this determines an isomorphism $D_i \simeq K$ as desired. Since the $K$-algebra $L^{\sigma}[G]$ is semisimple, the Wedderburn isomorphism gives $L^{\sigma}[G] \simeq \bigoplus M_{n_i}(\textnormal{End}_{L^{\sigma}[G]}(S_i))$ where the $S_i$'s are all the distinct irreducible left $L^{\sigma}[G]$-modules appearing with multiplicity $n_i$ in the decomposition of $L^{\sigma}[G]$ as a left module over itself. In particular, for $S_i = V_i'$ we have $n_i = \dim_{D_i}(V_i') = \dim_K (V_i') = [L : K]$. Thus, $M_{n_i}(\textnormal{End}_{L^{\sigma}[G]}(V_i')) \simeq M_{n_i}(K)$ has dimension $n_i^2 = [L : K]^2$ as a $K$-vector space. Since the $V_i'$'s are $d$ distinct irreducible $L^{\sigma}[G]$-modules, they contribute to a total of $d \cdot [L: K]^2$ to the dimension of $L^{\sigma}[G]$ as a $K$-vector space. On the other hand, $\dim_K(L^{\sigma}[G]) = \dim_L(L^{\sigma}[G]) \cdot [L : K] = |G| \cdot [L :K] = d \cdot [L : K]^2$. Then we see that there is no space for other irreducible semilinear representations in $\textnormal{Rep}^{\sigma}_L(G)$.
\end{proof}

To conclude this section, we recall the definition of \emph{induced module} in the semilinear setting and then we prove that it is isomorphic to a direct sum of modules indexed by a set of left coset representatives. This description will be used in the proof of Proposition \ref{proposition: étale morphism as in 8.10 of equivariant cycles}.

\begin{definition} \label{definition: induced semilinear}
Let $K$ be a field. Let $G$ be a group and $H \subseteq G$ be a subgroup. Let $\sigma \colon G \to \Aut(K)$ be a group homomorphism. Then $\sigma_{|H} \colon H \to \Aut(K)$ is again a group homomorphism. We again denote $\sigma_{|H}$ by $\sigma$. Note that $K^{\sigma}[G]$ is a right $K^{\sigma}[H]$-module, whose right-module structure is simply given by the inclusion $K^{\sigma}[H] \to K^{\sigma}[G]$ and the multiplication in $K^{\sigma}[G]$. Let $M$ be a left $K^{\sigma}[H]$-module. Then we define $\textnormal{Ind}^H_G(M)\coloneq K^{\sigma}[G] \otimes_{K^{\sigma}[H]} M$ to be the \emph{induced} left $K^{\sigma}[G]$-module.
\end{definition}

\begin{proposition} \label{proposition: induced semilinear module}
Let $K$ be a field. Let $G$ be a group and $H \subseteq G$ be a subgroup. Let $\sigma \colon G \to \Aut(K)$ be a group homomorphism. Let $\alpha_1, \ldots, \alpha_n$ be left coset representatives for the quotient $G/H$. Let $M$ be a left $K^{\sigma}[H]$-module and set $M' \coloneq \bigoplus_{i=1}^n M_{\alpha_i}$ as the direct sum of $n$ copies of $M$ indexed by the $\alpha_i$'s. Then:

\begin{enumerate}
	\item $M'$ has the structure of left $K^{\sigma}[G]$-module;
	\item there is an isomorphism of left $K^{\sigma}[G]$-modules $\varphi \colon K^{\sigma}[G] \otimes_{K^{\sigma}[H]} M \rightarrow \bigoplus_{i=1}^n M_{\alpha_i}$ mapping $\lambda e_g \otimes m$ to $(\alpha_i^{-1}(\lambda)(\alpha_i^{-1}g)( m))_i$
	where $\alpha_i$ is such that $\alpha_i^{-1}g \in H$ (i.e.\ $g \in \alpha_i H$) and $(\alpha_i^{-1}(\lambda)(\alpha_i^{-1}g)( m))_i$ denotes the copy of $\alpha_i^{-1}(\lambda)(\alpha_i^{-1}g)( m)$ inside $M_{\alpha_i}$.
\end{enumerate}
\end{proposition}

\begin{proof}
(1) We define a multiplication on the left on $M'$ by elements of $K^{\sigma}[G]$ as follows:
\begin{align*}
	&g ( m_i ) = (h (m))_{j(i)} \qquad \textnormal{for } g \in G \textnormal{ with } g \alpha_i = \alpha_{j(i)} h, h \in H,\\
	&\lambda m_i =  (\alpha_i^{-1}(\lambda) m)_i \qquad \lambda \in K,
\end{align*}
and extended linearly to all $K^{\sigma}[G]$. We show that these operations define a left $K^{\sigma}[G]$-module structure on $M'$. In order to do that, we verify that the equality 
\begin{equation} \label{equation: left module}
	(\lambda_1g_1 \cdot \lambda_2 g_2)(m)_i = \lambda_1g_1 (\lambda_2g_2(m)_i)
\end{equation}
holds for all $\lambda_1,\lambda_2 \in K, g_1,g_2 \in G$ and $m \in M$. First we write $g_2 \alpha_i = \alpha_{j_2(i)} h_2$ for some $h_2 \in H$ and $g_1 \alpha_{j_2(i)}= \alpha_{j_1(i)}h_1$ for some $h_1 \in H$. Then $g_1 g_2 \alpha_i = g_1\alpha_{j_2(i)}h_2 = \alpha_{j_1(i)}h_1h_2$. We compute the left hand side of Equation \ref{equation: left module}:
\begin{align*}
	(\lambda_1g_1 \cdot \lambda_2 g_2)(m)_i &=  (\lambda_1g_1(\lambda_2)g_1g_2)(m)_i =  \\ &= \lambda_1g_1(\lambda_2)(h_1h_2(m))_{{j_1(i)}} =  \\
	&= (\alpha_{j_1(i)}^{-1}(\lambda_1g_1(\lambda_2)) h_1h_2(m))_{j_{1}(i)}.
\end{align*}
While the right hand side of Equation \ref{equation: left module} is:
\begin{align*}
	\lambda_1g_1 (\lambda_2 g_2)(m)_i &=  \lambda_1g_1 (\lambda_2(h_2(m))_{j_2(i)}) = \\ &= \lambda_1g_1 (\alpha_{j_2(i)}^{-1}(\lambda_2)h_2(m))_{j_2(i)} = \\
	 &= \lambda_1(h_1\alpha_{j_2(i)}^{-1}(\lambda_2) h_1h_2(m))_{j_1(i)} = \\ &= (\alpha_{j_1(i)}^{-1}(\lambda_1) h_1\alpha_{j_2(i)}^{-1}(\lambda_2)h_1h_2(m))_{j_1(i)} =  \\
	 &= (\alpha_{j_1(i)}^{-1}(\lambda_1g_1(\lambda_2)) h_1h_2(m))_{j_{1}(i)},
\end{align*}
where in the last equality we used that $h_1\alpha_{j_2(i)}^{-1} = \alpha_{j_1(i)}^{-1}g_1$.

(2) Let $\lambda_1, \lambda_2 \in K$, $g_1,g_2 \in G$ and $m \in M$. We first show that 
\begin{equation} \label{equation: morphism of modules}
	\varphi((\lambda_1 e_{g_1})(\lambda_2 e_{g_2} \otimes m)) = (\lambda_1e_{g_1}) \varphi(\lambda_2e_{g_2} \otimes m),
\end{equation}
so that $\varphi$ is a map of left $K^{\sigma}[G]$-modules. Write $g_1 = \alpha_{i(1)} h_1, g_2= \alpha_{i(2)}h_2$ and $g_1g_2= \alpha_{i(12)}h_{12}$ for some $h_1,h_2, h_{12} \in H$. In the following computations we will use that $g_1 \alpha_{i(2)}= g_1 g_2 h_2^{-1} = \alpha_{i(12)}h_{12}h_2^{-1}$. The left hand side of Equation \ref{equation: morphism of modules} is
\begin{align*}
	\varphi(\lambda_1e_{g_1}\lambda_2e_{g_2} \otimes m) &= 
	\varphi(\lambda_1g_1(\lambda_2)g_1g_2 \otimes m) = \\
	&= (\alpha_{i(12)}^{-1}(\lambda_1g_1(\lambda_2))h_{12}(m))_{i(12)} = \\
	&= (\alpha_{i(12)}^{-1}(\lambda_1)\alpha_{i(12)}^{-1}g_1(\lambda_2)h_{12}(m))_{i(12)}.
\end{align*}
While the right hand side of Equation \ref{equation: morphism of modules} is:
\begin{align*}
	(\lambda_1e_{g_1}) \varphi(\lambda_2e_{g_2} \otimes m) &= (\lambda_1 e_{g_1})(\alpha_{i(2)}^{-1}(\lambda_2)h_2(m))_{i(2)} = \\
	&= \lambda_1 (h_{12}h_2^{-1}(\alpha_{i(2)}^{-1}(\lambda_2)h_2(m)))_{i(12)} = \\
	&= (\alpha_{i(12)}^{-1}(\lambda_1)(h_{12}h_2^{-1}\alpha_{i(2)}^{-1}(\lambda_2) h_{12}(m)))_{i(12)} = \\
	&= (\alpha_{i(12)}^{-1}(\lambda_1)(h_{12}g_2^{-1}(\lambda_2) h_{12}(m)))_{i(12)} = \\
	&= (\alpha_{i(12)}^{-1}(\lambda_1)(\alpha_{i(12)}^{-1}g_1(\lambda_2) h_{12}(m)))_{i(12)}.
\end{align*}

Finally, the morphism $\psi \colon M' \to K^{\sigma}[G] \otimes_{K^{\sigma}[H]} M$ mapping $(m)_i \mapsto e_{\alpha_i} \otimes m$ is a morphism of left $K^{\sigma}[G]$-modules and is an explicit inverse of $\varphi$.
\end{proof}

\section{Constructing equivariant étale morphisms} \label{section: constructing equivariant étale morphisms}

Let $G$ be a finite abelian group acting on a smooth $k$-scheme $X$.
In order to prove Theorem \ref{theorem: gabber}, we will need to exhibit an equivariant morphism from $X$ to an affine space with suitable linear $G$-action that is étale at some points of $X$. The existence of such morphisms is nontrivial, as we now explain. Suppose for simplicity that $k=\overline{k}$ is an algebraically closed field. Let $H \subset G$ be a subgroup. Suppose that $x \in X^H$ is a closed point fixed by the subgroup $H$. Then $T_x X$ is a $H$-representation over $k$. If $V$ is a $G$-representation over $k$ such that $X \subseteq \AA(V)$ and $T_xX \subseteq \textnormal{Res}^G_H(V)$ is an inclusion of $H$-representations over $k$, we may choose a $H$-equivariant splitting $V \to T_{x}X$, providing a $H$-equivariant map $X \to \AA(V) \to \AA(T_{x}X)$ which is, by construction, étale at $x$. However, if $H \subsetneq G$, we cannot lift in general this map to a $G$-equivariant morphism. This is not needed in \cite{Bachmann} (and neither in the case of a cyclic group of order a prime), as the only fixed-point case is for $H=C_2$ being the whole group.
We will instead assume that $k$ is a perfect field. In this setting, under suitable hypotheses on the group $G$, we prove in Proposition \ref{proposition: étale morphism as in 8.10 of equivariant cycles} that there always exists a $G$-equivariant morphism from $X$ to some $G$-representation which is étale at the point of interest. Moreover, the existence of such an étale map is based on Lemma \ref{lemma: exists representation 1}, thus it is important that the group $G$ is a finite abelian group.

\begin{remark}
Let $G$ be a group and let $X$ be a scheme with a $G$-action. Let $x \in X$. Then the tangent space $T_x X \in \textnormal{Rep}^{\sigma}_{\kappa(x)}(S_x)$ is a $\sigma$-semilinear $S_x$-representation, where $\sigma \colon S_x \to \Aut_k(\kappa(x))$ is the group homomorphism determined by the $G$-action on $X$. We will denote $\textnormal{Rep}^{\sigma}_{\kappa(x)}(S_x)$ by $\textnormal{Rep}^{\#}_{\kappa(x)}(S_x)$ and the twisted group ring $\kappa(x)^{\sigma}[S_x]$ by $\kappa(x)^{\#}[S_x]$.
\end{remark}

\begin{lemma} \label{lemma: exists representation 1}
Let $G$ be a finite abelian group such that $k$ contains a primitive $n$th root of unity, where $n$ is the exponent of $G$. Let $X$ be a scheme with a $G$-action and let $x \in X$ be a closed point. Then:
\begin{enumerate}
	\item there exists $V \in \textnormal{Rep}_k(G)$ such that $\textnormal{Res}^G_{S_{x}}(V) \otimes_k \kappa(x) \simeq T_{x}X$ in $\textnormal{Rep}^{\#}_{\kappa(x)}(S_x)$;
	\item if $G_{x}= \left\{e\right\}$, then for any $V \in \textnormal{Rep}_k(G)$ such that $\dim_{k}(V)= \dim_{\kappa(x)}T_{x}X$, there is an isomorphism $\textnormal{Res}^G_{S_{x}}(V) \otimes_k \kappa(x) \simeq T_{x}X$ in $\textnormal{Rep}^{\#}_{\kappa(x)}(S_x)$.
\end{enumerate} 
\end{lemma}
\begin{proof}
Conclusion (1) follows by combining the essential surjectivity of the functor $- \otimes_{k} \kappa(x) \colon \textnormal{Rep}_{k}(S_x) \to \textnormal{Rep}^{\#}_{\kappa(x)}(S_x)$ of Lemma \ref{lemma: essentially surjectivity} and the essential surjectivity of the functor $ \textnormal{Res}^{G}_{S_x}(-) \colon \textnormal{Rep}_{k}(G) \to \textnormal{Rep}_{k}(S_x)$ of Lemma \ref{lemma: lemma on linear representations}. To prove (2) we see that if $G_x= \left\{e\right\}$, by Lemma \ref{lemma: essentially surjectivity} there exists just one irreducible semilinear representation in $\textnormal{Rep}^{\#}_{\kappa(x)}(S_x)$ up to isomorphism. Since by Lemma \ref{lemma: semisimple} the category $\textnormal{Rep}^{\#}_{\kappa(x)}(S_x)$ is semisimple and it has one irreducible object up to isomorphism, any two semilinear representations in $\textnormal{Rep}^{\#}_{\kappa(x)}(S_x)$ of the same dimension are isomorphic.
\end{proof}

The following proposition is a simplified version of \cite[Lemma 8.10]{Equivariantcycles}, and the proof is identical.	Unlike \cite[Lemma 8.10]{Equivariantcycles}, we state the result in a form that yields an étale map at a given point, without restricting to an equivariant open subset on which the map is étale. We will then pass to a Nisnevich neighborhood where the map is étale in the proof of Theorem \ref{theorem: gabber}, using Lemma \ref{lemma: assume is étale and quasi-finite}.  We need the hypothesis that $k$ is perfect in order to apply Galois descent.

\begin{proposition} \label{proposition: étale morphism as in 8.10 of equivariant cycles}
Let $k$ be perfect and let $G$ be a finite abelian group such that $k$ contains a primitive root of unity of order the exponent of $G$. Let $X$ be an affine smooth $k$-scheme with a $G$-action and let $x \in X$ be a closed point. Let $V \in \textnormal{Rep}_k(G)$ such that $\textnormal{Res}^G_{S_x}(V) \otimes_k \kappa(x) \simeq T_x X$ in $\textnormal{Rep}^{\#}_{\kappa(x)}(S_x)$. Then there exists a $G$-equivariant morphism $\varphi \colon X \rightarrow \AA(V)$ which is étale at $x$.
\end{proposition} 
\begin{proof}
The proof is the same as that of \cite[Lemma 8.10]{Equivariantcycles}, but is easier because in our case we put $Z =X$.
Choose a finite Galois extension $L/k$ together with a $k$-embedding $\iota \colon \kappa(x) \hookrightarrow L$. Such an extension exists because $k$ is a perfect field, and thus $\kappa(x)/k$ is a finite separable field extension. Set $\Gamma \coloneq \textnormal{Gal}(L/k)$. The group $G \times \Gamma$ acts on the field $L$ via the group homomorphism $G \times \Gamma \xrightarrow{\textnormal{pr}} \Gamma  = \Aut_k(L)$.
If $A,B$ are $k$-algebras with a $G$-action, we can endow $A \otimes_k L$ and $B \otimes_k L$ with the diagonal $G \times \Gamma$-action. If that is the case, we have $(A \otimes_k L )^{\Gamma} = A$ and $(B \otimes_k L)^{\Gamma} =B$, where $\Gamma$ is identified with the subgroup $\left\{e\right\} \times \Gamma \subseteq G \times \Gamma$. If $\varphi_L \colon A \otimes_k L \to B \otimes_k L$ is a $G \times \Gamma$-equivariant morphism, then it restricts to the $\Gamma$-invariant elements and gives a $G$-equivariant morphism $\varphi \colon A \to B$.

If $S$ is an affine $k$-scheme with a $G$ action, we endow $S_L = S \times_k \Spec(L)$ with the diagonal $G \times \Gamma$-action.
Suppose that we have constructed a $G \times \Gamma$-equivariant morphism $\varphi_L\colon X_L \to \AA(V)_L$ which is étale at each point of the $G \times \Gamma$-orbit of $y$, where $y \in X_L$ is a point lying over $x$. Then the $G \times \Gamma$-equivariant morphism $\varphi_L \colon X_L \to \AA(V)_L$ descends to a $G$-equivariant morphism on the quotients $X \simeq X_L/\Gamma$ and $\AA(V) \simeq \AA(V)_L /\Gamma$, so that we have a $G$-equivariant morphism $\varphi \colon X \to \AA(V)$. We claim that $\varphi$ is étale at $x$. Consider the following cartesian diagram, where each horizontal arrow is a $\Gamma$-torsor:
% https://q.uiver.app/#q=WzAsNixbMCwwLCJYX3tMfSJdLFswLDEsIlxcQUEoVykgX0wiXSxbMSwxLCJcXEFBKFcpIl0sWzEsMCwiWCJdLFswLDIsIlxcU3BlYyhMKSJdLFsxLDIsIlxcU3BlYyhrKSJdLFszLDIsIlxcdmFycGhpIl0sWzEsMl0sWzAsMSwiXFx2YXJwaGlfTCIsMl0sWzAsM10sWzIsNV0sWzEsNF0sWzQsNV1d
\[\begin{tikzcd}
	{X_{L}} & X \\
	{\AA(V) _L} & {\AA(V)} \\
	{\Spec(L)} & {\Spec(k).}
	\arrow[from=1-1, to=1-2]
	\arrow["{\varphi_L}"', from=1-1, to=2-1]
	\arrow["\varphi", from=1-2, to=2-2]
	\arrow[from=2-1, to=2-2]
	\arrow[from=2-1, to=3-1]
	\arrow[from=2-2, to=3-2]
	\arrow[from=3-1, to=3-2]
\end{tikzcd}\]

Since $\varphi_L$ is étale at $y$, it follows that $\varphi$ is étale at $x$, because étaleness is an fpqc (even étale) local property on the base (see, for example, \cite[\href{https://stacks.math.columbia.edu/tag/0476}{Tag 0476}]{stacks-project}).
We thus need to prove the existence of a $G\times \Gamma$-equivariant morphism $\varphi_L$ which is étale at a point $y \in X_L$ over $x$.

Write $X \coloneq \Spec(R)$ for some $k$-algebra $R$. Let $\sigma \colon G \times \Gamma \to \Aut_k(L)= \Gamma$ be the group homomorphism defined by the projection to $\Gamma$. The $L$-algebra $R \otimes_k L$ has the structure of a left $L^{\sigma}[G \times \Gamma]$-module as follows:
\[
(g,\gamma) (l (r \otimes l')) = \gamma(l)(g(r) \otimes \gamma(l')) \qquad \text{for all } (g, \gamma) \in G \times \Gamma, r \in R, l,l' \in L.
\]
Let $y_1 \in X_L$ be a point over $x$. Then $y_1 \in \Spec(\kappa(x) \otimes_k L)$ corresponds to a $k$-linear homomorphism $\kappa(x) \to L$. Without loss of generality we may assume that $y_1$ corresponds to the embedding $\iota \colon \kappa(x) \to L$. Let $S_{y_1} \subseteq G \times \Gamma$ be the set-theoretic stabilizer of $y_1$. Let $\{\alpha_1, \ldots, \alpha_n\}$ be a complete set of left coset representatives for $(G \times \Gamma)/S_{y_1}$. Set $y_i \coloneq \alpha_i(y_1)$ for $i=1, \ldots, n$. Then the set-theoretic orbit of $y_1$ is $\left\{y_1, \ldots, y_n \right\}$. Let $\frakm_i$ be the maximal ideal in $R_L \coloneq R \otimes_k L$ corresponding to the point $y_i$, for $i=1, \ldots, n$ respectively. By construction, $\alpha_i^{-1}(\frakm_i) = \frakm_1$ for all $i=1, \ldots, n$. Note that $I \coloneq \bigcap \frakm_i$ is a $G \times \Gamma$-invariant ideal of $R_L$. In particular, it is a left $L^{\sigma}[G \times \Gamma]$-sub-module of $R_L$. For all $i=1, \ldots, n$ we have the quotient morphism $\frakm_i \to \frakm_i /\frakm_i^2$. All together, these yield a surjection of left $L^{\sigma}[G \times \Gamma]$-modules
\begin{align*}
	F \colon \bigcap_{i=1}^n \frakm_i &\longrightarrow \frakm_1 / \frakm_1 ^2 \times \ldots \times \frakm_n /\frakm_n^2 \\
	a &\longmapsto ([a]_1, \ldots, [a]_n)
\end{align*}
by the Chinese remainder theorem.

Let $\tau = (g, \gamma) \in G \times \Gamma$. Then for all $i=1, \ldots n$ we can write $\tau \alpha_i = \alpha_{\tau(i)} s$ for some $s \in S_{y_1}$ and some $\tau(i) \in \left\{1, \ldots, n\right\}$. 

We define an isomorphism of left $L^{\sigma}[G \times \Gamma]$-modules as follows:
\begin{align*}
	\mu \colon \frakm_1 /\frakm_1^2 \times \ldots \times \frakm_n/\frakm_n^2 &\longrightarrow \textnormal{Ind}^{S_{y_1}}_{G \times \Gamma}(\frakm_1/\frakm_1^2) \\
	(0, \ldots, [a]_i, \ldots, 0) &\longmapsto  [{\alpha_i^{-1}(a)}]_i,      
\end{align*}
where we use notation of Proposition \ref{proposition: induced semilinear module} for the description of the induced semilinear module.
Indeed, we can also give an explicit inverse of $\mu$ as
\begin{align*}
	\mu^{-1} \colon  \textnormal{Ind}^{S_{y_1}}_{G \times \Gamma}(\frakm_1/\frakm_1^2) &\longrightarrow \frakm_1 /\frakm_1^2 \times \ldots \times \frakm_n/\frakm_n^2 \\
	[a]_i &\longmapsto 
	(0, \ldots, [\alpha_i(a)]_i, \ldots, 0).
\end{align*}
Set $F' \coloneq \mu \circ F$. By construction, $F' \colon \bigcap_{i=1}^n \frakm_i \to \textnormal{Ind}^{S_{y_1}}_{G \times \Gamma}(\frakm_1/\frakm_1^2)$ is a surjection of left $L^{\sigma}[G \times \Gamma]$-modules and maps $a \mapsto \oplus_{i=1}^n [\alpha_i^{-1}(a)]_i$.

Note that $S_{y_1} \subseteq S_x \times \Gamma$ consists of pairs $(g, \gamma)$ such that the two maps $\kappa(x) \to L$ given by $\iota \circ g$ and $\gamma \circ \iota$ are equal.
Note further that $(\textnormal{Res}^G_{S_x}(V) \otimes_k \kappa(x)) \otimes_{\kappa(x)} L \simeq \textnormal{Res}^G_{S_x}(V) \otimes_k L$ as $L^{\sigma}[S_{y_1}]$-modules, because the canonical isomorphism $\kappa(x) \otimes_{\kappa(x)} L \xrightarrow{\sim} L$ is $L^{\sigma}[S_{y_1}]$-linear.  By hypothesis there exists an isomorphism $T_{x}X \simeq \textnormal{Res}^G_{S_x}(V) \otimes_k \kappa(x)$ in $\textnormal{Rep}^{\#}_{\kappa(x)}(S_x)$. In particular, they are isomorphic as left $\kappa(x)^{\#}[S_x]$-modules. Then we can tensor with $L$ to obtain an isomorphism of left $L^{\sigma}[S_{y_1}]$-modules:
\[
T_{y_1}X_L \simeq T_{x}(X) \otimes_{\kappa(x)} L \simeq \textnormal{Res}^G_{S_x}(V) \otimes_k L.
\]
On the other hand, $T_{y_1} X_L \simeq (\frakm_1/\frakm_1^2)^{\vee}$ as left $L^{\sigma}[S_{y_1}]$-modules. Thus we get an isomorphism $\frakm_1/\frakm_1^2 \xlongrightarrow{\sim} (\textnormal{Res}^G_{S_x}(V) \otimes_k L)^{\vee}$ of left $L^{\sigma}[S_{y_1}]$-modules.
Composing with $F'$ we get a surjection of left $L^{\sigma}[G \times \Gamma]$-modules 
\[
F'' \colon \bigcap_{i=1}^n \frakm_i \xlongrightarrow{F'} \textnormal{Ind}^{S_{y_1}}_{G \times \Gamma}(\frakm_1/\frakm_1^2) \xlongrightarrow{\sim} \textnormal{Ind}^{S_{y_1}}_{G \times \Gamma}((\textnormal{Res}^G_{S_x}(V) \otimes_k L)^{\vee}).
\]
The ring $L^{\sigma}[G \times \Gamma]$ is semisimple by Lemma \ref{lemma: semisimple}, because $\ker(\sigma) = G \times \left\{e\right\}$ so that $|\ker(\sigma)| = |G|$ and $|G|$ is invertible in $k$ because $k$ contains a primitive root of unity of order the exponent of $G$ by hypothesis. It follows that there exists an $L^{\sigma}[G \times \Gamma]$-module section \[
s \colon \textnormal{Ind}^{S_{y_1}}_{G \times \Gamma}((\textnormal{Res}^G_{S_x}(V) \otimes_k L)^{\vee}) \longrightarrow \bigcap_{i=1}^n \frakm_i
\]
of $F''$, so that $F'' \circ s = \textnormal{id}$. We claim that there exists also a morphism of left $L^{\sigma}[G \times \Gamma]$-modules
\begin{align*}
	c \colon (V \otimes_k L)^{\vee} &\longrightarrow \textnormal{Ind}^{S_{y_1}}_{G \times \Gamma}((\textnormal{Res}^G_{S_x}(V) \otimes_k L)^{\vee}) \\
	u &\longmapsto  \oplus_{i=1}^n (\alpha_i^{-1}(u))_i,
\end{align*}
where each $\alpha_i^{-1}(u)$ is an element of $(V \otimes_k L)^{\vee}$ and we view it as an element of $(\textnormal{Res}^G_{S_x}(V) \otimes_k L)^{\vee}$. Indeed, let $\lambda \in L$ and $\omega \in G \times \Gamma$. Write $\omega \alpha_i = \alpha_{\omega(i)}s_i$ for some $s_i \in S_{y_1}$. We have $c(lw(u)) = \oplus_{i=1}^n (\alpha_i^{-1}(l\omega(u)))_i$; while 
\begin{align*}
l\omega(c(u)) &= l\omega(\oplus_{i=1}^n (\alpha_i^{-1}(u))_i) =  \\
&= l(\oplus_{i=1}^n (s_i \alpha_i^{-1}(u))_{\omega(i)}) = \\
&= \oplus_{i=1}^n (\alpha_{\omega(i)}^{-1}(l) s_i \alpha_i^{-1}(u))_{\omega(i)} = \\
&= \oplus_{i=1}^n (\alpha_{\omega(i)}^{-1}(l)\alpha_{\omega(i)}^{-1}\omega(u))_{\omega(i)} = \\
&= \oplus_{i=1}^n (\alpha_{\omega(i)}^{-1}(l \omega(u))_{\omega(i)} = \\
&= \oplus_{i=1}^n (\alpha_i^{-1}(lw(u)))_i,
\end{align*}
where in the last equality we used that $\alpha_i \mapsto \alpha_{\omega(i)}$ is a permutation of the $\alpha_i$'s, and this proves that $c$ is a morphism of left $L^{\sigma}[G \times \Gamma]$-modules.

We put together all the maps to obtain the following diagram of left $L^{\sigma}[G \times \Gamma]$-modules:
% https://q.uiver.app/#q=WzAsNixbMCwxLCIoV18xIFxcb3RpbWVzIEwpXntcXHZlZX0iXSxbMSwxLCJcXHRleHRub3JtYWx7SW5kfSgoXFx0ZXh0bm9ybWFse1Jlc30oV18xKSBcXG90aW1lc19rIEwpXntcXHZlZX0pIl0sWzIsMSwiXFxiaWdjYXBfe2l9IFxcZnJha21faSJdLFsyLDAsIlxcdGV4dG5vcm1hbHtJbmR9KChcXHRleHRub3JtYWx7UmVzfShXXzEpIFxcb3RpbWVzX2sgTClee1xcdmVlfSkiXSxbMywxLCJcXHRleHRub3JtYWx7SW5kfShcXGZyYWttXzEvXFxmcmFrbV8xXjIpIl0sWzIsMiwiUl9MIl0sWzAsMSwiYyJdLFsxLDIsIkciXSxbMiwzLCJGJyIsMl0sWzIsNCwiRiJdLFsxLDMsIlxcdGV4dG5vcm1hbHtpZH0iXSxbMyw0LCJcXHNpbSJdLFsyLDUsIlxcaW90YSIsMCx7InN0eWxlIjp7InRhaWwiOnsibmFtZSI6Imhvb2siLCJzaWRlIjoidG9wIn19fV1d
\[\begin{tikzcd}
	&& {\textnormal{Ind}^{S_{y_1}}_{G \times \Gamma}((\textnormal{Res}^G_{S_x}(V) \otimes_k L)^{\vee})} & \\
	{(V \otimes_k L)^{\vee}} & {\textnormal{Ind}^{S_{y_1}}_{G \times \Gamma}((\textnormal{Res}^G_{S_x}(V) \otimes_k L)^{\vee})} & {\bigcap_{i} \frakm_i} & {\textnormal{Ind}^{S_{y_1}}_{G \times \Gamma}(\frakm_1/\frakm_1^2)} \\
	&& {R_L.}
	\arrow["\sim", from=1-3, to=2-4]
	\arrow["c", from=2-1, to=2-2]
	\arrow["{\textnormal{id}}", from=2-2, to=1-3]
	\arrow["s", from=2-2, to=2-3]
	\arrow["{F''}"', from=2-3, to=1-3]
	\arrow["F'", from=2-3, to=2-4]
	\arrow["\nu", hook, from=2-3, to=3-3]
\end{tikzcd}\]

Let $\{f_1,\ldots,f_t\}$ be a basis of $(V\otimes_k L)^\vee$ as an $L$-vector space. Note that $c(f_j) = \bigoplus_{i=1}^n (\alpha_i^{-1}(f_j))_i$. For each fixed index $i = 1, \ldots, n$, the set $\{\alpha_i^{-1}(f_j)\}_j$
is again a basis of $(V\otimes_k L)^\vee$, because $\alpha_i^{-1}$ acts as an $L$-semilinear automorphism. This is also a basis of $(\textnormal{Res}^G_{S_x}(V)\otimes_k L)^\vee$ as an $L$-vector space, and the isomorphism $(\textnormal{Res}^G_{S_x}(V)\otimes_k L)^\vee
\xrightarrow{\sim}
\frakm_1/\frakm_1^2$ of left $L^{\sigma}[S_{y_1}]$-modules
sends the basis $\{\alpha_i^{-1}(f_j)\}_j$ to some basis $\{[u_j^i]\}_j$ of $\frakm_1/\frakm_1^2$ as $L$-vector space.

The induced isomorphism $\textnormal{Ind}^{S_{y_1}}_{G \times \Gamma}((\textnormal{Res}^G_{S_x}(V)\otimes_k L)^\vee)
\xrightarrow{\sim}
\textnormal{Ind}^{S_{y_1}}_{G \times \Gamma}(\frakm_1/\frakm_1^2)$ therefore sends the element $(\alpha_i^{-1}(f_j))_i$ to $[u_j^i]_i$ for each $i=1, \ldots ,n$ and for all $j$.
Under the isomorphism $\mu^{-1} \colon
\textnormal{Ind}^{S_{y_1}}_{G \times \Gamma}(\frakm_1/\frakm_1^2)
\xrightarrow{\sim}
\prod_{i=1}^n \frakm_i/\frakm_i^2$,
the element $[u_j^i]_i$ maps to $(0,\ldots,[\alpha_i (u_j^i)]_i,\ldots,0)$. Hence, after composing with the projection $\textnormal{pr}_i$ onto the $i$-th factor, we obtain the element $[\alpha_i(u_j^i)]\in \frakm_i/\frakm_i^2$.

It follows that for every fixed index $i$, the composite map $\textnormal{pr}_i \circ \mu^{-1}\circ F' \circ s \circ c \colon (V\otimes_k L)^\vee
\rightarrow
\frakm_i/\frakm_i^2$ sends the basis $\{f_j\}_j$ to the set of elements $\{[\alpha_i(u_j^i)]\}_j$.
Since $\{[u_j^i]\}_j$ is a basis of $\frakm_1/\frakm_1^2$ as $L$-vector space, then $\{[\alpha_i (u_j^i)]\}_j$ is a basis of $\frakm_i/\frakm_i^2$ as $L$-vector space. Therefore, for every $i$, the composite map $(V\otimes_k L)^\vee
\rightarrow
\frakm_i/\frakm_i^2$ is an isomorphism of $L$-vector spaces.

Observe now that the morphism $\nu \circ s \circ c \colon (V \otimes_k L)^{\vee} \to R_L$ extends to $\psi\colon \textnormal{Sym}_L((V \otimes_k L)^{\vee}) \to R_L$. By construction, $\psi$ is $G \times \Gamma$-equivariant. Let $I = \bigoplus_{d \geq 1}  \textnormal{Sym}^d((V \otimes_k L)^{\vee})$ be the irrelevant ideal. By construction, $\psi((V \otimes_k L)^{\vee}) \subseteq \bigcap \frakm_i \subset \frakm_i$ for all $i$. Then $(V \otimes_k L)^{\vee} \subset \psi^{-1}(\frakm_i)$ for all $i=1, \cdots,n$. On the other hand, the ideal generated by $(V \otimes_k L)^{\vee}$ is exactly the irrelevant ideal in $A \coloneq \textnormal{Sym}((V \otimes_k L)^{\vee})$. It follows that $\psi^{-1}(\frakm_i) = I$ for all $i$, i.e.\ all of the points $y_i$ go to the origin of $\Spec(A)$ (because the origin corresponds to the irrelevant ideal). The induced map on cotangent spaces is given by the isomorphisms
\[
I/I^2 = \bigoplus_{d \geq 1}  \textnormal{Sym}^d((V \otimes_k L)^{\vee})/ \bigoplus_{d \geq 2}  \textnormal{Sym}^d((V \otimes_k L)^{\vee}) \simeq (V \otimes_k L)^{\vee} \xlongrightarrow{\sim} \frakm_i/\frakm_i^2.
\]

Finally note that 
\[
A =\textnormal{Sym}_L((V \otimes_k L)^{\vee}) \simeq \textnormal{Sym}_L(V^{\vee} \otimes_k L) \simeq \textnormal{Sym}_k(V^{\vee}) \otimes_k L
\]
where the last isomorphism is \cite[(11.1.2)]{GortzWedhorn1}. It follows that the induced morphism $\varphi_L\colon \Spec(R_L)= X_L  \rightarrow  	\Spec(A) \simeq \AA(V)_L$ is $G\times \Gamma$-equivariant and, by the Jacobian criterion, is étale at the points $y_1, \ldots, y_n$.
\end{proof}

\begin{example}
Let $X = V(x_1x_2-1) \subseteq \Spec (\CC[x_1,x_2])$ with $C_2$-action given by the sign action on each coordinate. Let $x = (1,1) \in X$ and note that $S_x = G_x = \left\{e \right\}$. Then by Lemma \ref{lemma: exists representation 1}.(2) we can choose any one dimensional $C_2$-representation $V$ in order to have an isomorphism $T_xX \simeq \textnormal{Res}^{C_2}_{\left\{e\right\}}(V)$ of $\CC$-vector spaces, satisfying hypotheses of Proposition \ref{proposition: étale morphism as in 8.10 of equivariant cycles}. Then, depending on the chosen $V$, we can follow the proof of Proposition \ref{proposition: étale morphism as in 8.10 of equivariant cycles} and see that it gives the following morphisms:
\begin{itemize}
	\item if $V = \CC$ with trivial $C_2$-action, then we get $\varphi \colon X \to \AA^1_{\CC}$ determined by $\varphi \coloneq \frac{x_1^2-1}{2}$;
	\item if $V = \CC$ with sign $C_2$-action, then we get $\varphi \colon X \to \AA^1_{\CC}$ determined by $\varphi \coloneq \frac{x_1 (x_1^2-1)}{2}$.
\end{itemize}
\end{example}

\section{Constructing equivariant quasi-finite morphisms} \label{section: constructing equivariant quasifinite morphisms}

Our next goal is to generalize \cite[Lemma 2.5]{Bachmann} to finite abelian groups $G$, and then apply the result in order to produce equivariant morphisms that are quasi-finite at a fixed closed point. We will use \cite[Lemma 2.4]{Bachmann} which we recall here as Lemma \ref{lemma: Bachmann 2.4} for the convenience of the reader.

Recall that a morphism of schemes $f \colon X \to Y$ is said to be \emph{radicial} if the underlying map of topological spaces is injective and if for every $x \in X$ the field extension $\kappa(x)/\kappa(f(x))$ is purely inseparable. A morphism is radicial if and only if is universally injective by \cite[\href{https://stacks.math.columbia.edu/tag/01S4}{Tag 01S4}]{stacks-project}.

\begin{definition}
Let $f \colon X \to Y$ be a morphism of schemes. Let $x \in X$ be a closed point. We say that $f$ is \emph{radicial at $x$} if the composite $\Spec(\kappa(x)) \xrightarrow{x} X \xrightarrow{f} Y$ is a radicial morphism.
\end{definition}

\begin{remark}
It is clear that $f$ is radicial at the closed point $x \in X$ if and only if $\kappa(x)/\kappa(f(x))$ is a purely inseparable field extension.
\end{remark}

\begin{lemma}\label{lemma: Bachmann 2.4}
Let $k$ be infinite and let $X$ be an affine finite type $k$-scheme. Let $S \subset X$ be a finite set of closed points. Then there exists a morphism $f\colon X \to \AA^1_k$ such that the composite $S \to X \to \AA^1_k$ is radicial. More generally, suppose that $T \subset X$ is another finite set of closed points disjoint from $S$ and let $\iota \colon X \hookrightarrow \AA^N_k$ be a closed immersion. Let $n \coloneq \sum_{t \in T} [\kappa(t) : k]$. Then there exists a polynomial function $f'\colon X \to \AA^1_k$ (with respect to the embedding $\iota$) of degree $\leq n+1$, which does not vanish on points of $S$, vanishes on points of $T$ and the composite $S \to X \to \AA^1_k$ is radicial.
\end{lemma}
\begin{proof}
See \cite[Lemma 2.4]{Bachmann}.
\end{proof}

\begin{remark} \label{remark: fix equivariant closed immersion}
Let $G$ be a finite group and let $A$ be a finitely generated $k$-algebra with a $G$-action.
Set $X \coloneq \Spec(A)$. Choose $k$-algebra generators $a_1, \ldots, a_N$ of $A$, and let $W \subseteq A$ be the finite-dimensional $G$-stable $k$-subspace spanned by all $g(a_i)$, with $g \in G$ and $1 \leq i \leq N$. The inclusion $W \hookrightarrow A$ induces a surjective $G$-equivariant homomorphism $\textnormal{Sym}(W) \twoheadrightarrow A$. If $V \coloneq W^{\vee}$, then $\textnormal{Sym}(V^{\vee}) = \textnormal{Sym}(W)$, so this surjection defines a $G$-equivariant closed immersion $X \hookrightarrow \AA(V)$.
\end{remark}

\begin{lemma}\label{lemma: generalization of Bachmann 2.5.(1)}
Let	$k$ be infinite and let $G$ be a finite abelian group. Let $X$ be an affine finite type $k$-scheme with a $G$-action. Let $x \in X$ be a closed point and let $S\subset X$ be a finite set of closed points disjoint from $Gx$.
Then there exists a $G$-invariant morphism $F\colon X \rightarrow \mathbb{A}^1_k$	satisfying $F(x)\notin F(S)$.
\end{lemma}

\begin{proof}
We may assume that the points of $S$ have pairwise distinct set-theoretic orbits. Indeed, if the set-theoretic orbits of two points $s, s' \in S$ intersect, then $F(s)=F(s')$ and thus the conclusion holds in that case too.
The GIT quotient (\cite[\S1.2]{GIT}) $X/G$ is an affine $k$-scheme of finite type.
Let $\pi :  X\to X/G$ be the quotient morphism and set $S' \coloneqq \{\pi(x)\} \cup \pi(S)$.
We claim that $S'$ is a finite set of distinct closed points of $X/G$. Indeed the geometric fibers of $\pi$ are the orbits of geometric points of $X$. If $\pi(x)=\pi(s)=y \in X/G$ for some $s \in S$, then any geometric fiber over $y$ would contain geometric points over $x$ and $s$, and thus $s \in Gx$, which is absurd. The same holds for all the elements of $S$ because we assumed they have pairwise distinct set-theoretic orbits.
Since $k$ is infinite and $S'$ consists of closed points, by Lemma \ref{lemma: Bachmann 2.4}, there exists a morphism
$\gamma \colon X/G \to \mathbb{A}^1_k$ that is injective on $S'$.
The composite $F \coloneqq \gamma \circ \pi \colon X \rightarrow \mathbb{A}^1_k$ is $G$-invariant and satisfies $F(x)\notin F(S)$ by construction.
\end{proof}

\begin{lemma}\label{lemma: generalization of Bachmann 2.5.(2,3)}
Let $k$ be infinite and let $G$ be a finite abelian group such that $|G|$ is invertible in $k$. Let $X$ be an affine finite type $k$-scheme with a $G$-action.
Let $S\subset X$ be a finite set of closed points and let $V \in \textnormal{Rep}_k(G)$ such that $\dim_k(V)=1$ and such that $\textnormal{Res}^G_{G_s}(V)$ is the trivial one-dimensional $G_s$-representation for all $s \in S$. Then:
\begin{enumerate}
	\item[(1)] there exists a $G$-equivariant morphism $F\colon X \rightarrow \AA(V) \simeq \AA^1_k$ such that $F(s)\neq 0$ for all $s\in S$; 
	
	\item[(2)] if $T\subset X$ is a finite set of closed points disjoint from $\bigcup_{s\in S} Gs$, then $F$ can be chosen so that $F(T)=0$.
\end{enumerate}
\end{lemma}

\begin{proof}
	We just need to prove $(2)$, since $(1)$ can be obtained by setting $T = \emptyset$.
	Without loss of generality, we may assume that distinct points of $S$ lie in distinct $G$-orbits. Indeed, suppose that $s_1, s_2 \in S$ belong to the same $G$-orbit, so that $g(s_1)=s_2$ for some $g \in G$. In that case, we remove $s_2$ from $S$. Repeating this process, we obtain a subset $S' \subseteq S$ containing at most one point from each $G$-orbit. Moreover, since $S' \subseteq S$, we still have that $T$ is disjoint from $\bigcup_{s \in S'} Gs$. Assume that we proved (1),(2) for $S'$ and $T$. Then the function $F \colon X \to \AA(V)$ also satisfies property (1) for the set $S$, because if $s \in S \setminus S'$, then $s = g(s')$ for some $g \in G$ and $s' \in S'$; then we have $F(s) = F(g(s')) = g(F(s')) \neq 0$ because $F(s') \neq 0$ and $g \in \Aut_k(V)$. Thus, replacing $S$ by $S'$, we may assume that the points of $S$ lie in distinct $G$-orbits.
	
	Let $\iota \colon X \hookrightarrow \AA^N_k$ be a $G$-equivariant closed immersion into an $N$-dimensional $G$-representation, which is possible by Remark \ref{remark: fix equivariant closed immersion}. Set $d \coloneq |G| \cdot (\sum_{s \in S}[\kappa(s) : k] + \sum_{t \in T}[\kappa(t) : k])+1$. Let $\cE$ be the affine space over $k$ associated to the $k$-vector space $E$ of polynomial functions $f' \in k[x_1, \ldots, x_N]$ of degree $\leq d$. Any polynomial function $f \in E$ defines a morphism $F \colon \AA^N_k \to \AA^1_k$. From this we can define a new morphism $\textnormal{Sym}(F) \colon \AA^N_k \to \AA^1_k$ by setting $\textnormal{Sym}(F) \coloneq \sum_{g \in G} g^{-1} \circ F \circ g$. Note that $\textnormal{Sym}(F)$ is a $G$-equivariant morphism, and it corresponds to the polynomial $f' = \sum_{g \in G} \chi(g)^{-1} g(f)$. Consider the $k$-vector subspace $J \subseteq E$ of those polynomials $f \in E$ such that the morphism $\textnormal{Sym}(F)$ vanishes on the closed points of $T$ (which we view as closed points of $\AA^N_k$ using the closed immersion $\iota$). Then $J$ corresponds to an affine subspace $\cJ$ of $\cE$. Inside $\cJ$ we can consider the open subspace $U \subseteq \cJ$ whose $k$-rational points corresponds to polynomials $f$ such that $\textnormal{Sym}(F)$ does not vanish on $S$. To prove the claim it is sufficient to show that $U$ has a $k$-rational point $f \in U(k)$, so that the claim is proved by the corresponding morphism $\textnormal{Sym}(F) \circ \iota$. In order to do that, since $k$ is infinite, it is sufficient to show that $U$ has a $\overline{k}$-rational point, where $\overline{k}$ is an algebraic closure of $k$. Let $\overline{X}$ be the base change of $X$ to the algebraic closure $\overline{k}$. Let $\overline{S}$, $\overline{T}$ be the closed points of $\overline{X}$ lying over $S$ and $T$ respectively. We define the following set made of closed points of $\overline{X}$: 
	\[
	W \coloneq \bigcup_{\overline{t} \in \overline{T}} G(\overline{t}) \cup \bigcup_{\overline{s} \in \overline{S}} G(\overline{s}) \setminus \{\overline{s}\}.
	\] 
	By construction, the sets $W$ and $\overline{S}$ are disjoint and $|W| < d$. Then by Lemma \ref{lemma: Bachmann 2.4} applied to $\overline{X}$, $\overline{S}$ and $W$, there exists a polynomial $\overline{f} \in \overline{k}[x_1, \ldots, x_N]$ with $\deg(\overline{f}) \leq |W| + 1 \leq d$ such that $\overline{F}(w) = 0$ for all $w \in W$ and $\overline{F}(\overline{s}) \neq 0$ for all $\overline{s} \in \overline{S}$. We claim that the morphism $\textnormal{Sym}(\overline{F}) = \sum_{g \in G} g^{-1} \circ \overline{F} \circ g$ gives the desired $\overline{k}$-rational point. Indeed, for each $\overline{t} \in \overline{T}$, we have
	\[
	\textnormal{Sym}(\overline{F}) (\overline{t}) = \sum_{g \in G}g^{-1}(\overline{F}(g(\overline{t}))) = 0
	\]
	because $\overline{F}$ vanishes on each $g(\overline{t})$ by construction. Moreover, for each $\overline{s} \in \overline{S}$ we have
	\begin{equation} \label{equation: symmetric polynomial in points of S}
	\textnormal{Sym}(\overline{F}) (\overline{s}) = \sum_{g \in G}g^{-1}(\overline{F}(g(\overline{s}))) = \sum_{g \in S_{\overline{s}}} g^{-1}(\overline{F}(g(\overline{s})))
	\end{equation}
	because $\overline{F}$ vanishes on each $g(\overline{s})$ which is different from $\overline{s}$ by construction. Then note that, if $\overline{s}$ lies over $s \in S$, then $S_{\overline{s}} = G_{\overline{s}} = G_s$ because $\overline{k}$ is algebraically closed. Since by hypothesis $\chi_{|G_{s}}$ is trivial, by Equation \ref{equation: symmetric polynomial in points of S} we get
	\[
	\textnormal{Sym}(\overline{F}) (\overline{s}) = \sum_{g \in G_s} \overline{F}(\overline{s}) = |G_s| \cdot \overline{F}(\overline{s}),
	\]
	which is different from $0$ because $|G_s|$ is invertible in $k$ and $\overline{F}(\overline{s}) \neq 0$ by construction.
\end{proof}

The next lemma will allow us to produce equivariant maps that are quasi-finite at a fixed point in Proposition \ref{proposition: quasi-finite morphisms construction}. Note that, even if the previous lemmas work for any finite abelian group $G$, our proof of Lemma \ref{lemma: decrease dimension by one} and of Proposition \ref{proposition: quasi-finite morphisms construction} holds only for cyclic $p$-groups. This is the reason why Theorem \ref{theorem: gabber} holds only for such groups.

\begin{lemma} \label{lemma: decrease dimension by one}
Let $k$ be infinite and let $p$ be a prime number coprime with $\textnormal{char}(k)$. Let $r \geq 1$ be an integer and set $G=C_{p^r}$. Let $X$ be an affine finite type $k$-scheme with a $G$-action and let $Z \subset X$ be a closed subscheme. Let $x \in Z$ be a closed point. Let $H \subseteq G$ be a subgroup and suppose that $\dim_x (Z \cap X^{H'})=0$ for all $H' \supsetneq H$. (By convention, $\dim_x (Z \cap X^{H'}) = 0$ if $x \notin X^{H'}$). Let $V \in \textnormal{Rep}_k(G)$ be a one-dimensional $G$-representation such that $\textnormal{Res}^G_H(V)$ is the trivial one-dimensional $H$-representation. Then there exists a $G$-equivariant morphism $f \colon X \to \AA(V)$, such that
\begin{equation} \label{equation: dimension counting}
\dim_x(f^{-1}(f(x)) \cap Z \cap X^K) \leq \dim_x (Z \cap X^K) -1
\end{equation}
for all subgroups $K \subseteq G$, where the right-hand side is understood as $\max \{ \dim_x (Z \cap X^K) -1,0 \}$.
\end{lemma}
\begin{proof}
Note that for a subgroup $K \subseteq G$, the dimension $\dim_x (Z \cap X^K)$ can be positive only if $K \subseteq G_x$, otherwise $x \notin X^K$. Let
\[
\cW \coloneq \{ W \mid W \textnormal{ is a positive-dimensional irreducible component of } Z \cap X^{K} \textnormal{ for some } K \subseteq G_x \}.
\]

For every $W \in \cW$, choose one closed point $s \in W$ such that 
\begin{itemize}
	\item $s \notin Gx$ and 
	\item $s \notin X^{H'}$ for all $H' \supsetneq H$.
\end{itemize}
This is possible because $k$ is infinite and because, by hypothesis, $\dim_x (Z \cap X^{H'})=0$ for all $H' \supsetneq H$. Let $S$ be the resulting finite set.

\textbf{Case when $V$ is the trivial $G$-representation}: in that case, Lemma \ref{lemma: generalization of Bachmann 2.5.(1)} yields a $G$-equivariant morphism
$f \colon X\to\AA(V) \simeq \AA^1_k$ such that $f(x)\notin f(S)$.
Hence $f^{-1}(f(x))$ cannot contain any $W \in \cW$. In particular we have that for all $K \subseteq G$ the inequality
\[
\dim_{x}(f^{-1}(f(x)) \cap Z \cap X^K) \leq \dim_{x}(Z \cap X^K) -1
\]
holds, where the right-hand side is understood as $\max \{ \dim_x (Z \cap X^K) -1,0 \}$.

\textbf{Case when $V$ is a nontrivial $G$-representation}: by construction, every $s \in S$ satisfies $G_s \subseteq H$, because $G = C_{p^r}$, and thus there exists a unique chain of subgroups of $G$. Since $\textnormal{Res}^G_H(V)$ is the trivial one-dimensional $H$-representation, then $V$ also restricts to a trivial $G_s$-representation for all $s \in S$. Then Lemma \ref{lemma: generalization of Bachmann 2.5.(2,3)} applied to $S$ and $T = \left\{x\right\}$ gives a $G$-equivariant morphism $f \colon X \to \AA(V)$ with $f(S) \neq 0$ for all $s \in S$ and $f(x)=0$. If any $W \in \cW$ were contained in
$f^{-1}(f(x))$, then the chosen point $s \in S \cap W$ would satisfy
$f(s)=f(x)=0$, contradicting $f(s)\neq 0$. In other words, we have a $G$-equivariant morphism $f \colon X \rightarrow \AA(V)$ such that for all $K \subseteq G$ the inequality
\[
\dim_{x}(f^{-1}(f(x)) \cap Z \cap X^K) \leq \dim_{x}(Z \cap X^K)-1
\]
holds, where the right-hand side is understood as $\max \{ \dim_x(Z \cap X^K) -1,0 \}$.
\end{proof}

\begin{proposition} \label{proposition: quasi-finite morphisms construction}
Let $k$ be infinite. Let $p$ be a prime and $r \geq 1$ be an integer such that $k$ contains a primitive $p^r$th root of unity. Set $G\coloneq C_{p^r}$. Let $(X,Z)$ be a $G$-pair with $X$ affine and let $x \in Z$ be a closed point. Let $V$ be a $G$-representation over $k$ such that there is an isomorphism $\textnormal{Res}^G_{G_x}(V) \otimes_{k} \kappa(x) \simeq T_{x}X$ in $\textnormal{Rep}_{\kappa(x)}(G_x)$. Set $\dim_{x}(X) = n+1$ for some $n \geq 0$. Then
\begin{enumerate}
	\item there exists a decomposition $V = V' \oplus V''$ of $G$-representations and a $G$-equivariant morphism $F \colon X \to \AA(V')$ such that $V'$ is a $G$-subrepresentation of dimension $n$ and the composite $Z \rightarrow X \xrightarrow{F} \AA(V')$ is quasi-finite at $x$.
	\item Moreover, if $T_{x}X$ is a nontrivial $G_x$-representation over $\kappa(x)$, the decomposition can be chosen so that $\textnormal{Res}^G_{G_x}(V'')$ is nontrivial.
\end{enumerate}
\end{proposition}
\begin{proof}
	Set $H_1 \coloneq G_x$.	Since $X$ is a smooth $k$-scheme and $G$ is a finite abelian group such that $\textnormal{char}(k) \nmid |G|$, for every $K \subseteq H_1$ we have
	\begin{equation} \label{equation: dimension fixed locus}
		\dim_x(X^K) = \dim_{\kappa(x)}(T_xX)^K.
	\end{equation}
	Moreover, since $X$ is smooth over $k$, $\dim_{\kappa(x)}T_x X = \dim_k V = n+1$. Since $Z$ has positive codimension in $X$, we have $\dim_x(Z) \leq n$.
	In order to prove (1) we need to show that there exists a $G$-subrepresentation $V' \subset V$ of dimension $n$ and a $G$-equivariant morphism $F\colon X \to \AA(V')$ such that the fiber $F^{-1}(F(x)) \cap Z$
	is zero-dimensional at $x$, because this means that $x$ is isolated in its fiber \cite[Lemma 12.72]{GortzWedhorn1} and hence that the composite map $Z \rightarrow X \xrightarrow{F} \AA(V')$ is quasi-finite at $x$.
	
	Suppose first that $n=0$. In that case we have $\dim_x(Z) = 0$. We set $V'' \coloneq V$ and $V' \coloneq 0$. Then $\AA(V') \simeq \Spec(k)$ and the map $Z \to X \to \Spec(k)$ is necessarily quasi-finite at $x$ because $\dim_x(Z) = 0$. Moreover, if $T_xX$ is a nontrivial $G_x$-representation over $\kappa(x)$, then $\textnormal{Res}^G_{G_x}(V)$ is a nontrivial $G_x$-representation over $k$ because $\textnormal{Res}^G_{G_x}(V) \otimes_{k} \kappa(x) \simeq T_{x}X$ by hypothesis; thus $\textnormal{Res}^G_{G_x}(V'') = \textnormal{Res}^G_{G_x}(V)$ is nontrivial, giving (2). 
	
	From now on we assume $n > 0$. We will construct the morphism $F$ by providing $n$ morphisms from $X$ to one-dimensional $G$-representations using Lemma \ref{lemma: decrease dimension by one}, with the property that each morphism decreases the dimension of the fiber $F^{-1}(F(x)) \cap Z$ at $x$ by one. 
	
	We first decompose $V$ as a direct sum of one-dimensional $G$-representations as
	\begin{equation}
		V = \bigoplus_{i = 1}^N V_{i}^{n_i}, \quad n_i \geq 0
	\end{equation}
	where the $V_i$'s are pairwise non-isomorphic one-dimensional $G$-representations, each appearing with multiplicity $n_i$ in the decomposition of $V$. Analogously, we decompose the $H_1$-representation $T_xX$ as a direct sum of one-dimensional $H_1$-representations over $\kappa(x)$ as
	\begin{equation}
		T_{x}X = \bigoplus_{j= 1}^M W_j^{m_j}, \quad m_j \geq 0
	\end{equation}
	where the $W_j$'s are pairwise non-isomorphic one-dimensional $H_1$-representations, each appearing with multiplicity $m_j$ in the decomposition of $T_xX$. Without loss of generality, we assume that $V_1$ (resp.\ $W_1$) is the trivial one-dimensional $G$-representation over $k$ (resp.\ $H_1$-representation over $\kappa(x)$).
	Note that
	\begin{equation} \label{equation: index sum}
		m_j = \sum_{(V_i)_{|H_1} \otimes_k \kappa(x) \simeq W_j} n_i
	\end{equation} 
	because $\textnormal{Res}^G_{H_1}(V) \otimes_k \kappa(x) \simeq T_{x}X$ in $\textnormal{Rep}_{\kappa(x)}(H_1)$ by hypothesis.
	
	 \textbf{Step 1}: let $I_1 \subseteq \{1, \ldots , N\}$ be the set of indexes such that $(V_i)_{|H_1}$ is a trivial $H_1$-representation. Without loss of generality we assume $I_1 = \{1, 2, \ldots, N_1 \}$. We have $m_1 = \sum_{i \in I_1} n_i$ by Equation \ref{equation: index sum}.
	 \begin{itemize}
	 	\item Case $m_1 = n+1$: in that case we have $\dim_x(X^{H_1}) = m_1 = n+1 = \dim_x(X)$ (the first equality is by Equation \ref{equation: dimension fixed locus}). It follows that the irreducible component of $X$ containing $x$ has trivial $H_1$-action. Then $\dim_x(Z^{H_1}) = \dim_x Z \leq \dim_x X -1 = \dim_x(X^{H_1}) - 1 = n$. Since $T_xX$ is a trivial $H_1$-representation over $\kappa(x)$, there is nothing to prove in (2). In order to prove (1), we do the following.
	 	If $n_1 > 0$, we apply Lemma \ref{lemma: decrease dimension by one} with respect to the subgroup $H_1$, to the closed subscheme $Z$ and to the trivial one-dimensional $G$-representation $V_1$ to obtain a $G$-equivariant morphism
	 		$f_1 \colon X\to\AA(V_1)$ such that
	 		\[
	 		\dim_{x}(f_1^{-1}(f_1(x)) \cap Z) \leq \dim_{x}Z -1,
	 		\]
	 		where we applied Equation \ref{equation: dimension counting} to the trivial subgroup $K= \{e\}$.
	 		We repeat the same process $d_1 \coloneq \min \{n_1, n\}$ times by applying Lemma \ref{lemma: decrease dimension by one} with respect to the subgroup $H_1$, to the trivial one-dimensional $G$-representation $V_1$ and to the relevant closed subscheme (for example, the second time will be applied to $f_1^{-1}(f_1(x)) \cap Z$) to get a $G$-equivariant morphism $F_1 \coloneq (f_1, \ldots, f_{d_1})\colon X \rightarrow \AA(V_1)^{d_1}$, such that 
	 		\[
	 		\dim_{x}(F_1^{-1}(F_1(x)) \cap Z) \leq \dim_{x}Z -d_1
	 		\]
	 		where the right-hand side is understood as $\max \{ \dim_x Z -d_1 ,0 \}$.
	 		If $d_1 = n$, then $\dim_x (Z) - d_1 \leq 0$ and thus the claim is proved by setting $V' \coloneq V_1^{n}$ and $F \coloneq F_1$.
	 		Suppose now that $d_1 = n_1 < n$. Since $n + 1 = m_1 = \sum_{i \in I_1} n_i$, then there exists some index $i \in I_1, i \neq 1$ such that $n_i > 0$. Without loss of generality we assume that $n_2 > 0$. We apply Lemma \ref{lemma: decrease dimension by one} with respect to the subgroup $H_1$, to the closed subscheme  $F_1^{-1}(F_1(x)) \cap Z$ and to the $G$-representation $V_2$, to obtain a $G$-equivariant morphism $f'_1 \colon X \rightarrow \AA(V_2)$ such that
	 		\[
	 		\dim_{x}((F_1, f_1')^{-1}((F_1, f'_1) (x)) \cap Z) \leq \dim_{x}Z- n_1 -1,
	 		\]
	 		where we applied Equation \ref{equation: dimension counting} to the trivial subgroup $K= \{e\}$ and we used that $(F_1, f_1')^{-1}((F_1, f'_1) (x)) \subseteq (f_1')^{-1}(f_1'(x)) \cap F_1^{-1}(F_1(x))$.
	 		We repeat the same process $d_2 \coloneq \min\{n_2, n - n_1 \}$ times by applying Lemma \ref{lemma: decrease dimension by one} with respect to the subgroup $H_1$, to the one-dimensional $G$-representation $V_2$ and to the relevant closed subscheme to get a $G$-equivariant morphism $(F_1, F_2) \coloneq (f_1, \ldots, f_{n_1}, f_1', \ldots, f_{d_2}') \colon X \rightarrow \AA(V_1)^{n_1} \oplus \AA(V_2)^{d_2}$, such that 
	 		\[
	 		\dim_{x}((F_1,F_2)^{-1}((F_1,F_2)(x)) \cap Z) \leq \dim_{x}Z -n_1 -d_2
	 		\]
	 		where the right-hand side is understood as $\max \{ \dim_x Z - n_1 -d_2 ,0 \}$.
	 		If $d_2 = n - n_1$, then $\dim_x Z - n_1 -d_2 \leq 0$ and thus we are done by setting $V' \coloneq V_1^{n_1} \oplus V_2^{d_2}$ and $F \coloneq (F_1, F_2)$. Otherwise, $d_2 = n_2 < n - n_1$. As before, in that case there exists some index $i \in I_1$, $i \neq 1,2$ such that $n_i > 0$. Repeating the same process for all the indexes of $I_1$ gives the desired conclusion.
	 		
	 \item Case $m_1 < n + 1$: in that case we claim that there exists a $G$-equivariant morphism $\alpha_1 = (F_1, F_2, \ldots, F_{N_1}) \colon X \to \AA(V_1)^{n_1} \oplus \ldots \oplus \AA(V_{N_1})^{n_{N_1}}$ whose restriction to $Z^{H_1}$ is quasi-finite at $x$. If $n_1 > 0$, we first apply Lemma \ref{lemma: decrease dimension by one} with respect to the subgroup $H_1$, to the closed subscheme $Z$ and to the trivial one-dimensional $G$-representation $V_1$ to obtain a $G$-equivariant morphism
	 $f_1 \colon X\to\AA(V_1)$ such that for all $K \subseteq G$ the inequality
	 \[
	 \dim_{x}(f_1^{-1}(f_1(x)) \cap Z \cap X^K) \leq \dim_{x}(Z \cap X^K) -1
	 \]
	 holds. We further apply Lemma \ref{lemma: decrease dimension by one} with respect to the subgroup $H_1$, to the $G$-representation $V_1$ and to the relevant closed subscheme (for example, the second time will be applied to $f_1^{-1}(f_1(x)) \cap Z$) a total of $n_1$ times and we get a $G$-invariant morphism $F_1 \coloneq (f_1, \ldots, f_{n_1})\colon X \rightarrow \AA(V_{1})^{n_1}$ such that 
	 \[
	 \dim_{x}(F_1^{-1}(F_1(x)) \cap Z \cap X^K) \leq \dim_{x}(Z \cap X^K) -n_1
	 \]
	 for all subgroups $K \subseteq G$, where the right-hand side is understood as $\max \{ \dim_x (Z \cap X^K) -n_1 ,0 \}$. In particular, since $Z \cap X^{H_1} = Z^{H_1}$,
	 \[
	 	\dim_{x}(F_1^{-1}(F_1(x)) \cap Z^{H_1}) \leq \dim_x(Z^{H_1}) -n_1 \leq \dim_x(X^{H_1}) -n_1 = m_1 - n_1.
	 \]
	 As before, we continue this process by applying Lemma \ref{lemma: decrease dimension by one} to the relevant closed subscheme, the subgroup $H_1$ and to the one-dimensional $G$-representation $V_i$, where $i \in I_1$ is an index such that $n_i > 0$. After $n_1 + n_2+ \ldots + n_{N_1}$ times we obtain a $G$-equivariant morphism $\alpha_1 = (F_1, F_2, \ldots, F_{N_1}) \colon X \to \AA(V_1)^{n_1} \oplus \ldots \oplus \AA(V_{N})^{n_{N_1}}$ such that
	 \begin{equation} \label{equation: dimension 0 case one}
	 \dim_x(\alpha_1^{-1}(\alpha_1(x)) \cap Z \cap X^K) \leq \dim_x(Z \cap X^K) - \sum_{i \in I_1} n_i = \dim_x(Z \cap X^K) - m_1
	 \end{equation}
	 for all $K \subseteq G$. In particular, for $K=H_1$,
	 \[
	 \dim_x(\alpha_1^{-1}(\alpha_1(x)) \cap Z^{H_1}) \leq \dim_x(Z^{H_1}) - m_1 \leq \dim_x(X^{H_1}) -m_1 = 0,
	 \]
	 as desired.
 
     If $m_1 = n$, then Equation \ref{equation: dimension 0 case one} applied to $K = \{e\}$ shows that $\dim_x(\alpha_1^{-1}(\alpha_1(x)) \cap Z) \leq \dim_x(Z) - n \leq 0$. In that case, we set $V' \coloneq V_{1}^{n_1} \oplus \ldots \oplus V_{N_1}^{n_{N_1}}$ and $F \coloneq \alpha_1$ to get the result. Note indeed that if we decompose $V = V' \oplus V''$, the remaining one-dimensional $G$-representation $V''$ is such that $\textnormal{Res}^G_{G_x}(V'')$ is non-trivial by construction. If $m_1 < n$ we proceed in Step 2.
	 \end{itemize}
	 
	 \textbf{Step 2}: suppose that $m_1 <n$. Let $H_2 \subsetneq H_1$ be the unique maximal proper subgroup of $H_1$. Note that $H_2$ is indeed unique because $G$ is a cyclic $p$-group. Let $I_2 = \{ N_1  +1, \ldots , N_1 + N_2 \}\subseteq \{1, \ldots, N \}$ be the set of indexes $i$ such that $(V_i)_{|H_2}$ is a trivial $H_2$-representation, but $(V_i)_{|H_1}$ is non-trivial. By Equation \ref{equation: dimension 0 case one}, using that $Z \cap X^{H_2} = Z^{H_2}$, we have
	 \[
	 \dim_{x}(\alpha_1^{-1}(\alpha_1(x)) \cap Z^{H_2}) \leq \dim_{x}(Z^{H_2}) - m_1 \leq \dim_x(X^{H_2}) -m_1 = \sum_{i \in I_2} n_i \eqcolon c_2.
	 \]
	 \begin{itemize}
	 	\item Suppose that $c_2 = 0$. In that case the morphism $\alpha_1$ is such that the restriction to $Z^{H_2}$ is quasi-finite at $x$ and we pass to Step 3.
	 	\item Suppose $c_2 \neq 0$. In that case, since $\dim_x(\alpha_1^{-1}(\alpha_1(x)) \cap Z^{H_1}) = 0$, we can apply Lemma \ref{lemma: decrease dimension by one} with respect to the subgroup $H_2$, to the closed subscheme $\alpha_1^{-1}(\alpha_1(x)) \cap Z^{H_2}$ and to a one-dimensional $G$-representation $V_i$ for some $i \in I_2$ with $n_i > 0$, in order to obtain a $G$-equivariant morphism $f_1' \colon X \to \AA(V_i)$ such that \[
	 	\dim_{x}((\alpha_1, f_1')^{-1}((\alpha_1, f'_1) (x)) \cap Z \cap X^K) \leq \dim_{x}(Z \cap X^K)- m_1 -1
	 	\]
	 	for all $K \subseteq G$. We repeat the same argument $d_3 \coloneq \min \{c_2 , n -m_1 \}$ times to obtain a $G$-equivariant morphism $\alpha_2 \coloneq (\alpha_1, f'_1, \ldots, f'_{d_3})\colon  X \rightarrow \bigoplus_{i \in I_1} \AA(V_i)^{n_i} \oplus \AA^{d_3}_k$ (where on each coordinate of the affine space $\AA^{d_3}_k$ we have the linear $G$-action given by some index appearing in $I_2$), such that for all $K \subseteq G$ we have
	 	\begin{equation} \label{equation: dimension 0 second case}
	 	\dim_{x}((\alpha_2)^{-1}(\alpha_2(x)) \cap Z \cap X^K) \leq \dim_{x}(Z \cap X^K) -m_1 - d_3.
	 	\end{equation}
	 
	 	If $d_3 = n -m_1$, then Equation \ref{equation: dimension 0 second case} applied to $K= \{e\}$ reads
	 	\[
	 	\dim_{x}((\alpha_2)^{-1}(\alpha_2(x)) \cap Z) \leq \dim_{x}Z -m_1 - n + m_1 = \dim_{x}Z - n  \leq 0,
	 	\]
	 	so that we are done by setting $F \coloneq \alpha_2$. The $G$-subrepresentation $V'$ is given by $V' \coloneq V_{1}^{n_1} \oplus \ldots \oplus V_{N_1}^{n_{N_1}} \oplus \tilde{V}$ where $\tilde{V}$ is the direct sum of the $d_3 = n-m_1$ representations used to construct the morphism $\alpha_2$. By construction, the remaining one-dimensional sub $G$-representation $V'' \subseteq V$ is such that $\textnormal{Res}^G_{G_x}(V'')$ is nontrivial, giving (2). 
	 	
	 	If $d_3 = c_2 < n -m_1$, then $m_1 + c_2 < n$. By Equation \ref{equation: dimension 0 second case}, we find in particular that
	 	\[
	 	\dim_{x}((\alpha_2)^{-1}(\alpha_2(x)) \cap Z^{H_2}) \leq \dim_x Z^{H_2} -m_1 -c_2  \leq 0,
	 	\]
	 	and thus we constructed a morphism $\alpha_2$ whose restriction to $Z^{H_1}$ and $Z^{H_2}$ is quasi-finite at $x$. In that case we proceed in Step 3.
	 \end{itemize}
	 
	 \textbf{Step 3}: recall that $m_1 + c_2 < n$ and that we have found a morphism $\alpha_2$ whose restriction to $Z^{H_1}$ and $Z^{H_2}$ is quasi-finite at $x$. Let $H_3 \subsetneq H_2$ be the unique maximal proper subgroup of $H_2$. Let $I_3 = \{ N_1 + N_2  +1, \ldots , N_1 + N_2+ N_3 \}\subseteq \{1, \ldots, N \}$ be the set of indexes $i$ such that $(V_i)_{|H_3}$ is a trivial $H_3$-representation, but $(V_i)_{|H_2}$ is non-trivial. By Equation \ref{equation: dimension 0 second case} we see that
	 \[
	 \dim_{x}(\alpha_2^{-1}(\alpha_2(x)) \cap Z^{H_3}) \leq \dim_{x}(Z^{H_3}) - m_1 - c_2 \leq \dim_{x}(X^{H_3}) - m_1 - c_2  = \sum_{i \in I_3} n_i \eqcolon c_3.
	 \]
	 \begin{itemize}
	 	\item Suppose that $c_3 = 0$. In that case the morphism $\alpha_2$ is such that the restriction to $Z^{H_3}$ is quasi-finite at $x$ and we pass to Step 4.
	 	\item Suppose $c_3 \neq 0$. In that case, since $\dim_x(\alpha_2^{-1}(\alpha_2(x)) \cap Z^{H_1}) = \dim_x(\alpha_2^{-1}(\alpha_2(x)) \cap Z^{H_2}) = 0$, we can apply Lemma \ref{lemma: decrease dimension by one} with respect to the subgroup $H_3$, to the closed subscheme $\alpha_2^{-1}(\alpha_2(x)) \cap Z^{H_3}$ and to a one-dimensional $G$-representation $V_i$ for some $i \in I_3$, in order to obtain a $G$-equivariant morphism $f_1'' \colon X \to \AA(V_i)$ such that \[
	 	\dim_{x}((\alpha_2, f_1'')^{-1}((\alpha_2, f''_1) (x)) \cap Z \cap X^K) \leq \dim_{x}(Z \cap X^K)- m_1 - c_2 -1
	 	\]
	 	for all $K \subseteq G$. We repeat the same argument $d_4 \coloneq \min \{c_3 , n -m_1 - c_2 \}$ times to obtain a $G$-equivariant morphism $\alpha_3 \coloneq (\alpha_2, f''_1, \ldots, f''_{d_4}) :  X \rightarrow \bigoplus_{i \in I_1} \AA(V_i)^{n_i} \oplus \AA^{c_2}_k \oplus \AA^{d_4}_k$ (where on each coordinate of the affine space $\AA^{d_4}_k$ we have the linear $G$-action given by some index appearing in $I_3$), such that for all $K \subseteq G$ we have
	 	\begin{equation} \label{equation: dimension 0 third case}
	 		\dim_{x}((\alpha_3)^{-1}(\alpha_3(x)) \cap Z \cap X^K) \leq \dim_{x}(Z \cap X^K) -m_1 - c_2 - d_4.
	 	\end{equation}

	   If $d_4 = n -m_1 - c_2$, then Equation \ref{equation: dimension 0 third case} applied to $K= \{e\}$ reads
	 		\[
	 		\dim_{x}((\alpha_3)^{-1}(\alpha_3(x)) \cap Z) \leq \dim_{x}Z -m_1 - c_2 - n + m_1 + c_2 = \dim_{x}Z - n  \leq 0,
	 		\]
	 		so that we are done by setting $F \coloneq \alpha_3$. The $G$-subrepresentation $V'$ is given by $V' \coloneq V_{1}^{n_1} \oplus \ldots \oplus V_{N_1}^{n_{N_1}} \oplus \tilde{V}$ where $\tilde{V}$ is the direct sum of the $c_2 + d_4 = n -m_1$ representations used to construct the morphisms $\alpha_2$ and $\alpha_3$. By construction, the remaining one-dimensional $G$-subrepresentation $V'' \subseteq V$ is such that $\textnormal{Res}^G_{G_x}(V'')$ is nontrivial, giving (2). 
	 		
	   If $d_4 = c_3 < n -m_1 - c_2$, then $m_1 + c_2 + c_3 < n$. By Equation \ref{equation: dimension 0 third case}, we find in particular that
	 		\[
	 		\dim_{x}((\alpha_3)^{-1}(\alpha_3(x)) \cap Z^{H_3}) \leq \dim_x (Z^{H_3}) -m_1 -c_2 - c_3  \leq 0,
	 		\]
	 		and thus we constructed a morphism $\alpha_3$ whose restriction to $Z^{H_1}$, $Z^{H_2}$ and $Z^{H_3}$ is quasi-finite at $x$. In that case we proceed as in Step 4.
	 \end{itemize}
	 
	 \textbf{Step 4}: recall that $m_1 + c_2 + c_3 < n$ and that we have found a morphism $\alpha_3$ whose restriction to $Z^{H_1}$, $Z^{H_2}$ and $Z^{H_3}$ is quasi-finite at $x$. We can now repeat Step 3 with respect to the unique maximal subgroup $H_4 \subsetneq H_3$. Then we repeat the same process for all the subgroups appearing in the unique chain of subgroups of $H_1$. As noted in Step 2 and Step 3, we just need to use all but one of the character summands $V_i$, counted with multiplicity, to reduce the intersection dimension to $0$ (cases $d_3 = n-m_1$ and $d_4 = n-m_1-c_2$). Indeed, after $n$ steps we will have a $G$-equivariant morphism $F \colon X \rightarrow \AA^{n}_k$
	 such that
	 \[
	 \dim_{x}(F^{-1}(F(x)) \cap Z) \leq \dim_{x} Z - n \leq 0.
	 \]
	 The $n$ one-dimensional $G$-representations $V_i$ (counted with multiplicity) used in the construction of the map $F$ give the required $n$-dimensional $G$-subrepresentation $V'$ of $V$, while the remaining unused one-dimensional $G$-representation gives the one-dimensional $G$-subrepresentation $V''$. By construction, if $T_x X$ is a nontrivial $G_x$-representation over $\kappa(x)$, $\textnormal{Res}^G_{G_x}(V'')$ is nontrivial.
\end{proof}

\begin{remark} \label{remark: why not possible to produce quasifinite equivariant morphisms}
It is important in the proof of Proposition \ref{proposition: quasi-finite morphisms construction} that $G=C_{p^r}$ because we use the fact that there exists a unique chain of subgroups in $G=C_{p^r}$. The failure of this property leads to the failure of the $G$-equivariant Gabber's lemma for finite abelian groups $G$ which are not cyclic $p$-groups. See also Example \ref{example: failure quasi-finiteness} and Theorem \ref{theorem: counterexample}.
\end{remark}

\begin{example}
Let $k = \CC$ and $G = C_4 = \langle \sigma \rangle$. Let $X = \Spec (\CC[x_1, x_2])$ with $G$-action given by $\sigma(x_1) = i x_1$ and $\sigma(x_2) = i x_2$. Let $Z= V(x_1x_2)$ be the closed $G$-invariant subscheme given by the union of the coordinate axes and let $x =(0,0)$ be the origin. Then $G_x = G$. Let $V  = V' \oplus V'' \simeq \CC \oplus \CC$ with $\sigma$ acting as multiplication by $i$ on each coordinate. Let $\AA^1_{\CC}$ be the one-dimensional $G$-representation with linear action given by multiplication by $i$. Then Proposition \ref{proposition: quasi-finite morphisms construction} gives the $G$-equivariant morphism $F \colon X \to \AA^1_{\CC}$ corresponding to $F(x_1, x_2)  = x_1 + x_2$. Indeed, we just need to apply Lemma \ref{lemma: decrease dimension by one} once, and the proof of the latter shows that it is sufficient to symmetrize the choice of a polynomial $f'$ which vanishes on $x= (0,0)$ but does not vanish on $x' = (1,0)$, which is another point of $Z$ different from $x$. Then we may choose $f' = x_1 + x_2$, and the symmetrization of $f'$ is precisely $F(x_1, x_2) = x_1 + x_2$. Then we see that $F^{-1}(F(x))= F^{-1}(0) = V(x_1 + x_2)$ and that $V(x_1+x_2) \cap V(x_1 x_2)$ has dimension zero.
\end{example}

\begin{example} \label{example: failure quasi-finiteness}
In this non-example we show the failure of Proposition \ref{proposition: quasi-finite morphisms construction} in the case $G$ is not a cyclic $p$-group.
Let $k= \CC$ and $G= C_2 \times C_2 = \langle \sigma_1 \rangle \times \langle \sigma_2 \rangle $. Let $(X,Z)$ be the $G$-pair of Proposition \ref{proposition: counterexample} and let $x=(0,0)$ be the origin. In that case, $G_x = G$ and $T_x X \simeq \CC^2$ are isomorphic as $G$-representations. Then $V = \CC^2 = V' \oplus V''$ where $V'$ is the one-dimensional $G$-representation with $\sigma_1$ acting as sign action and $\sigma_2$ acting trivially, while $V''$ is the one-dimensional $G$-representation with $\sigma_2$ acting as sign action and $\sigma_1$ acting trivially. Any $G$-equivariant morphism $F\colon X \to \AA(V')$ must send the whole $V(x_1) \subseteq Z$ into the origin, because $V(x_1)$ is fixed pointwise by $\sigma_1$ and $\AA(V')^{\langle \sigma_1 \rangle} = 0$. It follows that the composite $Z \to X \xrightarrow{F} \AA(V')$ cannot be quasi-finite at $x$. The same argument applies in order to show that there cannot exist a $G$-equivariant morphism $F \colon X \to \AA(V'')$ such that $Z \to X \to \AA(V'')$ is quasi-finite at $x$. Hence a decomposition of $V$ satisfying in Proposition \ref{proposition: quasi-finite morphisms construction}.(1) cannot exist.
\end{example}

If the scheme $X$ carries a free $G$-action, we can choose any $G$-representation over $k$ of the right dimension, as the following proposition shows.

\begin{proposition} \label{proposition: quasi-finite morphisms construction freecase}
Let $k$ be infinite. Let $p$ be a prime and $r \geq 1$ be an integer such that $k$ contains a primitive $p^r$th root of unity. Set $G\coloneq C_{p^r}$. Let $(X,Z)$ be a $G$-pair where $X$ is affine with free $G$-action and let $x \in Z$ be a closed point. Set $\dim_x (X) = n+1$ for some $n \geq 0$. For any $n$-dimensional $G$-representation $V$ over $k$, there exists a $G$-equivariant morphism $F \colon X \to \AA(V)$ whose restriction to $Z$ is quasi-finite at $x$.
\end{proposition}
\begin{proof}
We need to show that there exists a $G$-equivariant morphism $F\colon X \to \AA(V)$ such that the fiber $F^{-1}(F(x)) \cap Z$
is zero-dimensional at $x$, because this means that $x$ is isolated in its fiber \cite[Lemma 12.72]{GortzWedhorn1} and hence that the restriction of $F$ to $Z$ is quasi-finite at $x$.

Note that $\dim_{x}(Z) \leq \dim_x(X) -1 = n= \dim_k V$. If $n=0$ the result is trivial. Assume $n \geq 1$. Write $V= \oplus_{i=1}^n V_i$ as direct sum of one-dimensional $G$-representations. Since $X$ has free $G$-action, $\dim_x(Z \cap X^{H'}) =0$ for all subgroups $\{e\}\subsetneq H'$. Then we can apply Lemma \ref{lemma: decrease dimension by one} with respect to the trivial subgroup $\left\{e\right\}$, the $G$-representation $V_1$ and the closed subscheme $Z$ to obtain a $G$-equivariant morphism $f_1 \colon X \to \AA(V_1)$ such that $\dim_{x}(f_1^{-1}(f_1(x)) \cap Z) \leq \dim_{x}(Z)-1$ (apply Equation \ref{equation: dimension counting} to $K = \{e\}$). We again apply Lemma \ref{lemma: decrease dimension by one} with respect to the trivial subgroup $\left\{e\right\}$, the $G$-representation $V_2$ and the closed subscheme $f_1^{-1}(f_1(x)) \cap Z$ to obtain a $G$-equivariant morphism $f_2\colon X \to \AA(V_2)$ such that $\dim_{x}((f_1,f_2)^{-1}((f_1,f_2)(x)) \cap Z) \leq \dim_{x}(Z)-2$. Repeating the same process for all $i=1, \ldots, n$ we finally get a $G$-equivariant morphism $F \colon X \to \AA(V)$ with the desired property.
\end{proof}

\section{Properties of equivariant morphisms} \label{section: improving morphisms}

In this section we study some properties of equivariant morphisms of schemes. In particular, in $\S$\ref{subsection: improving properties} we refine some properties of equivariant morphisms of schemes by replacing the source scheme with a sufficiently small invariant open subset; we will need these improvements in the proof of Theorem \ref{theorem: gabber}. In $\S$\ref{subsection: openness of properties} we show that the properties of morphisms under consideration define an open condition in the affine space of equivariant morphisms.

\subsection{Improving properties} \label{subsection: improving properties}

We will show how to pass from an equivariant quasi-finite morphism to an equivariant finite morphism; then we show how to pass from an equivariant finite morphism to an equivariant closed immersion; and finally how to pass from an equivariant closed immersion to an isomorphism.

We will use the following notation. Let $G$ be a finite abelian group and let $Y= \Spec(B)$ be an affine scheme with a $G$-action. Let $y \in Y$ be a point. Let $Gy$ be the set-theoretic orbit of $y$ and let $\frakp_1, \ldots,\frakp_r \subseteq B$ be the prime ideals of $B$ corresponding to the points $gy \in Gy$. Set $\cO_{Y,Gy} \coloneq S^{-1}B$, where $S$ denotes the multiplicatively closed subset of $B$ defined by $S = B \setminus (\frakp_1 \cup \ldots \cup \frakp_r)$. Note that $\cO_{Y,Gy}$ is a semi-local ring, whose maximal ideals are $\frakm_i \coloneq \frakp_i\cdot \cO_{Y,Gy}$ for $i=1, \ldots, r$. Then we define $\cO^h_{Gy}$ to be the henselization of the semi-local ring $\cO_{Y,Gy}$ along the ideal $I = \bigcap_{i=1}^r \frakm_i $, i.e.\ the henselization of the pair $(\cO_{Y,Gy}, I)$ as in \cite[\href{https://stacks.math.columbia.edu/tag/0EM7}{Tag 0EM7}]{stacks-project}.

\begin{lemma} \label{lemma: henselization}
Let $Y$ be an affine scheme with a $G$-action and let $y \in Y$ be a point. Then
\[
Y^h_{Gy} \coloneq \Spec(\cO^h_{Gy}) \simeq \coprod_{gy \in Gy} \Spec(\cO^h_{Y,gy}) = \coprod_{gy \in Gy} Y^h_{gy}.
\]
\end{lemma}
\begin{proof}
Write $A = \cO_{Y,Gy}$ for the semi-local ring associated to the $G$-orbit of $y$ and let $\frakm_1, \ldots, \frakm_r$ be its maximal ideals. Set $I \coloneq \frakm_1 \cap \ldots \cap \frakm_r$. By \cite[\href{https://stacks.math.columbia.edu/tag/0H7Q}{Tag 0H7Q}]{stacks-project}, there is an isomorphism $(A,I)^h \simeq \prod_{i=1}^r (A, \frakm_i)^h$. On the other hand, by \cite[\href{https://stacks.math.columbia.edu/tag/04GV}{Tag 04GV}]{stacks-project}, for each $i=1, \ldots, r$ we have that $(A, \frakm_i)^h = A_{\frakm_i}^h$. Then we see that $\cO_{Gy}^h \simeq \prod_{gy \in Gy} \cO_{gy}^h$ and the claim follows.
\end{proof}

We first show how to pass from an equivariant quasi-finite morphism to an equivariant finite morphism.

\begin{lemma} \label{lemma: pass from quasi-finite to finite}
Let $G$ be a finite abelian group and let $X,Y$ be affine $G$-schemes over $k$. Let $Z \subseteq X$ be a closed $G$-invariant subscheme. Let $y \in Y$ and let $f \colon X \to Y^h_{Gy}$ be a $G$-equivariant morphism. Suppose that $x \in Z$ satisfies $f(x)=y$ and that $f$ induces a bijection $f_{|Gx} \colon Gx \xrightarrow{\sim} Gy$ on the $G$-orbits of $x$ and $y$. Suppose further that the restriction $\beta \coloneq f_{|Z} \colon Z \to Y^h_{Gy}$ is quasi-finite. Then there exists an open $G$-invariant subscheme $U \subset X$ containing $Gx$ such that the following hold:

\begin{enumerate}
	\item  the induced morphism $\beta_U \colon Z_U \to U \to Y^{h}_{Gy}$ is finite;
	
	\item the scheme $Z_U$ decomposes as $Z_U = \coprod_{gx \in Gx} \Spec(A_{1,gx})$, where each $A_{1,gx}$ is a local ring whose unique closed point is $gx$, and $\beta_U$ is the disjoint union of $|Gx|$ local morphisms $\Spec(A_{1,gx}) \to Y_{gy}^h$.
	
	\item Suppose that $K \subseteq X$ is a closed subscheme with a decomposition $K = \coprod_{gx \in Gx} \Spec(C_{gx})$, where each $C_{gx}$ is a local ring whose unique closed point is $gx$. Then $K \subseteq U$.
\end{enumerate}
\end{lemma}
\begin{proof}
We first show (1) and (2). Since $X$ is affine, the closed subscheme $Z \subseteq X$ is affine. Write $Z = \Spec (S)$ for some $k$-algebra $S$. By Lemma \ref{lemma: henselization} we can write $Y^h_{Gy} = \coprod_{gy \in Gy} Y^h_{gy}$ .
Let $g \in G$ and let $Z_{gx}$ denote the fiber product		% https://q.uiver.app/#q=WzAsNCxbMSwwLCJaIl0sWzEsMSwiWV5oX3tHeX0iXSxbMCwxLCJZXmhfe2d5fSJdLFswLDAsIlpfZyJdLFszLDIsIlxcYmV0YV9nIiwyXSxbMiwxXSxbMCwxLCJcXGJldGEiXSxbMywwXV0=
\[\begin{tikzcd}
	{Z_{gx}} & Z \\
	{Y^h_{gy}} & {Y^h_{Gy}.}
	\arrow[from=1-1, to=1-2]
	\arrow["{\beta_g}"', from=1-1, to=2-1]
	\arrow["\beta", from=1-2, to=2-2]
	\arrow[from=2-1, to=2-2]
\end{tikzcd}\]
Note indeed that $gx \in Z_{gx}$ and $\beta_g(gx) = gy \in Y^h_{gy}$. Moreover, $gx$ is the only point of the orbit $Gx$ that belongs to $Z_{gx}$, because by hypothesis $f$ induces a bijection on the $G$-orbits of $x$ and $y$. Since $Z_{gx}$ is an affine scheme, we can write $Z_{gx} = \Spec (S_{gx})$ for some $k$-algebra $S_{gx}$. Then the morphism $\beta_g$ corresponds to a finite-type morphism of rings  $\cO_{Y,gy}^h \rightarrow S_{gx}$.
By \cite[\href{https://stacks.math.columbia.edu/tag/04GJ}{Tag 04GJ}]{stacks-project} we can write
\[
S_{gx}= A_{1,gx} \times \ldots \times A_{n_g,gx} \times B_{gx} \qquad n_g \geq 1
\]
with $A_{i,gx}$ local and finite over $\cO_{Y,gy}^h$ and $\cO^h_{Y,gy} \to B_{gx}$ not quasi-finite at any prime of $B_{gx}$ lying over the only maximal ideal of $\cO_{Y,gy}^h$. In particular, $gx \notin \Spec(B_{gx})$, because, by hypothesis, $\beta$ is quasi-finite, and thus $\beta_g$ is quasi-finite at $gx \in Z_{gx}$. Then  we may assume without loss of generality that $gx \in \Spec(A_{1,gx})$. Consider the composite morphism
\[
\Spec(A_{1,gx}) \longrightarrow  Z_{gx} \xlongrightarrow{\beta_g} Y^h_{gy}.
\]
We know that $gx \in \Spec(A_{1,gx})$ maps to $gy$, which is the unique closed point of $Y^h_{gy}$. This forces the unique closed point of $\Spec(A_{1,gx})$ to go to $gy$. Since $\cO^h_{Y,gy} \to A_{1,gx}$ is a finite ring map, it is also integral, and it follows by \cite[\href{https://stacks.math.columbia.edu/tag/00GT}{Tag 00GT}]{stacks-project} that $gx$ is indeed the unique closed point of $\Spec(A_{1,gx})$. For all $g \in G$ and $i \in \{ 1, \ldots, n_g\}$, the schemes $\Spec(A_{i,gx}), \Spec(B_{gx})$ are closed subschemes of $Z_{gx}$, and thus they are also closed subschemes of $X$. Set
\[
U \coloneq X \setminus \bigg (\bigcup_{g} \bigcup_{i > 1} \Spec(A_{i,gx}) \cup \Spec(B_{gx}) \bigg ).
\]
We claim that $U$ is a $G$-invariant open subscheme of $X$ and that $\beta_{U} \colon Z_U \to Y^h_{Gy}$ is finite.
To see that $U$ is a $G$-invariant open subscheme we do the following. Let $g \in G$. Note that $\beta \colon Z \to Y_{Gy}^h$ is $G$-equivariant, and thus for each $\tilde{g} \in G$ we have a commutative square
% https://q.uiver.app/#q=WzAsNCxbMCwwLCJaX3tcXHRpbGRle1h9XzJ9Il0sWzEsMCwiWl97XFx0aWxkZXtYfV8yfSJdLFswLDEsIldfe0d6fV5oIl0sWzEsMSwiV197R3p9XmgiXSxbMCwyXSxbMSwzXSxbMiwzLCJoIiwyXSxbMCwxLCJoIl1d
\[\begin{tikzcd}
	{Z} & {Z} \\
	{Y_{Gy}^h} & {Y_{Gy}^h.}
	\arrow["\tilde{g}", from=1-1, to=1-2]
	\arrow[from=1-1, to=2-1]
	\arrow[from=1-2, to=2-2]
	\arrow["\tilde{g}"', from=2-1, to=2-2]
\end{tikzcd}\]
Write $g' = \tilde{g}g \in G$. The action of $\tilde{g}$ on $Y_{Gy}^h$ maps $Y_{gy}^h$ to $Y_{g'y}^h$. It follows that there is also a commutative square
% https://q.uiver.app/#q=WzAsNCxbMSwwLCJaX3tcXHRpbGRle2d9Z30iXSxbMSwxLCJZXmhfe1xcdGlsZGV7Z31neX0iXSxbMCwxLCJZXmhfe2d5fSJdLFswLDAsIlpfZyJdLFszLDIsIlxcYmV0YV9nIiwyXSxbMiwxLCJcXHRpbGRle2d9IiwyXSxbMCwxLCJcXGJldGFfe1xcdGlsZGV7Z31nfSJdLFszLDAsIlxcdGlsZGV7Z30iXV0=
\[\begin{tikzcd}
	{Z_{gx}} & {Z_{g'x}} \\
	{Y^h_{gy}} & {Y^h_{g'y}.}
	\arrow["{\tilde{g}}", from=1-1, to=1-2]
	\arrow["{\beta_g}"', from=1-1, to=2-1]
	\arrow["{\beta_{g'}}", from=1-2, to=2-2]
	\arrow["{\tilde{g}}"', from=2-1, to=2-2]
\end{tikzcd}\]

The top horizontal arrow maps $gx \mapsto g'x$. Since these are precisely the unique closed points of $\Spec(A_{1,gx})$ and $\Spec(A_{1, g'x})$ respectively, this forces $\tilde{g}(\Spec(A_{1,gx})) = \Spec(A_{1,g'x})$, because open subsets are stable under generalization. Then the closed subscheme $\coprod_{gx \in Gx} \Spec(A_{1,gx})$ of $Z$ is $G$-invariant, and in particular its complement in $Z$ is $G$-invariant. It follows that $U$ is a $G$-invariant open subscheme of $X$. 

It remains to show that $\beta_U \colon Z_U \to U \to Y^h_{Gy}$ is a finite morphism. By construction we have $Z_U = Z \cap U \simeq \coprod_{gx \in Gx} \Spec(A_{1,gx})$, and each $A_{1,gx}$ is finite over $\cO_{Y,gy}^h$.

(3) Suppose now that $K = \coprod_{gx \in Gx} \Spec(C_{gx})$ is a closed subscheme of $X$, where the $C_{gx}$ are local rings as in the hypotheses. For $g' \in G$, consider the intersections of $\Spec(A_{i,gx})$ and $\Spec(B_{gx})$ with $\Spec(C_{g'x})$. If $g' \neq g$, then $\Spec(A_{i,gx})$ and $\Spec(B_{gx})$ lie over $Y^h_{gy}$, whereas the connected local scheme $\Spec(C_{g'x})$ lies over the open-and-closed component $Y^h_{g'y}$; hence both intersections are empty. If $g' = g$, each intersection is a closed subset of the local scheme $\Spec(C_{gx})$. Neither contains its unique closed point $gx$, because $gx \notin \Spec(A_{i,gx})$ for $i > 1$ and $gx \notin \Spec(B_{gx})$. Therefore, both intersections are empty in this case as well.
It follows that $K$ does not intersect the closed subschemes $\Spec(B_{gx})$ and $\Spec(A_{i,gx})$ for $i > 1$, and thus $K \subseteq U$ by definition of the open subscheme $U$.
\end{proof}

We now show how to pass from an equivariant finite morphism to an equivariant closed immersion. We will need the following.

\begin{lemma} \label{lemma: closed fiber of a finite local morphism}
Let $A$ and $B$ be local rings and let $B \to A$ be a finite local homomorphism. Let $\frakm$ and $\frakn$ be the maximal ideals of $A$ and $B$ respectively. Then the closed fiber $\Spec(A/\frakn A)$ is a local Artinian scheme whose underlying topological space consists of the closed point of $\Spec(A)$.
\end{lemma}
\begin{proof}
The quotient $A/\frakn A$ is finite-dimensional over $B/\frakn$, hence Artinian. It is local, with maximal ideal $\frakm/\frakn A$, because $A$ is local and $\frakn A \subseteq \frakm$. Thus its spectrum has a single point.
\end{proof}

\begin{lemma} \label{lemma: nakayama's lemma for closed immersion}
Let $A$ and $B$ be local $k$-algebras, and let $C$ be a $k$-algebra. Suppose there exists an unramified morphism $f \colon \Spec(A)\rightarrow \Spec(B\otimes_k C)$ such that the composite $\textnormal{pr}_B \circ f \colon \Spec(A)\xrightarrow{f}  \Spec(B \otimes_k C) \xrightarrow{\textnormal{pr}_B} \Spec(B)$ is finite and sends the unique closed point $x\in \Spec(A)$ to the unique closed point $y\in \Spec(B)$.
Let $\Spec(\kappa(y)\otimes_k C)$
be the fiber of $\textnormal{pr}_B \colon \Spec(B\otimes_k C)\to \Spec(B)$ over $y$. Assume that the induced morphism $h \colon \Spec(\kappa(y) \otimes_B A)\to \Spec(\kappa(y)\otimes_k C)$
is radicial. Then $f$ is a closed immersion.
\end{lemma}

\begin{proof}
Consider the following cartesian diagram:
% https://q.uiver.app/#q=WzAsNixbMCwwLCJcXFNwZWMoQS9cXGZyYWtuIEEpIl0sWzAsMSwiXFxTcGVjKFxca2FwcGEoeSkgXFxvdGltZXNfayBDKSJdLFswLDIsIlxcU3BlYyhcXGthcHBhKHkpKSJdLFsxLDIsIlxcU3BlYyhCKSJdLFsxLDEsIlxcU3BlYyhCIFxcb3RpbWVzX2sgQykiXSxbMSwwLCJcXFNwZWMoQSkiXSxbMCw1XSxbNSw0LCJmIl0sWzQsM10sWzIsMywieSIsMl0sWzEsMl0sWzAsMSwiaCIsMl0sWzEsNF1d
\[\begin{tikzcd}
	{\Spec(\kappa(y) \otimes_B A)} & {\Spec(A)} \\
	{\Spec(\kappa(y)\otimes_k C)} & {\Spec(B \otimes_k C)} \\
	{\Spec(\kappa(y))} & {\Spec(B).}
	\arrow[from=1-1, to=1-2]
	\arrow["h"', from=1-1, to=2-1]
	\arrow["f", from=1-2, to=2-2]
	\arrow[from=2-1, to=2-2]
	\arrow[from=2-1, to=3-1]
	\arrow["\textnormal{pr}_B", from=2-2, to=3-2]
	\arrow["y"', from=3-1, to=3-2]
\end{tikzcd}\]
The morphism $f$ is unramified by hypotheses. Moreover, by \cite[\href{https://stacks.math.columbia.edu/tag/02JM}{Tag 02JM}]{stacks-project}, the morphism $f$ is also finite, because the post-composition with $\textnormal{pr}_B$ is finite by hypothesis. It follows that the morphism $h$ is unramified and finite, hence universally closed. Moreover, $h$ is also radicial by hypotheses.
By \cite[\href{https://stacks.math.columbia.edu/tag/04XV}{Tag 04XV}]{stacks-project}, it follows that $h$ is a closed immersion.

Consider now the following diagram of pushouts in the category of $k$-algebras:
% https://q.uiver.app/#q=WzAsNCxbMCwwLCJCIFxcb3RpbWVzX2sgQyJdLFsxLDAsIkEiXSxbMCwxLCJcXGthcHBhKHkpIFxcb3RpbWVzX2sgQyJdLFsxLDEsIkEvXFxmcmFrbiBBIl0sWzAsMSwibCJdLFsyLDMsImwnIiwyXSxbMCwyXSxbMSwzXV0=
\[\begin{tikzcd}
	{B \otimes_k C} & A \\
	{\kappa(y) \otimes_k C} & {\kappa(y) \otimes_B A.}
	\arrow["l", from=1-1, to=1-2]
	\arrow[from=1-1, to=2-1]
	\arrow[from=1-2, to=2-2]
	\arrow["{l'}"', from=2-1, to=2-2]
\end{tikzcd}\]
Set $G \coloneq \textnormal{coker}(l)$. Since $B \to A$ is a finite ring map, then $G$ is a finitely generated $B$-module. Then there is an exact sequence of $B$-modules $B \otimes_k C \xrightarrow{l} A \to G \to 0$. We tensor this sequence with the $B$-module $\kappa(y) = B /\frakn B$, where $\frakn$ is the unique maximal ideal of $B$. Since tensor is a right-exact functor, we get the following exact sequence of $B$-modules: $\kappa(y) \otimes_k C \xrightarrow{l'} \kappa(y) \otimes_B A \to \kappa(y) \otimes_B G \to 0$. Since $l'$ is surjective, we have $\kappa(y) \otimes_B G = 0$. Then $0 = B/\frakn B \otimes_B G \simeq G/\frakn G$. It follows that $G= \frakn G$ and $G=0$ by Nakayama's lemma, so that $l$ is surjective.
\end{proof}

Finally we show how to pass from an equivariant closed immersion to an equivariant isomorphism.

\begin{lemma} \label{lemma: exist open with isomorphism in the image}
Let $f \colon X \to Y$ be a $G$-equivariant étale morphism of schemes and let $Z \subseteq X$ be a closed $G$-invariant subscheme. Suppose that $f_{|Z} \colon Z \to Y$ is a closed immersion. Then there exists an open $G$-invariant subscheme $U \subset X$ such that $Z_U \coloneq Z \cap U = Z$ and the composite $f_{|U} \colon U \rightarrow X \rightarrow Y$ induces an isomorphism $f_{|U}^{-1}(f_{|U}(Z)) \xrightarrow{\sim} f_{|U} (Z)$.
\end{lemma}

\begin{proof}

We first show that the induced morphism $Z \to f^{-1}(f(Z))$ is an open and closed immersion.
Since $f_{|Z}$ is a closed immersion, there is an isomorphism $Z \simeq f(Z)$, where $f(Z)$ denotes the scheme-theoretic image of $Z$. Consider the cartesian diagram
% https://q.uiver.app/#q=WzAsNCxbMCwwLCJcXGdhbW1hXnstMX0oXFxnYW1tYShaX3tYXzN9KSkiXSxbMCwxLCJcXGdhbW1hKFpfe1hfM30pIl0sWzEsMCwiWF8zIl0sWzEsMSwiV196XmgiXSxbMCwxXSxbMSwzLCIiLDAseyJzdHlsZSI6eyJ0YWlsIjp7Im5hbWUiOiJob29rIiwic2lkZSI6InRvcCJ9fX1dLFsyLDMsIlxcZ2FtbWEiXSxbMCwyLCIiLDAseyJzdHlsZSI6eyJ0YWlsIjp7Im5hbWUiOiJob29rIiwic2lkZSI6InRvcCJ9fX1dLFszLDAsIiIsMSx7InN0eWxlIjp7Im5hbWUiOiJjb3JuZXItaW52ZXJzZSJ9fV1d
\[\begin{tikzcd}
	{f^{-1}(f(Z))} & {X} \\
	{f(Z)} & {Y.}
	\arrow[hook, from=1-1, to=1-2]
	\arrow[from=1-1, to=2-1]
	\arrow["f", from=1-2, to=2-2]
	\arrow[hook, from=2-1, to=2-2]
	%\arrow["\lrcorner"{anchor=center, pos=0.125}, draw=none, from=1-1, to=2-2]
	%\arrow["\lrcorner"{anchor=center, pos=0.125}, draw=none, from=1-1, to=2-2]
\end{tikzcd}\]
The induced morphism $Z \to f^{-1}(f(Z))$ is a closed immersion by \cite[\href{https://stacks.math.columbia.edu/tag/01QP}{Tag 01QP}]{stacks-project}. On the other hand, since $f$ is étale, the map $f^{-1}(f(Z)) \to f(Z) $ is étale. Under the identification $Z \simeq f(Z)$, the induced morphism $Z \to f^{-1}(f(Z))$ is a section of this étale morphism, and therefore it is an open immersion. It follows that $Z \to f^{-1}(f(Z))$ is an open morphism. Thus, $Z \to f^{-1}(f(Z))$ is an open and closed immersion. It follows that we can write $f^{-1}(f(Z)) \simeq Z \coprod Z'$	for some open and closed subscheme $Z'$ of $f^{-1}(f(Z))$. Then the composite $\psi \colon Z' \to f^{-1}(f(Z)) \to X$ is also a closed immersion. Set $U \coloneqq X \setminus Z'$, where we identify $Z'$ with its scheme-theoretic image $\psi(Z')$ in $X$. The open subscheme $U$ of $X$ gives the desired conclusions as follows. First note that $U$ is a $G$-invariant open subscheme of $X$, because $Z$ is $G$-invariant, and thus also $Z'$ is $G$-invariant. Moreover, $Z_U  \simeq Z$ because we only removed $Z'$ from $X$, and $Z'$ does not intersect $Z$. Finally, the restriction map $f_{|U} \colon U \rightarrow X \rightarrow Y$ induces an isomorphism of $Z_U=Z$ with its image, because $f_{|U}^{-1}(f_{|U}(Z)) \simeq Z$ by construction.
\end{proof}	

\subsection{Openness of properties} \label{subsection: openness of properties}

In this section, we prove in Lemma \ref{lemma: openness of properties} that smoothness, quasi-finiteness and non-vanishing at a point define open conditions in the space of equivariant linear maps between two representations; moreover in Lemma \ref{lemma: assume is étale and quasi-finite} we show that if an equivariant morphism is étale (resp.\ smooth) at some point, we can restrict the source scheme to an open neighborhood of the point to get an étale (resp.\ smooth) equivariant morphism. While the results are classical in algebraic geometry, we restate them here for completeness.

Let $G$ be a finite abelian group. Let $V$ and $W$ be $G$-representations over $k$, i.e.\ finite-dimensional $k[G]$-modules. Suppose that $\dim_k(W) =M$ and $\dim_k(V) = N$, with $M \geq N$. 
Let $\cF \simeq \AA^{MN}_k$ denote the affine space parametrizing linear maps $\AA(W) \simeq \AA^{M}_k \to \AA(V) \simeq \AA^N_k$. Let $\cF^{*}$ denote the open subscheme of $\cF$ parametrizing surjective linear maps $\AA(W) \to \AA(V)$; in other words, a $k$-rational point of $\cF^{*}$ determines a morphism $p \colon \AA(W) \to \AA(V)$ which corresponds to $N$ independent linear forms in $\cO(\AA(W))$. We will refer to surjective linear maps as \emph{projections}. Analogously, let $\cE$ denote the affine space over $k$ parametrizing $G$-equivariant linear maps $\AA(W) \to \AA(V)$ and $\cE^{*}$ be the open subscheme of $\cE$ parametrizing $G$-equivariant projections $\AA(W) \to \AA(V)$. Let $\cP$ be a property of a projection. We say that $\cP$ \emph{defines an open condition on $\cF$} (resp.\ $\cE$) if the set of rational points of $\cF$ (resp.\ $\cE$) corresponding to
projections which satisfy the condition $\cP$ is the set of rational points of an open
subscheme of $\cF$ (resp.\ $\cE$). 

\begin{remark} \label{remark: openness is sufficient ignoring equivariance}
To show that a property $\cP$ defines an open condition on $\cE$ it is sufficient to show that $\cP$ defines an open condition on $\cF$. Indeed, the set of $G$-equivariant linear maps $\Hom_{k[G]}(W,V)$ form a $k$-vector subspace of the space of linear maps $\Hom_k(W,V)$, and thus $\cE$ is a closed subscheme of $\cF$. If $\cP$ defines an open condition on $\cF$, then projections satisfying $\cP$ are the $k$-rational points of an open subscheme $U \subseteq \cF$. In particular, equivariant projections satisfying $\cP$ are the $k$-rational points of the open $U' \coloneq U \cap \cE \subseteq \cE$.
\end{remark}

\begin{lemma} \label{lemma: openness of properties}
	Let $k$ be infinite and let $G$ be a finite abelian group. Let $X$ be an affine smooth $k$-scheme with a $G$-action. Let $Z \subseteq X$ be a $G$-invariant closed subscheme of positive codimension and let $x \in Z$ be a closed point. Let $\iota \colon X \to \AA(W)$ be a $G$-equivariant closed immersion of $X$ into a $G$-representation. Let $V$ be a $G$-representation over $k$ of dimension $\leq \dim X$. Let $\cE$ denote the affine space over $k$ parametrizing $G$-equivariant linear maps $\psi \colon \AA(W) \to \AA(V)$. Then the following properties define an open condition on $\cE$:
	\begin{enumerate}
		\item the composite $\psi \circ \iota \colon X \to \AA(V)$ is smooth at $x$;
		\item the composite $\psi \circ \iota \colon X \to \AA(V)$ does not vanish at $x$;
		\item the composite $c_{\psi} \colon Z \to \AA(W) \to \AA(V)$ is quasi-finite at $x$.
	\end{enumerate} 
\end{lemma}
\begin{proof}
	By Remark \ref{remark: openness is sufficient ignoring equivariance} it is sufficient to show that these properties define an open condition ignoring equivariance. Then (1) is \cite[Proposition 1.2]{Grayson}. The openness of the locus where the composite $ \psi \circ \iota$ does not vanish is well known, because the complement is the locus of $\psi \in \cE$ such that $(\psi \circ \iota) (x)=0$, and this defines a closed subspace of $\cE$, giving (2). To show that (3) gives an open condition we do the following. Let $\cF^{*}$ be the open subscheme of $\AA(\Hom_k(W,V))$ parametrizing surjective linear maps $\AA(W) \to \AA(V)$. Consider the universal morphism
	\[
	a \colon Z \times_k \cF^{*} \longrightarrow \AA(V) \times_k \cF^{*}
	\]
	whose fiber over a $k$-rational point $\psi \in \cF^{*}(k)$ is given by $c_{\psi} \colon Z \to \AA(V)$. The locus $\Omega \subset Z \times_k \cF^{*} $ where $a$ is quasi-finite is open by \cite[\href{https://stacks.math.columbia.edu/tag/01TI}{Tag 01TI}]{stacks-project}. Consider the morphism
	\[
	i \colon \Spec(\kappa(x)) \times_k \cF^{*} \longrightarrow Z \times_k \cF^{*}
	\]
	given by the inclusion of $x$ in $Z$. Then $\Omega' = i^{-1}(\Omega)$ is open in $\Spec(\kappa(x)) \times_k \cF^{*}$. The projection $p \colon \Spec(\kappa(x)) \times_k \cF^{*} \to \cF^{*}$ is an open morphism, and thus $U \coloneq p(\Omega')$ is open in $\cF^{*}$. We claim that $U \subseteq \cF^{*} \subseteq \AA(\Hom_k(W,V))$ is the open subscheme whose $k$-rational points are precisely the projections $\psi$ such that $c_{\psi}$ is quasi-finite at $x$. Indeed, if $c_{\psi} \in U(k)$, then $(x,\psi) \in \Omega$, and thus the morphism $a$ is quasi-finite at $(x, \psi)$ by definition of $\Omega$. It follows that the fiber $\psi \colon Z \to \AA(V)$ of $a$ over $\psi \in \cF^{*}(k)$ is quasi-finite at $x$ by \cite[\href{https://stacks.math.columbia.edu/tag/01TM}{Tag 01TM}]{stacks-project}. Conversely, if $\psi \in \cF^{*}(k)$ is such that $c_{\psi} \colon Z \to \AA(V)$ is quasi-finite at $x$, then again by \cite[\href{https://stacks.math.columbia.edu/tag/01TM}{Tag 01TM}]{stacks-project} we see that the morphism $a$ is quasi-finite at $(x,\psi)$. Then $(x, \psi) \in \Omega$ and thus $\psi \in U(k)$. This proves that also (3) defines an open condition.
\end{proof}

\begin{remark}
The Zariski opens for which properties of Lemma \ref{lemma: openness of properties} hold might be empty. In the proof of Theorem \ref{theorem: gabber} we will show they have $k$-rational points.
\end{remark}

\begin{remark} \label{remark: étaleness open condition}
	Suppose in Lemma \ref{lemma: openness of properties} that $\dim_x(X) = \dim(V)$ holds. Then the property that $\psi \circ \iota \colon X \to \AA(V) $ is étale at $x$ defines an open condition on $\cE$. Indeed, if $f = \psi \circ \iota$ is smooth at $x$, then by \cite[\href{https://stacks.math.columbia.edu/tag/01V9}{Tag 01V9}]{stacks-project} the local ring map $\cO_{\AA(V), f(x)} \to \cO_{X, x}$ is flat and the $\cO_{X,x}$-module $\Omega_{X/\AA(V),x}$ can be generated by at most $\dim_{x}(X_{f(x)})$ elements, where $X_{f(x)}$ denotes the fiber of $f$ over $f(x)$. On the other hand, $\dim_{x}(X_{f(x)})= 0$ by \cite[\href{https://stacks.math.columbia.edu/tag/0AFF}{Tag 0AFF}]{stacks-project} because $\dim_{x}(X) = \dim \AA(V)$. Then $f$ is étale at $x$ by \cite[\href{https://stacks.math.columbia.edu/tag/02GU}{Tag 02GU}]{stacks-project}. The openness of the étaleness condition the follows by Lemma \ref{lemma: openness of properties}.(1).
\end{remark}

We will now prove that if a morphism is étale (resp.\ smooth) and quasi-finite at the closed specialization $z'$ of a point $z$, we may pass to a $G$-invariant open neighborhood and assume that the morphism is étale (resp.\ smooth) and quasi-finite.

\begin{lemma} \label{lemma: assume is étale and quasi-finite}
Let $G$ be a finite abelian group. Let $X,Y$ be affine $k$-schemes with a $G$-action. Assume $X$ is noetherian. Let $Z \subseteq X$ be a $G$-invariant closed subscheme, let $z \in Z$ be a point and let $z' \in Z$ be a closed specialization of $z$. Suppose that there exists a $G$-equivariant morphism of schemes $\varphi \colon X \to Y$ which is étale (resp.\ smooth) at $z'$ and such that the restriction $\varphi_{|Z} \colon Z \to Y$ is quasi-finite at $z'$. Then there exists an open $G$-invariant neighborhood $U \subseteq X$ of $z$ such that $\psi_{|U} \colon U \to Y$ is étale (resp.\ smooth) and the restriction $\psi_{|Z \cap U} \colon Z_{U} \to Y$ is quasi-finite. 
\end{lemma}
\begin{proof}
The locus where a morphism of schemes is étale is open in the source
by \cite[\href{https://stacks.math.columbia.edu/tag/02GI}{Tag 02GI}]{stacks-project}.
Hence, there exists an open
neighborhood $V \subseteq X$ of $z'$ such that the restriction $\varphi_{|V}$ is étale
and the restriction of $\varphi_{|V}$ to $Z_V \coloneqq Z \cap V$ is quasi-finite at $z'$. After shrinking $V$, we may assume it is affine. Moreover, the locus where a morphism of schemes is quasi-finite is open
by \cite[\href{https://stacks.math.columbia.edu/tag/01TI}{Tag 01TI}]{stacks-project}. Let $\Omega \subseteq Z_V$ be this open locus. Note that the restriction $\varphi_{|\Omega} \coloneq \Omega \to Y$ is quasi-finite by \cite[\href{https://stacks.math.columbia.edu/tag/01TJ}{Tag 01TJ}]{stacks-project}, because it is locally quasi-finite by \cite[\href{https://stacks.math.columbia.edu/tag/01TI}{Tag 01TI}]{stacks-project} and it is also quasi-compact.
Since $\Omega$ is open in $Z_V$, we may write $\Omega = Z_V \cap \Omega'$ for some open subset $\Omega' \subseteq V$.	Set $U \coloneq \bigcup_{g \in G} g(\Omega')$. Then $U$ is an open $G$-invariant neighborhood of $z'$ in $X$. Moreover, $\varphi_{|U} \colon U \to Y$ is étale and the restriction to $Z_U$ is quasi-finite by construction, as desired. Finally, note that $z \in U$, because any open neighborhood of $z'$ contains $z$. The smoothness case is identical, replacing the étale locus by the smooth locus, which is open by \cite[\href{https://stacks.math.columbia.edu/tag/01V5}{Tag 01V5}]{stacks-project}.
\end{proof}

\section{Radicial and étale equivariant morphisms over local rings}

Let $R$ be a local ring equipped with an action of a finite abelian group $G$. Set $W\coloneq \Spec(R)$ and let $w \in W$ be the unique closed point. Throughout this section, $\AA^1_k$ denotes the trivial one-dimensional $G$-representation over $k$ (i.e.\ $\AA^1_k = \AA(V)$ where $V$ is the trivial one-dimensional $G$-representation). We show that, under suitable hypotheses, a $G$-equivariant morphism with target $\AA^1_{\kappa(w)}$ can be lifted to one with target $\AA^1_W$ and we verify that certain properties are preserved under this lifting. Then in Lemma \ref{lemma: exists a fiber that is étale and radicial at a point} we show how to produce an equivariant morphism with target $\AA^1_{\kappa(w)}$
enjoying these properties.

\begin{lemma} \label{lemma: lift the morphism}
Let $G$ be a finite abelian group such that $|G|$ is invertible in $k$. Let $l \colon X \to W$ be a $G$-equivariant morphism of affine schemes, where $W$ is the spectrum of a local ring. Let $w \in W$ be the unique closed point and suppose that $w$ is set-theoretically fixed (i.e.\ $Gw=w$). Let $x \in X$ be a closed point such that $l(x) = w$. Let $\AA^1_k$ be the trivial one-dimensional $G$-representation. Let $X_w$ and $\AA^1_{\kappa(w)}$ be the fibers over $w$ of $l \colon X \to W$ and $p \colon \AA^1_W \to W$ respectively. Suppose that we are given a $G$-equivariant morphism of schemes $f_w \colon X_w \to \AA^1_{\kappa(w)}$. Then
\begin{itemize}
	\item[(1)] there exists a (not necessarily $G$-equivariant) lift $f \colon X \to \AA^1_W$ of $f_w$;
	
	\item[(2)] if $f_w$ is radicial at $x$, then the symmetrization $\tilde{f}$ of $f$ is a $G$-equivariant morphism of schemes $\tilde{f}\colon X \to \AA^1_W$ such that the fiber $\tilde{f}_w \colon X_{w} \to \AA^1_{\kappa(w)}$ is radicial at $x$;
	\item[(3)] if $f_w$ is étale at $x$, then the symmetrization $\tilde{f}$ of $f$ is a $G$-equivariant morphism of schemes $\tilde{f}\colon X \to \AA^1_W$ such that the fiber $\tilde{f}_w \colon X_{w} \to \AA^1_{\kappa(w)}$ is étale at $x$.
\end{itemize} 
\end{lemma}

\begin{proof}
Write $X= \Spec(A)$ for some $k$-algebra $A$ and $W= \Spec(R)$ for some local $k$-algebra $R$. The morphism $f_w$ corresponds to an element $ f'_w \in A \otimes_R \kappa(w)$ which is the image of $t$ under the morphism $\kappa(w)[t] \to A \otimes_R\kappa(w)$ corresponding to $X_w \to \AA^1_{\kappa(w)}$. A lift $f \colon X \to \AA^1_W$ of $f_w$ would correspond to an element $f' \in A$ with specified image $f'_w \in A \otimes_R \kappa(w)$. Since $A \to A \otimes_R \kappa(w)$ is surjective, such a lift $f$ exists, giving (1).

We define the symmetrization $\tilde{f}$ of $f$ as follows.
Note that $A \otimes_R \kappa(w) \simeq A/\frakn A$ where $\frakn$ is the maximal ideal of $R$ corresponding to the unique closed point $w$. We thus have the following commutative diagram:
% https://q.uiver.app/#q=WzAsNCxbMCwwLCJBIC9cXGZyYWtuIEEiXSxbMCwxLCJcXGthcHBhKHcpW3RdIl0sWzEsMSwiUlt0XSJdLFsxLDAsIkEiXSxbMywwXSxbMSwwLCJ0IFxcbWFwc3RvIGZfdyJdLFsyLDFdLFsyLDMsInQgXFxtYXBzdG8gZiIsMl1d
\[\begin{tikzcd}
	{A /\frakn A} & A \\
	{\kappa(w)[t]} & {R[t].}
	\arrow[from=1-2, to=1-1]
	\arrow["{t \mapsto f'_w}", from=2-1, to=1-1]
	\arrow["{t \mapsto f'}"', from=2-2, to=1-2]
	\arrow[from=2-2, to=2-1]
\end{tikzcd}\]
Set $\tilde{f}' \coloneq \sum_{g \in G} g(f')$. Then $\tilde{f}'$ defines a $G$-equivariant morphism $\tilde{f}\colon X \to \AA^1_k \times_k W \simeq \Spec(R[t])$, because $\AA^1_k$ has the trivial $G$-action, and thus the diagram
% https://q.uiver.app/#q=WzAsNCxbMCwwLCJSW3RdIl0sWzAsMSwiUlt0XSJdLFsxLDEsIkEiXSxbMSwwLCJBIl0sWzAsMSwiZyIsMl0sWzEsMiwidCBcXHRvIGYiLDJdLFszLDIsImciXSxbMCwzLCJ0IFxcdG8gZiJdXQ==
\[\begin{tikzcd}
	{R[t]} & A \\
	{R[t]} & A
	\arrow["{t \to \tilde{f}'}", from=1-1, to=1-2]
	\arrow["g"', from=1-1, to=2-1]
	\arrow["g", from=1-2, to=2-2]
	\arrow["{t \to \tilde{f}'}"', from=2-1, to=2-2]
\end{tikzcd}\]
commutes for all $g \in G$. 

We next show the permanence of properties (2) and (3) on the fiber $\tilde{f}_w \colon X_w \to \AA^1_{\kappa(w)}$ of $\tilde{f}$ over $w$. In order to do that, we give an explicit description of the morphism $\tilde{f}_w$. Consider the cartesian square
% https://q.uiver.app/#q=WzAsNCxbMCwwLCJYX3ciXSxbMCwxLCJcXEFBXjFfe1xca2FwcGEodyl9Il0sWzEsMSwiXFxBQV4xX1ciXSxbMSwwLCJYIl0sWzMsMiwiXFx0aWxkZXtmfSJdLFsxLDJdLFswLDEsIlxcdGlsZGV7Zn1fdyIsMl0sWzAsM11d
\[\begin{tikzcd}
	{X_w} & X \\
	{\AA^1_{\kappa(w)}} & {\AA^1_W}
	\arrow[from=1-1, to=1-2]
	\arrow["{\tilde{f}_w}"', from=1-1, to=2-1]
	\arrow["{\tilde{f}}", from=1-2, to=2-2]
	\arrow[from=2-1, to=2-2]
\end{tikzcd}\]
and the corresponding push-out square of algebras
% https://q.uiver.app/#q=WzAsNCxbMSwxLCJSW3RdIl0sWzEsMCwiQSJdLFswLDEsIlxca2FwcGEodylbdF0iXSxbMCwwLCJBIC9cXGZyYWtuIEEiXSxbMiwzXSxbMCwyXSxbMCwxLCJ0IFxcbWFwc3RvXFx0aWxkZXtmfSIsMl0sWzEsM11d
\begin{equation} \label{equation: commutative pushout in the algebras}
	\begin{tikzcd}
		{A /\frakn A} & A \\
		{\kappa(w)[t]} & {R[t].}
		\arrow[from=1-2, to=1-1]
		\arrow["{t \mapsto [\tilde{f}']}", from=2-1, to=1-1]
		\arrow["{t \mapsto\tilde{f}'}"', from=2-2, to=1-2]
		\arrow[ from=2-2, to=2-1]
	\end{tikzcd}
\end{equation}
Note that the fiber morphism in diagram \ref{equation: commutative pushout in the algebras} sends $t$ to the class $[\tilde{f}']$ of $\tilde{f}'$ in $A/\frakn A$. Since $w$ is set-theoretically fixed, the morphism $w \colon \Spec(\kappa(w)) \to W$ is $G$-equivariant. In particular, the base change $X_w \to X$ is $G$-equivariant, so that the corresponding quotient morphism $A \to A/\frakn A$ appearing in diagram \ref{equation: commutative pushout in the algebras} is $G$-equivariant. It follows that
\[
[\tilde{f}'] = [\sum_{g \in G} g(f')] = \sum_{g \in G} [g(f')] = \sum_{g \in G} g([f']) = \sum_{g \in G} g(f'_w).
\]
On the other hand, $\AA^1_k$ has trivial $G$-action, and thus the $G$-equivariance of the map $f_w \colon X_w \to \AA^1_{\kappa(w)}$ gives $g(f'_w) = f'_w$ for all $g \in G$. It follows that $[\tilde{f}'] = |G|f'_w \in A/ \frakn A$. In other words, the morphism $\tilde{f}_w \colon X_w \to \AA^1_{\kappa(w)}$ corresponds to the morphism $\kappa(w)[t] \to A/\frakn A$ mapping $t \mapsto |G|f'_w$. 

The multiplication by $|G|$ is an automorphism of $\AA^1_{\kappa(w)}$ because $|G|$ is invertible in $k$ by hypotheses. 

The field extension $\kappa(x)/ \kappa(\tilde{f}_w(x))$ is the same as $\kappa(x)/ \kappa(f_w(x))$, because it is not affected by the multiplication by $|G|$. It follows that if $f_{w}$ is radicial at $x$, also $\tilde{f}_w$ is radicial at $x$, giving (2). 

If $f_w$ is étale at $x$, then also $\tilde{f}_w$ is étale at $x$, because the multiplication by $|G|$ is an automorphism of $\AA^1_{\kappa(w)}$, giving (3).
\end{proof}

Recall that any $G$-equivariant morphism of schemes $X \to Y$ induces a morphism $X/G \to Y/G$, whenever the quotients exist as schemes. Conversely, if $X \to X/G$ and $Y \to Y/G$ are $G$-torsors and a morphism $X/G \to Y/G$ identifies $X$ with the pullback $(X/G) \times_{Y/G} Y$, then the base-change projection gives the corresponding $G$-equivariant morphism $X \to Y$.

In order to guarantee the existence of radicial and étale morphisms on the fiber as in the hypotheses of Lemma \ref{lemma: lift the morphism}, we will use the following lemma. 

\begin{lemma} \label{lemma: exists a fiber that is étale and radicial at a point}
Let $k$ be infinite. Let $G$ be a finite abelian group. Let $L/k$ be a field extension. Suppose that $G$ acts faithfully on $L$ as $k$-automorphisms (i.e.\ $\textnormal{ker}(G \to \Aut_k(L)) = \left\{ e\right\}$). Let $X$ be an affine $k$-scheme of pure dimension one with free $G$-action and let $x \in X$ be a closed point. Let $f \colon X \to \Spec(L)$ be a $G$-equivariant smooth morphism of schemes over $k$ and let $\AA^1_k$ be the trivial one-dimensional $G$-representation. Then there exists a $G$-equivariant morphism $X \to \AA^1_{L}= \AA^1_k \times_k \Spec(L)$ which is étale at $x$ and radicial at $x$.
\end{lemma}

\begin{proof}
Write $X= \Spec(A)$. The quotients $X/G$ and $\Spec(L)/G$ exist as schemes; more precisely, they are $X/G = \Spec(A^G)$ and $\Spec(L)/G = \Spec(F)$ where $F \coloneq L^G$ is the subfield of $L$ of fixed elements. Note that $\AA^1_L /G \simeq \AA^1_F$. A $G$-equivariant morphism $X \to \AA^1_L$ over $L$ induces a morphism $X/G \to \AA^1_F$ over $F$. Conversely, we claim that any morphism of schemes $X/G \to \AA^1_F$ over $F$ defines a $G$-equivariant morphism $X \to \AA^1_L$. Indeed, $X \to X/G$ and $\AA^1_L \to \AA^1_F$ are $G$-torsors and there is a cartesian diagram of $G$-torsors as follows:
% https://q.uiver.app/#q=WzAsNixbMiwxLCJcXFNwZWMoRikiXSxbMiwwLCJcXFNwZWMoTCkiXSxbMSwxLCJcXEFBXjFfRiJdLFsxLDAsIlxcQUFeMV9MIl0sWzAsMSwiWC9HIl0sWzAsMCwiWC9HIFxcdGltZXNfRiBcXFNwZWMoTCkiXSxbNSw0XSxbNCwyXSxbMSwwXSxbMiwwXSxbMywyXSxbNSwzXSxbMywxXV0=
\[\begin{tikzcd}
	{X/G \times_F \Spec(L)} & {\AA^1_L} & {\Spec(L)} \\
	{X/G} & {\AA^1_F} & {\Spec(F).}
	\arrow[from=1-1, to=1-2]
	\arrow[from=1-1, to=2-1]
	\arrow[from=1-2, to=1-3]
	\arrow[from=1-2, to=2-2]
	\arrow[from=1-3, to=2-3]
	\arrow[from=2-1, to=2-2]
	\arrow[from=2-2, to=2-3]
\end{tikzcd}\]
Then the $G$-equivariant morphism $f \colon X \to \Spec(L)$ and the projection morphism $X \to X/G$ define a morphism $X \to X/G \times_F \Spec(L)$ of $G$-torsors over $X/G$, which is necessarily an isomorphism. Then the base change $X \to \AA^1_L$ of $X/G \to \AA^1_F$ is a $G$-equivariant morphism as desired.
Let $\pi \colon X \to X/G$ be the projection to the quotient. We claim that it is sufficient to obtain the same non-equivariant conclusions over the quotient schemes $X/G$ and $\Spec(L)/G$.   Suppose that we have constructed a morphism $h \colon X/G \to \AA^1_F$ over $F$ which is radicial at $\pi(x)$ (note that this is indeed a closed point of $X/G$ because $\pi$ is integral, and thus is a closed morphism) and étale at $\pi(x)$. Consider the following cartesian diagram of $G$-torsors:
% https://q.uiver.app/#q=WzAsOCxbMywwLCJcXFNwZWMoTCkiXSxbMywxLCJcXFNwZWMoTCkvRyBcXHNpbWVxIFxcU3BlYyhGKSJdLFsyLDEsIlxcQUFeMV9GIl0sWzIsMCwiXFxBQV4xX0wiXSxbMSwxLCJYL0ciXSxbMSwwLCJYIl0sWzAsMSwiXFxwaSh4KSJdLFswLDAsIlxccGleey0xfShcXHBpKHgpKSJdLFswLDFdLFs3LDZdLFs2LDQsIiIsMCx7InN0eWxlIjp7InRhaWwiOnsibmFtZSI6Imhvb2siLCJzaWRlIjoidG9wIn19fV0sWzcsNSwiIiwyLHsic3R5bGUiOnsidGFpbCI6eyJuYW1lIjoiaG9vayIsInNpZGUiOiJ0b3AifX19XSxbNSw0LCJcXHBpIiwyXSxbNCwyXSxbNSwzXSxbMiwxXSxbMywwXSxbMywyXV0=
\[\begin{tikzcd}
	{\pi^{-1}(\pi(x))} & X & {\AA^1_L} & {\Spec(L)} \\
	{\pi(x)} & {X/G} & {\AA^1_F} & {\Spec(F).}
	\arrow[hook, from=1-1, to=1-2]
	\arrow[from=1-1, to=2-1]
	\arrow[from=1-2, to=1-3]
	\arrow["\pi"', from=1-2, to=2-2]
	\arrow[from=1-3, to=1-4]
	\arrow[from=1-3, to=2-3]
	\arrow[from=1-4, to=2-4]
	\arrow[hook, from=2-1, to=2-2]
	\arrow[from=2-2, to=2-3]
	\arrow[from=2-3, to=2-4]
\end{tikzcd}\]
Then we see that $\pi^{-1}(\pi(x)) \to \AA^1_L$ is radicial, and thus in particular $X \to \AA^1_L$ is a $G$-equivariant morphism which is radicial at $x$. 
Moreover $X/G \to \AA^1_F$ is étale at $\pi(x)$. By definition, this means that there exists an affine open $\pi(x) \in U \subseteq X/G$ and an affine open $V \subset \AA^1_F$ such that $h(U) \subset V$ and $h_{|U} \colon U \to V$ is étale. In particular, $h_{|U} \colon U \to \AA^1_F$ is étale. Then the base change $\pi^{-1}(U) \to  X \to \AA^1_L$ is étale. Since $\pi^{-1}(U)$ is an open neighborhood of $x$, this means that $X \to \AA^1_L$ is étale at $x$.

It remains to show that there exists a morphism of schemes $h \colon X/G \to \AA^1_F$ which is radicial at $\pi(x)$ and étale at $\pi(x)$. In order to find such a morphism, we will use \cite[Theorem 3.2.2]{Colliot-Thelene}. Indeed, the scheme $X/G$ is an affine smooth $F$-scheme because $X$ is a smooth $L$-scheme and smoothness is an fpqc local property on the target. Moreover, $X/G$ has pure dimension one because $X$ has pure dimension one and $\pi \colon X \to X/G$ is integral. Let $(X/G)'$ be the irreducible component of $X/G$ containing $\pi(x)$. If we find a morphism $h' \colon (X/G)' \to \AA^1_F$ which is radicial at $\pi(x)$ and étale at $\pi(x)$, then any morphism $h \colon X/G \to \AA^1_F$ whose restriction to $(X/G)'$ is $h'$ will be radicial at $\pi(x)$ and étale at $\pi(x)$. We may thus assume that $X/G$ is also irreducible.

To summarize, $X/G$ is an affine irreducible smooth $F$-scheme of dimension one and $\pi(x) \in X/G$ is a closed point. Since $k$ is infinite, then $F$ is also an infinite field. By \cite[Theorem 3.2.2]{Colliot-Thelene}, there exists a morphism $h \colon X/G \to \AA^1_F$ over $F$ which is radicial at $\pi(x)$ and étale at $\pi(x)$, as desired.
\end{proof}

\begin{remark}
Note that we could not apply Proposition \ref{proposition: étale morphism as in 8.10 of equivariant cycles} to conclude the existence of such an étale morphism, because even if we assumed $k$ perfect, it might be that the field $L$ is not perfect. Instead, we pass to the quotient and conclude by non-equivariant case, where we do not need the base field to be perfect.
\end{remark}

\section{An equivariant version of Gabber's lemma for cyclic $p$-groups}\label{section: proof of theorem}

In view of Section \ref{section: counterexamples}, it remains to determine whether the equivariant version of Gabber's lemma proposed by Bachmann in \cite{Bachmann} holds for cyclic $p$-groups. In this section we prove that under suitable hypotheses on the ground field $k$, the equivariant Gabber's lemma generalizes to such groups. We will use the terminology introduced in Definition \ref{definition: gabber presentation}.

\begin{theorem} \label{theorem: gabber}
Let $k$ be an infinite perfect field. Let $p$ be a prime and let $r \geq 1$ be an integer such that $k$ contains a primitive $p^r$th root of unity. Let $(X,Z)$ be a $C_{p^r}$-pair and let $z \in Z$ be a point. Then there exists a $C_{p^r}$-equivariant Gabber's presentation $(X', W, V, \varphi = (\psi, \nu))$ of $(X,Z)$ with respect to the point $z$ (see Definition \ref{definition: gabber presentation}). Moreover, if $V$ is a nontrivial $C_{p^r}$-representation, then $\psi \colon X' \to W$ admits a $C_{p^r}$-equivariant section.
\end{theorem}

Combining Theorem \ref{theorem: counterexample} and Theorem \ref{theorem: gabber} we get the following.

\begin{corollary}
Let $k$ be an infinite perfect field. Let $G$ be a finite abelian group such that $k$ contains a primitive $p^r$th root of unity for each cyclic summand $C_{p^r}$
appearing in the decomposition of $G$. Then every $G$-pair  admits a $G$-equivariant Gabber's presentation over $k$ if and only if $G$ is a cyclic $p$-group.
\end{corollary}

\begin{remark}
By Theorem \ref{theorem: counterexample}, if $G$ decomposes into more than one factor under the structure theorem for finite abelian groups, it is enough that the decomposition contain cyclic summands $C_{p_1^{r_1}}$ and $C_{p_2^{r_2}}$ for which $k$ contains  primitive $p_i$th roots of unity, $i=1,2$, to produce a counterexample.
\end{remark}

The proof of Theorem \ref{theorem: gabber} proceeds through a sequence of claims, each of which establishes an additional property of the morphism $\varphi$, until all required conditions are satisfied. We follow the same line of reasoning as \cite{Bachmann}. In particular, the argument is divided into two cases: first, we treat the case where $z$ is a point with free $G$-action; then we treat the case where $z \in X^H$ for some nontrivial subgroup $H \subseteq G$.

We will first obtain the conclusions (i.e.\ a Gabber's presentation) over the henselization $W^h_{Gw}$ of a scheme $W$ at the orbit of a point $w \in W$. Since the scheme $W^h_{Gw}$ is the limit of $G$-equivariant Nisnevich neighborhoods of $Gw$ in $W$, we will then apply the following limit-descending argument in order to have the same conclusions over an actual $G$-equivariant Nisnevich neighborhood of $Gw$ in $W$.

\begin{lemma} \label{lemma: limit descending}
Let $G$ be a finite group, viewed as a constant group scheme over $k$. Let $S = \lim_{i \in I} S_i$ be the limit of an inverse system over $I$ of affine $G$-schemes over $k$ such that the transition morphisms are $G$-equivariant. Let
% https://q.uiver.app/#q=WzAsMyxbMCwwLCJYIl0sWzIsMCwiWSJdLFsxLDEsIlMiXSxbMCwyLCJmX1giLDJdLFswLDEsImYiXSxbMSwyLCJmX1kiXV0=
\[\begin{tikzcd}
	X && Y \\
	& S
	\arrow["f", from=1-1, to=1-3]
	\arrow["{f_X}"', from=1-1, to=2-2]
	\arrow["{f_Y}", from=1-3, to=2-2]
\end{tikzcd}\]
be a diagram of $G$-equivariant maps, where	both $X$ and $Y$ are of finite presentation over $S$. Then there exists an index $i \in I$ and $G$-schemes $X_i, Y_i$ with a $G$-equivariant morphism $f_i \colon X_i \to Y_i$ over $S_i$ fitting into a cartesian diagram
% https://q.uiver.app/#q=WzAsNCxbMCwwLCJYIl0sWzAsMSwiWSJdLFsxLDAsIlhfaiJdLFsxLDEsIllfaiJdLFsxLDNdLFsyLDMsImZfaiJdLFswLDEsImYiLDJdLFswLDJdXQ==
\[\begin{tikzcd}
	X & {X_i} \\
	Y & {Y_i.}
	\arrow[from=1-1, to=1-2]
	\arrow["f"', from=1-1, to=2-1]
	\arrow["{f_i}", from=1-2, to=2-2]
	\arrow[from=2-1, to=2-2]
\end{tikzcd}\]
Suppose further that $f$ has one of the following properties:
étale, open immersion, closed immersion, isomorphism, finite. Then we may assume that $f_i$ has the same property.
\end{lemma}

\begin{proof}
This is an easy consequence of \cite[\href{https://stacks.math.columbia.edu/tag/01ZM}{Tag 01ZM}]{stacks-project} and
\cite[\href{https://stacks.math.columbia.edu/tag/081C}{Tag 081C}]{stacks-project}.
\end{proof}

We will also use the following lemma, that allows us to restrict the situation to an affine Nisnevich neighborhood every time we need. 

\begin{lemma}  \label{lemma: assume is affine}
Let $X$ be a $G$-scheme separated over $k$. Then for any point $x \in X$ there exists an affine $G$-equivariant Nisnevich neighborhood of $Gx$.
\end{lemma}
\begin{proof}
Let $x \in X$. We first show that $x$ is contained in an affine $S_x$-invariant open neighborhood $U$. Indeed, let $V$ be any affine open neighborhood of $x$ in $X$. Set $U \coloneq \bigcap_{s \in S_x} s(V)$. Then $U$ is affine by \cite[\href{https://stacks.math.columbia.edu/tag/01KP}{Tag 01KP}]{stacks-project}, because $X \to \Spec(k)$ is a separated morphism. Moreover $U$ is non-empty, because $x \in U$. Finally, $U$ is clearly $S_x$-invariant as desired. 

The scheme $G \times^{S_x} U$ is affine, and the map $\varphi \colon G \times^{S_x} U \to X$, $[(g,u)] \mapsto gu$, together with the $G$-equivariant section $s \colon G \times^{S_x} \Spec(\kappa(x)) \to G \times^{S_x} U$ mapping $[(g,x)] \mapsto [(g,x)]$, form a $G$-equivariant Nisnevich neighborhood of $Gx$.	
\end{proof}

\begin{remark} \label{rem: open neighborhoods are Nisnevich neighborhoods}
Any open $G$-invariant subscheme $U \subset X$ containing $x$ is a $G$-equivariant Nisnevich neighborhood of $Gx$. Indeed the open immersion $U \to X$ is étale, and $Gx \in U$.
\end{remark}

The following useful fact will allow us to reduce the case in which the tangent space at our point of interest is a non-trivial representation.

\begin{lemma} \label{lemma: trivial action on tangent space}
	Let $k$ be a perfect field. Let $G$ be a finite abelian group such that $\textnormal{char}(k) \nmid |G|$ and let $X$ be a smooth $k$-scheme of finite type with a $G$-action. Let $x \in X$ be a closed point and suppose that the $G_x$-action on $T_xX$ is trivial. Then there exists a $G$-invariant open neighborhood of $x$ in $X$ over which the $G_x$-action is trivial. 
\end{lemma}
\begin{proof}
	Set $H \coloneq G_x$. Let $\overline{k}$ be an algebraic closure of $k$ and set $\overline{X} \coloneq X \times_k \Spec(\overline{k})$. The $G$-action on $X$ over $k$ induces a $G$-action on $\overline{X}$ over $\overline{k}$. Let $\overline{x} \in \overline{X}$ be a point over $x$. Since $H$ is the geometric stabilizer of $x$, we have $\overline{x} \in \overline{X}^H$. Since $T_{x}X$ is a trivial $H$-representation over $\kappa(x)$ and $T_{\overline{x}}\overline{X}= T_{x}X \otimes_{\kappa(x)}\overline{k}$, then $T_{\overline{x}}\overline{X}$ is a trivial $H$-representation over $\overline{k}$. By \cite[Theorem 9]{Boratynski}, the $H$-action on $\widehat{\cO}_{\overline{x}, \overline{X}}$ is linearizable, and hence trivial, because is trivial on the tangent space. Then the $H$-action on $\cO_{\overline{x}, \overline{X}}$ is also trivial, because the completion morphism $\cO_{\overline{x}, \overline{X}} \to \widehat{\cO}_{\overline{x}, \overline{X}}$ is injective and $H$-equivariant. Because $\overline{X}$ is of finite type over $\overline{k}$, after shrinking around $\overline{x}$, the corresponding automorphisms agree with the identity. Intersecting the finitely many such neighborhoods and then taking their $G$-translates gives a $G$-invariant open neighborhood $U$ of $\overline{x}$ on which $H$ acts trivially. The image of $U$ in $X$ is an open neighborhood of $x$ in $X$ which is $G$-invariant and has trivial $H$-action. 
\end{proof}

We are now ready to prove Theorem \ref{theorem: gabber}.

\begin{proof}[Proof of Theorem \ref{theorem: gabber}]
Set $G = C_{p^r}$. We divide the proof into three cases as follows:
\begin{itemize}
	\item Case \ref{subsection: case A}: $G_z = \left\{e\right\}$ and $S_z = G$;
	\item Case \ref{subsection: case B}: $G_z = \left\{e\right\}$ and $S_z \subsetneq G$;
	\item Case \ref{subsection: case C}: $G_z \neq \left\{e\right\}$.
\end{itemize}

\subsection{The case when $G_z = \left\{e\right\}$ and $S_z = G$}\label{subsection: case A}

We first treat the case $G_z = \left\{e\right\}$ and $S_z= G$. After replacing $X$ by the open $G$-invariant subscheme $X \setminus \bigcup_{H \neq \left\{e\right\}} X^H$ we may assume that the $G$-action on $X$ is free and, after replacing $X$ again by an affine equivariant Nisnevich neighborhood of $Gz=z$ as in Lemma \ref{lemma: assume is affine}, we may also assume that $X$ is affine. Finally, since the point $z$ is set-theoretically fixed, the unique irreducible component of $X$ containing $z$ is $G$-invariant. After replacing $X$ with the irreducible component containing $z$ we may assume that $X$ is also irreducible. Let $z \colon \Spec(\kappa(z)) \to X$ be the $G$-equivariant morphism corresponding to the point $z$. The $G$-action on the field $\kappa(z)$ is faithful by hypothesis of Case \ref{subsection: case A}. The following claim will be used to treat the one-dimensional case in Claim \ref{claim: one-dimensional case}.

\begin{claim} \label{claim: trivial subcase}
	Suppose that the $G$-action on $X$ is free. Let $k \subset K$ be a finitely generated separable field extension and suppose that $G$ acts faithfully on $K$ by $k$-automorphisms. If there exist an equivariant Nisnevich neighborhood $f \colon \tilde{X} \to X$ of $Gz=z$ with $f$ separated and a $G$-equivariant smooth morphism $\tilde{X} \to \Spec(K)$, then Theorem \ref{theorem: gabber} holds.
\end{claim}

\begin{proof}
	Without loss of generality, we may assume that $\tilde{X}$ is affine by Lemma \ref{lemma: assume is affine}, because $f$ is a separated morphism, and thus $\tilde{X}$ is separated over $k$. We claim that we may also assume $\tilde{X}$ to be irreducible. Indeed, $\tilde{X}$ is smooth over $k$, thus it is the disjoint union of its irreducible components. Let $X_1$ be the unique irreducible component of $\tilde{X}$ containing $z$. Since $g(z)=z$ for all $g \in G$ and $g(X_1)$ is again an irreducible component of $\tilde{X}$, we must have $g(X_1)=X_1$ for all $g \in G$, because the intersection between different irreducible components is empty. Thus, $X_1$ is an open $G$-invariant irreducible subscheme of $\tilde{X}$ containing $z$. After replacing $\tilde{X}$ with $X_1$, we may assume that $\tilde{X}$ is irreducible, as desired.
	
	The group $G$ acts freely on $\tilde{X}$, because the morphism  $f \colon \tilde{X} \to X$ is $G$-equivariant and the $G$-action on $X$ is free by hypothesis. Set $\tilde{Z} \coloneq Z \times_X \tilde{X}$. We will show that the ordered pair $(\tilde{X}/G, \tilde{Z}/G)$ forms a pair in the sense of Definition \ref{definition: gabber presentation} with $\tilde{X}/G$ affine and irreducible, and then we apply non-equivariant Gabber's lemma to conclude.

	Note that $\Spec(K)/G \simeq \Spec(K^G)$ and by \cite[\href{https://stacks.math.columbia.edu/tag/09I3}{Tag 09I3}]{stacks-project}, since the $G$-action on $K$ is faithful, $K/K^G$ is Galois with Galois group $G$. We can consider the following cartesian diagram of $G$-torsors:
	% https://q.uiver.app/#q=WzAsNixbMSwwLCJYJyJdLFsxLDEsIlgnL0ciXSxbMiwxLCJcXFNwZWMoS15HKSJdLFsyLDAsIlxcU3BlYyhLKSJdLFswLDEsIlonL0ciXSxbMCwwLCJaJyJdLFszLDJdLFsxLDJdLFswLDFdLFswLDNdLFs1LDRdLFs0LDFdLFs1LDBdXQ==
	\begin{equation} \label{equation: diagram of torsors}\begin{tikzcd}
			{\tilde{Z}} & {\tilde{X}} & {\Spec(K)} \\
			{\tilde{Z}/G} & {\tilde{X}/G} & {\Spec(K^G).}
			\arrow[from=1-1, to=1-2]
			\arrow[from=1-1, to=2-1]
			\arrow[from=1-2, to=1-3]
			\arrow["\pi", from=1-2, to=2-2]
			\arrow[from=1-3, to=2-3]
			\arrow[from=2-1, to=2-2]
			\arrow[from=2-2, to=2-3]
		\end{tikzcd}
	\end{equation}
	Set $d \coloneq \dim_x(X)$ and $F \coloneq K^G$. The morphism $\tilde{Z}/G \to \tilde{X}/G$ in diagram \ref{equation: diagram of torsors} is a closed immersion, because the property of being a closed immersion is fpqc local on the base and $\pi$ is a faithfully flat quasi-compact morphism.
	We have $\dim(\tilde{X}) = d$ because $f \colon \tilde{X} \to X$ is étale. Moreover, the induced morphism $\tilde{Z} \to Z$ is étale, and thus $\dim(\tilde{Z}) \leq \dim(Z)$. Since the codimension of $Z$ in $X$ is positive, then also the codimension of $\tilde{Z}$ in $\tilde{X}$ is positive. Since $\pi$ is integral, we also have $\dim(\tilde{X}/G) = d$ and the closed subscheme $\tilde{Z}/G$ has positive codimension in $\tilde{X}/G$. Finally, note that $\tilde{X}/G$ is irreducible and it is smooth over $\Spec(F)$ because $\tilde{X} \to \Spec(K)$ is smooth and smoothness is an fpqc (even étale) local property on the base. It follows that $(\tilde{X}/G,\tilde{Z}/G)$ is a pair over $F$ with $\tilde{X}/G$ affine and irreducible, as desired. Set $\tilde{z} \coloneq \pi(z) \in \tilde{X}/G$.
	
	By Gabber's lemma (Theorem \ref{theorem: gabber original}), there exists a Gabber's presentation with respect to the point $\tilde{z}$. More precisely, the following data exist:
	\begin{itemize}
		\item an open neighborhood $U \subset \tilde{X}/G$ of $\tilde{z}$;
		\item an open subset $\tilde{W} \subset \AA^{d-1}_{F}$;
		\item an étale morphism $\varphi  = (\psi, \nu) \colon U \to \AA^1_{\tilde{W}} = \tilde{W} \times_F \AA^1_F$ such that $\varphi_{|Z_U}$ is a closed immersion (where $Z_U \coloneq \tilde{Z}/G \cap U$), the composite $ 	\psi_{|Z_U} \colon Z_U \rightarrow \AA^1_{\tilde{W}} \to \tilde{W}$ is finite and $\varphi$ restricts to an isomorphism $\varphi^{-1}(\varphi(Z_U)) \xrightarrow{\sim} \varphi(Z_U)$.
	\end{itemize}
	We will now show that these data provide a $G$-equivariant Gabber's presentation of $(X,Z)$ with respect to the point $z$, as desired. We use notation of Definition \ref{definition: gabber presentation}. 
	Set $X' \coloneq U \times_F \Spec(K) =  \pi^{-1}(U)$, which is a $G$-equivariant open subscheme of $\tilde{X}$ containing $z$. In particular, $X' \to \tilde{X} \to X$ is a G-equivariant Nisnevich neighborhood of $Gz=z$, giving $\Gnum{1}$. Let $Z' \coloneq \tilde{Z} \times_{\tilde{X}} X' = \tilde{Z}\cap X'$. Set $W \coloneq \tilde{W} \times_F \Spec(K)$ so that $\AA^1_{\tilde{W}} \times_{F} \Spec(K) \simeq \AA^1_W$.

	Note that the scheme $W$ has a $G$-action which is given by the action of $G$ on $\Spec(K)$. Since $W$ is a smooth $K$-scheme, then it is also a smooth $k$-scheme, giving $\Gnum{2}$. We consider $V$ to be the trivial one-dimensional $G$-representation over $k$, giving $\Gnum{3}$. By construction, we have the following cartesian diagram of $G$-torsors:
	% https://q.uiver.app/#q=WzAsMTAsWzQsMCwiXFxTcGVjKEspIl0sWzQsMSwiXFxTcGVjKEYpIl0sWzMsMSwiXFx0aWxkZXtXfSJdLFszLDAsIlciXSxbMiwxLCJcXHRpbGRle1d9IFxcdGltZXNfayBcXEFBXjFfayJdLFsyLDAsIlcgXFx0aW1lc19rIFxcQUEoVikgIl0sWzEsMSwiVSJdLFsxLDAsIlgnIl0sWzAsMSwiWl9VIl0sWzAsMCwiWiciXSxbNyw1LCJcXHRpbGRle1xcdmFycGhpfSJdLFs2LDQsIlxcdmFycGhpIl0sWzcsNl0sWzUsNCwiXFx0ZXh0bm9ybWFse3ByfSBcXHRpbWVzIFxcdGV4dG5vcm1hbHtpZH0iXSxbMywyLCJcXHRleHRub3JtYWx7cHJ9Il0sWzQsMl0sWzUsM10sWzMsMF0sWzIsMV0sWzAsMV0sWzksNywiIiwwLHsic3R5bGUiOnsidGFpbCI6eyJuYW1lIjoiaG9vayIsInNpZGUiOiJ0b3AifX19XSxbOCw2LCIiLDAseyJzdHlsZSI6eyJ0YWlsIjp7Im5hbWUiOiJob29rIiwic2lkZSI6InRvcCJ9fX1dLFs5LDhdXQ==
	\begin{equation}\label{equation: diagram of torsors 2}\begin{tikzcd}
			{Z'} & {X'} & {W \times_k \AA(V) } & W & {\Spec(K)} \\
			{Z_U} & U & {\tilde{W} \times_k \AA^1_k} & {\tilde{W}} & {\Spec(F).}
			\arrow[hook, from=1-1, to=1-2]
			\arrow[from=1-1, to=2-1]
			\arrow["{\tilde{\varphi}}", from=1-2, to=1-3]
			\arrow[from=1-2, to=2-2]
			\arrow[from=1-3, to=1-4]
			\arrow["{\textnormal{pr} \times \textnormal{id}}", from=1-3, to=2-3]
			\arrow[from=1-4, to=1-5]
			\arrow["{\textnormal{pr}}", from=1-4, to=2-4]
			\arrow[from=1-5, to=2-5]
			\arrow[hook, from=2-1, to=2-2]
			\arrow["\varphi", from=2-2, to=2-3]
			\arrow[from=2-3, to=2-4]
			\arrow[from=2-4, to=2-5]
		\end{tikzcd}
	\end{equation}
	
	The morphism $\tilde{\varphi}$ gives $\Gnum{4}$, and we need to show that it satisfies properties $\Pnum{1}, \Pnum{2}, \Pnum{3}, \Pnum{4}$. Since $\varphi$ is étale, then also the base change $\tilde{\varphi}$ is étale, giving $\Pnum{1}$. Then observe that there exists a diagram of cartesian squares as follows:
	% https://q.uiver.app/#q=WzAsOCxbMCwxLCJaJyJdLFswLDMsIlpfVSJdLFsxLDIsIlgnIl0sWzEsNCwiVSJdLFszLDEsIlxcdGlsZGV7WH0iXSxbMiwwLCJcXHRpbGRle1p9Il0sWzMsMywiXFx0aWxkZXtYfS9HIl0sWzIsMiwiXFx0aWxkZXtafS9HIl0sWzAsMV0sWzIsM10sWzEsMywiIiwwLHsic3R5bGUiOnsidGFpbCI6eyJuYW1lIjoiaG9vayIsInNpZGUiOiJ0b3AifSwiYm9keSI6eyJuYW1lIjoiYmFycmVkIn19fV0sWzAsMiwiIiwwLHsic3R5bGUiOnsidGFpbCI6eyJuYW1lIjoiaG9vayIsInNpZGUiOiJ0b3AifSwiYm9keSI6eyJuYW1lIjoiYmFycmVkIn19fV0sWzIsNCwiIiwwLHsic3R5bGUiOnsidGFpbCI6eyJuYW1lIjoiaG9vayIsInNpZGUiOiJ0b3AifSwiYm9keSI6eyJuYW1lIjoiYnVsbGV0IGhvbGxvdyJ9fX1dLFswLDUsIiIsMCx7InN0eWxlIjp7InRhaWwiOnsibmFtZSI6Imhvb2siLCJzaWRlIjoidG9wIn0sImJvZHkiOnsibmFtZSI6ImJ1bGxldCBob2xsb3cifX19XSxbNSw3XSxbNSw0LCIiLDAseyJzdHlsZSI6eyJ0YWlsIjp7Im5hbWUiOiJob29rIiwic2lkZSI6InRvcCJ9LCJib2R5Ijp7Im5hbWUiOiJiYXJyZWQifX19XSxbNCw2XSxbMyw2LCIiLDAseyJzdHlsZSI6eyJ0YWlsIjp7Im5hbWUiOiJob29rIiwic2lkZSI6InRvcCJ9LCJib2R5Ijp7Im5hbWUiOiJidWxsZXQgaG9sbG93In19fV0sWzEsNywiIiwwLHsic3R5bGUiOnsidGFpbCI6eyJuYW1lIjoiaG9vayIsInNpZGUiOiJ0b3AifSwiYm9keSI6eyJuYW1lIjoiYnVsbGV0IGhvbGxvdyJ9fX1dLFs3LDYsIiIsMCx7InN0eWxlIjp7InRhaWwiOnsibmFtZSI6Imhvb2siLCJzaWRlIjoidG9wIn0sImJvZHkiOnsibmFtZSI6ImJhcnJlZCJ9fX1dXQ==
	\[\begin{tikzcd}
		&& {\tilde{Z}} & \\
		{Z'} &&& {\tilde{X}} \\
		& {X'} & {\tilde{Z}/G} \\
		{Z_U} &&& {\tilde{X}/G} \\
		& U
		\arrow["\shortmid"{marking}, hook, from=1-3, to=2-4]
		\arrow[from=1-3, to=3-3]
		\arrow["\bullet"{marking, text=\pgfkeysvalueof{/tikz/commutative diagrams/background color}}, "\circ"{marking}, hook, from=2-1, to=1-3]
		\arrow["\shortmid"{marking}, hook, from=2-1, to=3-2]
		\arrow[from=2-1, to=4-1]
		\arrow[from=2-4, to=4-4]
		\arrow["\bullet"{marking, text=\pgfkeysvalueof{/tikz/commutative diagrams/background color}}, "\circ"{marking}, hook, from=3-2, to=2-4]
		\arrow[from=3-2, to=5-2]
		\arrow["\shortmid"{marking}, hook, from=3-3, to=4-4]
		\arrow["\bullet"{marking, text=\pgfkeysvalueof{/tikz/commutative diagrams/background color}}, "\circ"{marking}, hook, from=4-1, to=3-3]
		\arrow["\shortmid"{marking}, hook, from=4-1, to=5-2]
		\arrow["\bullet"{marking, text=\pgfkeysvalueof{/tikz/commutative diagrams/background color}}, "\circ"{marking}, hook, from=5-2, to=4-4]
	\end{tikzcd}\]
	
	where all the vertical arrows are $G$-torsors. Since $\varphi_{|Z_U}$ is a closed immersion, then also the base change $\tilde{\varphi}_{|Z'}$ in diagram \ref{equation: diagram of torsors 2} is a closed immersion, giving $\Pnum{2}$. Since $\psi_{|Z_U} \colon Z_U \to \tilde{W}$ is finite, then also the base change $Z' \to  W$ is finite, giving $\Pnum{3}$. Finally, since $\varphi$ restricts to an isomorphism $\varphi^{-1}(\varphi(Z_U)) \xlongrightarrow{\sim} \varphi(Z_U)$, then we also have an isomorphism $\tilde{\varphi}^{-1}(\tilde{\varphi}(Z')) \simeq \tilde{\varphi}(Z')$, giving $\Pnum{4}$. Note that we do not need to provide a $G$-equivariant section of the morphism $X' \to W$ because $V$ is the trivial one-dimensional $G$-representation.
\end{proof}

\begin{claim} \label{claim: one-dimensional case}
	Suppose that $\dim(X) = 1$. Then Theorem \ref{theorem: gabber} holds under the assumption of case \ref{subsection: case A}.
\end{claim}
\begin{proof}
	Since $Z$ has positive codimension in $X$, we have $\dim(Z)=0$ and thus $z$ is a closed point. Then the field extension $\kappa(z)/k$ is separable because $k$ is a perfect field. By hypothesis of Case \ref{subsection: case A}, the $G$-action on $\kappa(z)$ is faithful, i.e.\ $G$ acts freely on $\Spec(\kappa(z))$. Consider the base change of $X$ to $\kappa(z)$:
	% https://q.uiver.app/#q=WzAsNCxbMCwwLCJYX3tcXGthcHBhKHopfSJdLFswLDEsIlxcU3BlYyhcXGthcHBhKHopKSJdLFsxLDAsIlgiXSxbMSwxLCJcXFNwZWMoaykiXSxbMSwzXSxbMCwxXSxbMCwyXSxbMiwzXV0=
	\[\begin{tikzcd}
		{X_{\kappa(z)}} & X \\
		{\Spec(\kappa(z))} & {\Spec(k).}
		\arrow["f", from=1-1, to=1-2]
		\arrow[from=1-1, to=2-1]
		\arrow[from=1-2, to=2-2]
		\arrow[from=2-1, to=2-2]
	\end{tikzcd}\]
	
	We claim that the morphism $f \colon X_{\kappa(z)} \to X$ is a $G$-equivariant Nisnevich neighborhood of $Gz=z$, so that the conclusion follows by Claim \ref{claim: trivial subcase}. Indeed, $f$ is a $G$-equivariant étale morphism, because $\kappa(z)/k$ is a finite separable field extension. Moreover, since $S_z = G$, we clearly have a section $s \colon G \times^{S_z} \Spec(\kappa(z)) \simeq \Spec(\kappa(z)) \to X_{\kappa(z)}$ of the point $z \colon \Spec(\kappa(z)) \to X$, as desired. Finally, note that $X_{\kappa(z)}$ is a smooth $\kappa(z)$-scheme and the morphism $f$ is separated, because is the base change of an affine morphism. Then the conclusion follows by Claim \ref{claim: trivial subcase}.
\end{proof}

\begin{remark}
	If $k$ is imperfect, the proof of Claim \ref{claim: one-dimensional case} does not go through.
\end{remark}

Recall that we assumed $X$ to be irreducible. By Claim \ref{claim: one-dimensional case} we may assume that $\dim X > 1$. Set $\dim X = n+1$ with $n \geq 1$. Let $z'$ be a closed specialization of $z$. In what follows, every time that we construct a scheme $X_i$ with a morphism $X_i \to X$ such that there exists a distinguished point $z_i \in X_i$ mapping to $z \in X$ and inducing an isomorphism on the residue fields, we will still denote by $z$ the point $z_i$. Moreover, we will denote by $Z_i $ the fiber product $Z_i \coloneq X_i \times_X Z$.

Let $V'$ be an $n$-dimensional faithful $G$-representation over $k$, which exists because $G$ is cyclic (of order $p^r$) and $k$ contains a primitive root of unity of order $p^r$ by hypotheses (we may, for example, take $n$ copies of a one-dimensional faithful representation).

\begin{claim}
	There exist an affine irreducible $G$-equivariant Nisnevich neighborhood $X_0 \to X$ of $Gz =z$ and a $G$-equivariant morphism $\varphi_0 \colon X_0 \to \AA(V')$ which is smooth of relative dimension one, the restriction $\varphi_{0_{|Z_0}} \colon Z_0 \to \AA(V')$ is quasi-finite and $\varphi_0(z) \neq 0$.
\end{claim}

\begin{proof}
	We show that the properties of smoothness, quasi-finiteness and non-vanishing on $z$ can be established independently, and we then conclude by invoking the openness of these properties.
	
	Let $V_e$ be the trivial one-dimensional $G$-representation over $k$. Set $\tilde{V} \coloneq V' \oplus V_e$. Since $G_{z'} = \left\{e\right\}$ and $\dim_{\kappa(z')}T_{z'}X = \dim_k(\tilde{V}) = n+1$, by Lemma \ref{lemma: exists representation 1}.(2) there exists an isomorphism $\textnormal{Res}^G_{S_{z'}}(\tilde{V}) \otimes_k \kappa(z') \simeq T_{z'}X$ in $\textnormal{Rep}^{\#}_{\kappa(z')}(S_{z'})$. Then by Proposition \ref{proposition: étale morphism as in 8.10 of equivariant cycles} there exists a $G$-equivariant morphism $f'_1 \colon X \to \AA(\tilde{V})$ which is étale at $z'$. In particular, the composite $f_1 \colon X \xrightarrow{f_1'} \AA(\tilde{V}) \xrightarrow{\textnormal{pr}} \AA(V')$ is smooth at $z'$. By Proposition \ref{proposition: quasi-finite morphisms construction freecase} there exists a morphism $f_2 \colon X \to \AA(V')$ whose restriction $f_{2_{|Z}} \colon Z \hookrightarrow X \rightarrow \AA(V')$ is quasi-finite at $z'$.
	We may apply Lemma \ref{lemma: generalization of Bachmann 2.5.(2,3)} to the one-dimensional $G$-subrepresentations of $V'$ in order to get a $G$-equivariant morphism $f_3 \colon X \to \AA(V')$ such that $f_3(z') \neq 0$, giving the third property. 
	
	We thus have a morphism $f_1 \colon X \to \AA(V')$ which is smooth at $z'$, a morphism $f_2 \colon X \to \AA(V')$ whose restriction to $Z$ is quasi-finite at $z'$ and a morphism $f_3 \colon X \to \AA(V')$ such that $f_3(z') \neq 0$.
	
	Let $\iota \colon X \hookrightarrow \AA(W)$ be a $G$-equivariant closed immersion of $X$ into a $G$-representation as in Remark \ref{remark: fix equivariant closed immersion}. Up to adding the equivariant elements defining the morphisms $f_1, f_2,f_3$ to the generators of the closed immersion $\iota$, we may assume that they are among them.
	
	By Lemma \ref{lemma: openness of properties}.(1),(2),(3), all of the three required properties defines a non-empty open subset of the affine space parametrizing $G$-equivariant linear maps $\AA(W) \to \AA(V')$: the maps $f_1, f_3, f_2$ respectively show non-emptiness. This parameter space is irreducible, so the intersection $U$ of these three non-empty open subsets is non-empty. Since $k$ is infinite, the non-empty open subset $U$ has a $k$-rational point. We conclude that there exists a morphism $f \colon X \to \AA(V')$ which is smooth at $z'$, $f(z') \neq 0$ and the restriction of $f$ to $Z$ is quasi-finite at $z'$. Since $\dim(X) = n+1$ and $\dim(\AA(V')) = n$, the morphism $f$ is smooth of relative dimension one. Moreover, by Lemma \ref{lemma: assume is étale and quasi-finite} we conclude that there exists an open $G$-equivariant neighborhood $U \to X$ of $Gz=z$ such that $f_{|U} \colon U \to \AA(V')$ is smooth and its restriction to $Z \cap U$ is quasi-finite. By Lemma \ref{lemma: assume is affine} we know that there exists an affine $G$-equivariant Nisnevich neighborhood $X_0 \to U$ of $Gz$. The composite $\varphi_0 \colon X_0 \to U \to \AA(V')$ is again smooth, being a composition of smooth morphisms and  $Z_{X_0} \to \AA(V')$ is again quasi-finite (indeed is quasi-compact and locally quasi-finite, because the property of being locally quasi-finite is étale local on the source-and-target, and $Z \cap U \to \AA(V')$ is quasi-finite), as desired.
	Since $f(z') \neq 0$, then also $f(z) \neq 0$, and in particular $\varphi_0(z) \neq 0$. It remains to show that we may assume $X_0$ to be also irreducible. Indeed,  $X_0$ is smooth over $k$, and thus it is the disjoint union of its irreducible components. Then there exists exactly one irreducible component $X_0'$ of $X_0$ containing $z$. Since $z$ is set-theoretically fixed, $X_0'$ is a $G$-invariant open subscheme of $X_0$. Since $X_0 \to X$ is étale and $X$ is irreducible of dimension $n+1$, then $\dim(X_0')= n+1$. After restricting the morphism $\varphi_0$ to $X_0'$ and replacing $X_0$ by $X_0'$, we may thus assume that $X_0$ is irreducible. In particular, $\varphi_0 \colon X_0 \to \AA(V')$ is a smooth morphism of relative dimension one, the restriction $\varphi_{0_{|Z_0}} \colon Z_0 \to \AA(V')$ is quasi-finite and $\varphi_0(z) \neq 0$.
\end{proof}

Set $w \coloneq \varphi_0(z) \neq 0$ and let $w \colon \Spec(\kappa(w)) \to \AA(V')$ be the associated $G$-equivariant morphism. Note that $w$ is set-theoretically fixed, but $\Spec(\kappa(w))$ has free $G$-action, because $V'$ is a direct sum of copies of a faithful one-dimensional character; hence every nonzero geometric point of $\AA(V')$ has trivial stabilizer. Set $W^h_w \coloneq \AA(V')^h_w$, and note that we have a lift $w \colon \Spec(\kappa(w)) \to W^h_w$ of the point $w$. Set $X_1 \coloneq X_0 \times_{\AA(V')} W^h_w$. Then the point $z \colon \Spec(\kappa(z)) \to X_0$ lifts to a point $z \colon \Spec(\kappa(z)) \to X_1$.
We will obtain conclusions of Theorem \ref{theorem: gabber} over the henselian scheme $W^h_w$ and then use a limit-descending argument to find the same conclusions on an actual model. The quasi-finite morphism $\varphi_{0_{|Z_0}} \colon Z_0 \to \AA(V')$ induces a quasi-finite morphism $Z_1 \to W^h_w$. Let $V_e$ be the trivial one-dimensional $G$-representation. For the remaining part of Case \ref{subsection: case A}, we write $\AA^1_k = \AA(V_e)$ to denote the trivial one-dimensional $G$-representation over $k$.

\begin{claim} \label{claim: X_2}
	There exists an affine open $G$-invariant subscheme $X_2$ of $X_1$ containing $z$ and a $G$-equivariant morphism $\varphi_2 \colon X_2 \to W^h_w \times_k \AA^1_k$ such that
	\begin{enumerate}
		\item the composite $Z_2 \hookrightarrow X_2 \to W^h_w \times_k \AA^1_k \to  W^h_w$ is finite;
		\item $Z_2$ is the spectrum of a local ring and $z \in Z_2$ is its unique closed point;
		\item the fiber $(\varphi_2)_{\kappa(w)} \colon (X_2)_{\kappa(w)} \to \AA^1_{\kappa(w)}$ of $\varphi_2$ over $w$ is étale at $z$ and radicial at $z$.
	\end{enumerate}
\end{claim}

\begin{proof}
	We first show that there exists an affine open $G$-invariant subscheme $X_2$ of $X_1$ containing $z$ such that $Z_2$ satisfies (2) and such that the composite $Z_2 \to X_2 \to X_1 \to  W^h_w$ is finite. In order to do that, we use Lemma \ref{lemma: pass from quasi-finite to finite}. Indeed, we can apply Lemma \ref{lemma: pass from quasi-finite to finite} to the schemes $X_1$ and $W^h_w$ to obtain an open $G$-invariant subscheme $X_2 \subset X_1$ containing $z$ such that the composite $Z_2 \hookrightarrow X_2 \to W^h_w$ is finite and such that $Z_{2}$ is local with unique closed point $z$. In order to obtain an affine $G$-invariant neighborhood, we do the following. Let $U \subseteq X_2$ be any open affine neighborhood of $z$ in $X_2$. Set $\tilde{U} \coloneq \bigcap_{g \in G} g(U)$. We claim that $\tilde{U}$ is an open non-empty affine $G$-invariant subscheme of $X_2$. First note that it is non-empty because $z$ is set-theoretically fixed by the $G$-action. Then it is sufficient to show that $X_2$ is separated, because then the intersection of a finite number of affine open subsets in $X_2$ is affine. To see this, note that $X_1$ is separated because it is affine, and $X_2$ is open in $X_1$, hence is separated. Since $Z_{2}$ is local, we have $Z_{2} \subset \tilde{U}$, and thus $Z_{\tilde{U}} = Z_2 \to W^h_w$ is still finite. After replacing $X_2$ by $\tilde{U}$ we may thus assume that $X_2$ is affine.
	
	Next we produce a morphism $\varphi_2 \colon X_2 \to W^h_w \times_k \AA^1_k$ over $W^h_w$ satisfying property (3). In order to do that, we verify that the hypotheses of Lemma \ref{lemma: exists a fiber that is étale and radicial at a point} are satisfied by the field extension $\kappa(w)/k$ and the affine scheme $(X_2)_{\kappa(w)} = X_2 \times_{W^h_w}\Spec(\kappa(w))$. The $G$-action on $\kappa(w)$ is faithful because $\Spec(\kappa(w))$ has free $G$-action. Since $\varphi_0 \colon X_0 \to \AA(V')$ is smooth of relative dimension one, then also $X_1 \to W^h_w$ is smooth of relative dimension one. Since $X_2 \subset X_1$ is an open subscheme of $X_1$, then also $X_2 \to W^h_w$ is smooth of relative dimension $1$ by \cite[\href{https://stacks.math.columbia.edu/tag/02NL}{Tag 02NL}]{stacks-project}. In particular, also the fiber morphism $(X_2)_{\kappa(w)} \to \Spec(\kappa(w))$ is smooth of relative dimension $1$. It follows that $(X_2)_{\kappa(w)}$ is an affine $\kappa(w)$-scheme of pure dimension one. Since $X_0$ has free $G$-action, then also $(X_2)_{\kappa(w)}$ has free $G$-action. Finally, since $Z_{2}$ has unique closed point $z$, it follows that $z$ is a closed point of $X_2$.
	
	We may thus apply Lemma \ref{lemma: exists a fiber that is étale and radicial at a point} to conclude that there exists a $G$-equivariant morphism $f_w \colon (X_2)_{\kappa(w)} \to \AA^1_{\kappa(w)}$ over $\kappa(w)$ which is étale at $z$ and radicial at $z$. Then we can apply Lemma \ref{lemma: lift the morphism} to conclude that there exists a $G$-equivariant morphism $\varphi_2 \colon X_2 \to W^h_w \times_k \AA^1_k$ such that the fiber $(\varphi_2)_{\kappa(w)} \colon (X_2)_{\kappa(w)} \to \AA^1_{\kappa(w)}$ over $w \colon \Spec(\kappa(w)) \to W^h_w$ is étale at $z$ and radicial at $z$.
\end{proof}

Recall that the morphism $(\varphi_2)_{\kappa(w)}$ is given by the following cartesian diagram
% https://q.uiver.app/#q=WzAsNixbMSwwLCJYIl0sWzEsMSwiXFxBQV4xX1ciXSxbMCwxLCJcXEFBXjFfdyJdLFsxLDIsIlciXSxbMCwyLCJcXFNwZWMoXFxrYXBwYSh3KSkiXSxbMCwwLCJYX3ciXSxbMCwxLCJcXHZhcnBoaSJdLFsxLDNdLFs0LDMsInciLDJdLFsyLDRdLFsyLDFdLFs1LDIsIlxcdmFycGhpX3ciLDJdLFs1LDBdXQ==
\[\begin{tikzcd}
	{(X_2)_{\kappa(w)}} & X_2 \\
	{\AA^1_{\kappa(w)}} & {\AA^1_{W^h_w}} \\
	{\Spec(\kappa(w))} & {W^h_w.}
	\arrow[from=1-1, to=1-2]
	\arrow["{(\varphi_2)_{\kappa(w)}}"', from=1-1, to=2-1]
	\arrow["\varphi_2", from=1-2, to=2-2]
	\arrow[from=2-1, to=2-2]
	\arrow[from=2-1, to=3-1]
	\arrow[from=2-2, to=3-2]
	\arrow["w"', from=3-1, to=3-2]
\end{tikzcd}\]

\begin{claim} \label{claim: X_2 étale}
	The morphism $\varphi_2 \colon X_2 \to \AA^1_{W^h_w}$ is étale at $z$.
\end{claim}
\begin{proof}
	We will use the criterion \cite[\href{https://stacks.math.columbia.edu/tag/02GU}{Tag 02GU}]{stacks-project}. First note that $\varphi_2$ is locally of finite presentation. Indeed, $X_2 \to W^h_w$ is locally of finite presentation, being the composite of the open immersion $X_2 \to X_1$ and of the smooth morphism $X_1 \to W^h_w$. Also the morphism $\AA^1_{W^h_w} \to W^h_w$ is of finite presentation. Thus $\varphi_2 \colon X_2 \to \AA^1_{W^h_w}$ is locally of finite presentation by \cite[\href{https://stacks.math.columbia.edu/tag/02FV}{Tag 02FV}]{stacks-project}. Set $u = \varphi_2(z) \in \AA^1_{W^h_w}$. The fiber of $\varphi_2$ over $u$ is given by the following cartesian diagram:
	% https://q.uiver.app/#q=WzAsNixbMCwwLCJmXnstMX0oeikiXSxbMiwwLCJYXzIiXSxbMiwxLCJBQV4xX1ciXSxbMSwxLCJcXEFBXjFfayBcXHRpbWVzX2sgXFxTcGVjKFxca2FwcGEodykpIl0sWzEsMCwiKFhfMilfdyJdLFswLDEsIlxcU3BlYyhcXGthcHBhKHMpKSJdLFswLDVdLFs1LDNdLFs0LDMsImZfdyJdLFsxLDIsImYiXSxbMywyXSxbNCwxXSxbMCw0XV0=
	\[\begin{tikzcd}
		{(X_2)_{\kappa(u)}} & {(X_2)_{\kappa(w)}} & {X_2} \\
		{\Spec(\kappa(u))} & {\mathbb{A}^1_{\kappa(w)}} & {\AA^1_{W^h_w}}.
		\arrow[ from=1-1, to=1-2]
		\arrow["(\varphi_2)_{\kappa(u)}", from=1-1, to=2-1]
		\arrow[from=1-2, to=1-3]
		\arrow["{(\varphi_2)_{\kappa(w)}}", from=1-2, to=2-2]
		\arrow["\varphi_2", from=1-3, to=2-3]
		\arrow[from=2-1, to=2-2]
		\arrow[from=2-2, to=2-3]
	\end{tikzcd}\] 
	Since by Claim \ref{claim: X_2} $(\varphi_2)_{\kappa(w)}$ is étale at $z$, then by \cite[\href{https://stacks.math.columbia.edu/tag/02GU}{Tag 02GU, (1)$\iff$(2)}]{stacks-project} the morphism $(\varphi_2)_{\kappa(u)}$ is étale at $z$. Then, again by \cite[\href{https://stacks.math.columbia.edu/tag/02GU}{Tag 02GU, (1)$\iff$(2)}]{stacks-project}, we see that it is sufficient to show that the ring homomorphism $\cO_{\AA^1_{W^h_w}, u} \to \cO_{X_2, z}$ is flat. Since $X_2$ is smooth over $W^h_w$, $\cO_{X_2, z}$ is flat over $\cO_{W^h_w,w}$. Also, $\cO_{(X_2)_{\kappa(w)}, z}$ is flat over $\cO_{\AA^1_{\kappa(w)}, u}$ because $(\varphi_2)_{\kappa(w)}$ is étale at $z$. Therefore, once we know that $X_2$ and $W^h_w$ are noetherian, we may apply the fiberwise criterion for flatness \cite[\href{https://stacks.math.columbia.edu/tag/039B}{Tag 039B, (1) $\Rightarrow$ (2)}]{stacks-project} to conclude that $\cO_{X_2, z}$ is flat over $\cO_{\AA^1_{W^h_w},u}$. We argue that $X_2$ and $W^h_w$ are noetherian as follows. Since $\AA(V')$ is noetherian, by \cite[Théorème 18.6.6]{EGAcapitoloIVparteIV} the scheme $W^h_w = \Spec(\cO_{\AA(V'), w}^h)$ is noetherian. Moreover, $X$ is noetherian, and thus $X_0$ is noetherian. Since $\varphi_0 \colon X_0 \to \AA(V')$ is smooth, the base change $X_1 \to W^h_w$ is a smooth morphism and since $X_1$ is affine, in particular is noetherian. Finally, $X_2$ is open in $X_1$, and hence is noetherian. Also $\AA^1_{W^h_w}$ is noetherian by the Hilbert basis theorem. 
\end{proof}

Note that Claim \ref{claim: X_2} and Claim \ref{claim: X_2 étale} give properties $\Pnum{1}$ and $\Pnum{3}$ of a Gabber's presentation.

Let $X_3 \subset X_2$ be an open $G$-invariant subscheme such that the restriction $\varphi_3 \colon X_3 \to X_2 \xrightarrow{\varphi_2} \AA^1_{W^h_w}$ is étale. Since $Z_2$ is local, we have $Z_2 \subset X_3$. Then $Z_3 = Z_2$ and thus the composite $ Z_3 \xrightarrow{\varphi_{3_{|Z_3}}} \AA^1_{W^h_w} \to W^h_w$ is still finite.

\begin{claim} \label{claim: closed immersion}
	The restriction $\varphi_{3_{|Z_3}} \colon Z_3 \to \AA^1_{W^h_w}$ is a closed immersion.
\end{claim}

\begin{proof}
	It is sufficient to show that hypotheses of Lemma \ref{lemma: nakayama's lemma for closed immersion} are satisfied. Indeed, both $Z_3$ and $W_w^h$ are spectra of local $k$-algebras; the morphism $\varphi_{3_{|Z_3}}$ is unramified, because it is the composite of a closed immersion $Z_3 \subseteq X_3$ and an étale morphism $\varphi_3 \colon X_3 \to \AA^1_{W^h_w}$. The composite $Z_3 \to W^h_w$ is finite and the unique closed point $z \in Z_3$ maps to the unique closed point $w \in W^h_w$. The morphism $Z_3 \to W^h_w$ is finite, and thus the underlying topological space of the fiber $F$ over $w$ consists of a single point by Lemma \ref{lemma: closed fiber of a finite local morphism}. Moreover, the induced morphism $F \to \AA^1_{\kappa(w)}$ is radicial by Claim \ref{claim: X_2}. Then $\varphi_{3_{|Z_3}}$ is a closed immersion by Lemma \ref{lemma: nakayama's lemma for closed immersion}.
\end{proof}

Note that Claim \ref{claim: closed immersion} gives property $\Pnum{2}$ of a Gabber's presentation.

\begin{claim}
	There exists an open $G$-invariant subscheme $X_4 \subset X_3$ and an étale morphism $\varphi_4 \colon X_4 \to \AA^1_{W^h_w}$ such that $\varphi_{4_{|Z_4}}$ is a closed immersion, $Z_4$ is finite over $W^h_w$ and $\varphi_4$ induces an isomorphism of $Z_4$ with its image.
\end{claim}

\begin{proof}
	By Lemma \ref{lemma: exist open with isomorphism in the image} there exists an open $G$-invariant subscheme $X_4 \subset X_3$ such that the restriction $\varphi_4 \colon X_4 \to X_3 \xrightarrow{\varphi_3} \AA^1_{W^h_w}$
	induces an isomorphism $\varphi_4^{-1}(\varphi_4(Z_4)) \xrightarrow{\sim} \varphi_4(Z_4)$, where $Z_4=Z_3$. By construction, $\varphi_4 \colon X_4 \to \AA^1_{W^h_w}$ is a $G$-equivariant étale morphism with all the required properties.
\end{proof}

Recall that $\AA^1_k$ denotes the trivial one-dimensional $G$-representation.
The scheme $X_4$ is by construction an open $G$-invariant subscheme of $X_1 = X_0 \times_{\AA(V')} W^h_w$.

To summarize, we have obtained the conclusions $\Pnum{1}, \Pnum{2}, \Pnum{3}$ and $\Pnum{4}$ of Definition \ref{definition: gabber presentation} over the henselization $W^h_w$. In Claim \ref{claim: produce actual model} we will find an actual model with a limit-descending argument.

\begin{remark} \label{remark: schemes of finite presentation}
	The schemes $\AA^1_{W^h_w}, X_1, X_4, Z_{4}, \varphi_4(Z_{4}), \varphi_4^{-1}(\varphi_4(Z_{4}))$ are of finite presentation over $W^h_w$. Indeed $\AA^1_{W^h_w} \to W^h_w$ is the base change of $\AA^1_k \to \Spec(k)$ which is a morphism of finite presentation. The morphism $\varphi_4 \colon X_4 \to \AA^1_{W^h_w}$ is the composite of an open immersion $X_4 \subset X_2$ with target a noetherian affine scheme, and of an étale morphism $ \varphi_2 \colon X_2 \to \AA^1_{W^h_w}$ with affine source. It follows that $X_4$ is finitely presented over $W^h_w$. The scheme $X_1$ is by definition the fiber product $X_1 = X_0 \times_{\AA(V')} W^h_w$; hence $X_1$ is an affine smooth scheme over the noetherian scheme $W^h_w$. In particular, the morphism $X_1 \to W^h_w$ is of finite presentation. Then $Z_{4} \to X_4$ is a closed immersion into a noetherian scheme, and thus it is a morphism of finite presentation and the composite $Z_{4} \to X_4 \to W^h_w$ is too. Also $\varphi_4(Z_{4})$ is closed in $\AA^1_{W^h_w}$, which is a noetherian scheme. Hence $\varphi_4(Z_{4})$ is also finitely presented over $W^h_w$. Finally, $\varphi_4^{-1}(\varphi_4(Z_{4}))$ is isomorphic to $\varphi_4(Z_{4})$ by construction.
\end{remark}

\begin{claim} \label{claim: produce actual model}
	There exists a $G$-equivariant Gabber's presentation $(X', W, V , \varphi)$ of the $G$-pair $(X,Z)$ with respect to the point $z$ such that the one-dimensional $G$-representation $V$ in $\Gnum{3}$ is the trivial one-dimensional $G$-representation.
\end{claim}

\begin{proof}
	We will use a limit-descending argument to obtain conclusions $\Pnum{1}, \Pnum{2}, \Pnum{3}$ and $\Pnum{4}$ of Definition \ref{definition: gabber presentation} over an actual model.
	Write $W^h_w = \lim_{ U \in N_{G}(Gw)}U$ as the limit over the filtering category of affine equivariant Nisnevich neighborhoods of $Gw$ as in \cite[\S 2]{MotivicHomotopyTheoryOfGroupSchemesActions}. Since $\AA(V')$ is affine, by \cite[\S 2]{MotivicHomotopyTheoryOfGroupSchemesActions} there is an isomorphism $W^h_w \simeq \Spec(\cO_{\AA(V'),w}^h)$.
	Since $\textnormal{Sch}^G_k$ is essentially small, the category
	$N_G(Gw)$ of equivariant Nisnevich neighborhoods is essentially small as well. Replacing it by a small skeleton, we may apply \cite[Theorem 1.5]{PresentableCategories} and write $W^h_w= \lim_{i \in I } U_i$ as the limit of an inverse system of schemes over a directed set $I$. All the required properties descend  through limits of morphisms of finite presentation. More precisely we do the following.
	
	Since $X_4 \to X_1 = X_0 \times_{\AA(V')} W^h_w$ is an open immersion, by Lemma \ref{lemma: limit descending}, there exists an index $i \in I$, a scheme $S_i$ with a $G$-action and a $G$-equivariant morphism $a_i \colon S_i \to X_0 \times_{\AA(V')} U_i$ which is an open immersion, fitting into a cartesian diagram
	% https://q.uiver.app/#q=WzAsNCxbMCwwLCJYXzQiXSxbMSwwLCJYX2oiXSxbMSwxLCJYIFxcdGltZXNfe1xcQUFebl9rfVVfaiJdLFswLDEsIlggXFx0aW1lc197XFxBQV5uX2t9IFdeaF93Il0sWzMsMl0sWzEsMiwiYV9qIl0sWzAsM10sWzAsMV1d
	\begin{equation} \label{equation: first cartesian diagram}
	\begin{tikzcd}
		{X_4} & {S_i} \\
		{X_0 \times_{\AA(V')} W^h_w} & {X_0 \times_{\AA(V')} U_i.}
		\arrow[from=1-1, to=1-2]
		\arrow[from=1-1, to=2-1]
		\arrow["{a_i}", from=1-2, to=2-2]
		\arrow[from=2-1, to=2-2]
	\end{tikzcd}
	\end{equation}
	
	We first show how to obtain $\Pnum{1}$. Since $\varphi_4 \colon X_4 \to \AA^1_{W^h_w}$ is étale, by Lemma \ref{lemma: limit descending}, there exists an index $j \in I$, a scheme $S_j$ with a $G$-action and a $G$-equivariant morphism $b_j \colon S_j \to \AA^1_{U_j}$ which is étale, fitting into a cartesian diagram
	% https://q.uiver.app/#q=WzAsNCxbMCwwLCJYXzQiXSxbMSwwLCJYX2kiXSxbMSwxLCJcXEFBXjFfe1VfaX0iXSxbMCwxLCJcXEFBXjFfe1deaF93fSJdLFszLDJdLFsxLDIsImEiXSxbMCwzLCJcXHZhcnBoaSIsMl0sWzAsMV1d
	\begin{equation} \label{equation: second cartesian diagram}
	\begin{tikzcd}
		{X_4} & {S_j} \\
		{\AA^1_{W^h_w}} & {\AA^1_{U_j}.}
		\arrow[from=1-1, to=1-2]
		\arrow["\varphi_4"', from=1-1, to=2-1]
		\arrow["b_j", from=1-2, to=2-2]
		\arrow[from=2-1, to=2-2]
	\end{tikzcd}
	\end{equation}
	After increasing the index, we may assume $j \geq i$. We rename $j = i$ and thus we have a scheme $S_i$ and morphisms $a_i, b_i$ as in diagram \ref{equation: first cartesian diagram} and diagram \ref{equation: second cartesian diagram}.
	
	We now show how to obtain property $\Pnum{2}$. Set $V_i \coloneq S_i \times_X Z$, so that we have a cartesian diagram as follows:
	% https://q.uiver.app/#q=WzAsMTQsWzEsMCwiVl9pIl0sWzEsMSwiU19pIl0sWzIsMSwiWF8wIFxcdGltZXNfe1xcQUEoVicpfVVfaSJdLFsyLDAsIlpfMCBcXHRpbWVzX3tcXEFBKFYnKX1VX2kiXSxbMywwLCJaXzAiXSxbMywxLCJYXzAiXSxbNCwxLCJYIl0sWzQsMCwiWiJdLFswLDEsIlhfNCJdLFswLDAsIlpfNCJdLFsxLDIsIlxcQUFeMV97VV9pfSJdLFswLDIsIlxcQUFeMV97V157aH1fd30iXSxbMSwzLCJVX2kiXSxbMCwzLCJXXntofV93Il0sWzcsNl0sWzUsNl0sWzIsNV0sWzQsN10sWzMsMl0sWzAsMSwiY19pIiwyXSxbMSwyLCJhX2kiLDJdLFswLDNdLFszLDRdLFs0LDVdLFs5LDhdLFs4LDFdLFs5LDBdLFs4LDExLCJcXHZhcnBoaV80IiwyXSxbMTEsMTBdLFsxLDEwLCJiX2kiLDJdLFsxMywxMl0sWzEwLDEyXSxbMTEsMTNdXQ==
	\[\begin{tikzcd}
		{Z_4} & {V_i} & {Z_0 \times_{\AA(V')}U_i} & {Z_0} & Z \\
		{X_4} & {S_i} & {X_0 \times_{\AA(V')}U_i} & {X_0} & X \\
		{\AA^1_{W^{h}_w}} & {\AA^1_{U_i}} \\
		{W^{h}_w} & {U_i.}
		\arrow[from=1-1, to=1-2]
		\arrow[from=1-1, to=2-1]
		\arrow[from=1-2, to=1-3]
		\arrow["{c_i}"', from=1-2, to=2-2]
		\arrow[from=1-3, to=1-4]
		\arrow[from=1-3, to=2-3]
		\arrow[from=1-4, to=1-5]
		\arrow[from=1-4, to=2-4]
		\arrow[from=1-5, to=2-5]
		\arrow[from=2-1, to=2-2]
		\arrow["{\varphi_4}"', from=2-1, to=3-1]
		\arrow["{a_i}"', from=2-2, to=2-3]
		\arrow["{b_i}"', from=2-2, to=3-2]
		\arrow[from=2-3, to=2-4]
		\arrow[from=2-4, to=2-5]
		\arrow[from=3-1, to=3-2]
		\arrow[from=3-1, to=4-1]
		\arrow[from=3-2, to=4-2]
		\arrow[from=4-1, to=4-2]
	\end{tikzcd}\]
	 Since $\varphi_{4_{|Z_4}} \colon Z_4 \to \AA^1_{W^h_w}$ is a closed immersion, again by Lemma \ref{lemma: limit descending}, after increasing the index, we may assume that $b_i \circ c_i \colon V_i \to \AA^1_{U_i}$ is a closed immersion.
	Next we show how to obtain property $\Pnum{3}$. Since $Z_4 \to W^h_w$ is finite, again by Lemma \ref{lemma: limit descending}, after increasing the index, we may assume that $V_i \to U_i$ is finite.
	
	Finally, we obtain property $\Pnum{4}$ as follows. The morphism $b_{i} \colon S_i \to \AA^1_{U_i}$ induces $b_i^{-1}(b_i(V_i)) \to b_i(V_i)$ which is an isomorphism on the base change over $W^h_w$, because $\varphi_4^{-1}(\varphi_4(Z_4)) \simeq \varphi_4(Z_4)$. Then, after increasing the index $i$ again, we may assume by Lemma \ref{lemma: limit descending} that $b_i^{-1}(b_i(V_i)) \xrightarrow{\sim} b_i(V_i)$ is an isomorphism.

	Set $X' \coloneq S_i$, $W \coloneq U_i$ and $\varphi \coloneq b_i \colon X' \to \AA^1_{W}$. Then $Z_{X'} = X' \times_X Z = V_i$. Note that $z \in X'$ because $z \in X_4$. Let $V$ be the trivial one-dimensional $G$-representation. We claim that $(X',W, V, \varphi)$ forms a $G$-equivariant Gabber's presentation with respect to the point $z$.
	
	In the following diagram
	% https://q.uiver.app/#q=WzAsNSxbMSwwLCJYIFxcdGltZXNfe1xcQUFebl9rfVciXSxbMSwxLCJYIl0sWzIsMSwiXFxBQV5uX2siXSxbMiwwLCJXIl0sWzAsMCwiWCciXSxbMCwxLCJjIl0sWzEsMl0sWzMsMiwiYiJdLFswLDNdLFs0LDAsIiIsMCx7InN0eWxlIjp7InRhaWwiOnsibmFtZSI6Imhvb2siLCJzaWRlIjoidG9wIn0sImJvZHkiOnsibmFtZSI6ImJ1bGxldCBob2xsb3cifX19XV0=
	\[\begin{tikzcd}
		{X'} & {X_0 \times_{\AA(V')} W} & W \\
		& X_0 & {\AA(V')}
		\arrow["\bullet"{marking, text=\pgfkeysvalueof{/tikz/commutative diagrams/background color}}, "\circ"{marking}, hook, from=1-1, to=1-2]
		\arrow[from=1-2, to=1-3]
		\arrow["\alpha", from=1-2, to=2-2]
		\arrow["\beta", from=1-3, to=2-3]
		\arrow[from=2-2, to=2-3]
	\end{tikzcd}\]
	the morphism $\beta$ is, by definition of $W$, an equivariant Nisnevich neighborhood of $Gw=w$; hence $\alpha$ is an equivariant Nisnevich neighborhood of $Gz=z$. Indeed the morphisms
	\[
	\Spec(\kappa(z)) \xlongrightarrow{z} X_0 \qquad \textnormal{and} \qquad \Spec(\kappa(z)) \xlongrightarrow{w}  \Spec(\kappa(w)) \longrightarrow W
	\]
	give rise to the section $\Spec(\kappa(z)) \to X_0 \times_{\AA(V')} W$, because they form a commutative diagram over $\AA(V')$.
	In particular, also $X' \to X_0$ is a $G$-equivariant Nisnevich neighborhood of $z$, because $X'$ is open in $X_0 \times_{\AA(V')} W$ and it contains $z$. Since $X_0 \to X$ is a $G$-equivariant Nisnevich neighborhood of $Gz=z$, then also the composite $X' \to X_0 \to X$ is a $G$-equivariant Nisnevich neighborhood of $Gz=z$, giving $\Gnum{1}$. The scheme $W$ is a smooth $k$-scheme endowed with a $G$-action, giving $\Gnum{2}$. The datum $\Gnum{3}$ is given by the trivial one-dimensional $G$-representation $V$. The datum $\Gnum{4}$ is given by the $G$-equivariant morphism $\varphi \colon X' \to \AA^1_W \simeq W \times_k \AA(V)$. By construction, $\varphi$ is étale, giving $\Pnum{1}$; the restriction to $Z_{X'}$ is a closed immersion, giving $\Pnum{2}$; $Z_{X'}$ is finite over $W$, giving $\Pnum{3}$; and $\varphi$ restricts to an isomorphism $\varphi^{-1}(\varphi(Z_{X'})) \xrightarrow{\sim} \varphi(Z_{X'})$, giving $\Pnum{4}$.
\end{proof}

This concludes the proof of the case $G_z=\left\{e\right\}$ and $S_z = G$.

\subsection{The case when $G_z = \left\{e\right\}$ and $S_z \subsetneq G$} \label{subsection: case B}

We now prove the case in which $G_z = \left\{e\right\}$ and the set-theoretic stabilizer $S_z \subsetneq G$ is a strictly smaller subgroup of $G$.

First note that the morphism $e \colon G \times^{S_z} X \to X$, mapping $[(g,x)] \mapsto gx$ is a $G$-equivariant Nisnevich neighborhood of $Gz$, where the scheme $X$ on the left is endowed with the $S_z$-action coming from the restriction of the $G$-action. Indeed the $S_z$-equivariant morphism $\Spec(\kappa(z)) \to X$ induces a $G$-equivariant morphism $s \colon G \times^{S_z} \Spec(\kappa(z)) \to G \times^{S_z} X$ which fits into a $G$-equivariant commutative diagram as follows:

% https://q.uiver.app/#q=WzAsMyxbMCwxLCJHIFxcdGltZXNee1Nfen1cXFNwZWMoXFxrYXBwYSh6KSkiXSxbMSwxLCJYIl0sWzEsMCwiRyBcXHRpbWVzXntTX3p9IFgiXSxbMiwxLCJlIl0sWzAsMV0sWzAsMiwicyJdXQ==
\[\begin{tikzcd}
	& {G \times^{S_z} X} \\
	{G \times^{S_z}\Spec(\kappa(z))} & X.
	\arrow["e", from=1-2, to=2-2]
	\arrow["s", from=2-1, to=1-2]
	\arrow[from=2-1, to=2-2]
\end{tikzcd}\]

Considering the $S_z$-equivariant setting, the point $z \in X$ has again free action but is set-theoretically fixed (with respect to $S_z$). In that case, we proved in the previous section that Theorem \ref{theorem: gabber} holds in the $S_z$-equivariant setting. It follows that there exists an $S_z$-equivariant Gabber's presentation $(\tilde{X}, \tilde{W}, V, \tilde{\varphi})$ with respect to the point $z$ as in Definition \ref{definition: gabber presentation}. More precisely there exist:
\begin{itemize}
	\item a $S_z$-equivariant Nisnevich neighborhood $l \colon  \tilde{X} \to X$ of $S_zz=z$;
	\item a smooth $k$-scheme $\tilde{W}$ endowed with an $S_z$-action;
	\item a one-dimensional $S_z$-representation $V$ which is the trivial one-dimensional $S_z$-representation by Claim \ref{claim: produce actual model};
	\item a $S_z$-equivariant morphism $\tilde{\varphi} \colon \tilde{X} \to \tilde{W} \times_k \AA(V)$ satisfying $\Pnum{1}, \Pnum{2}, \Pnum{3}$ and $\Pnum{4}$ of Definition \ref{definition: gabber presentation}.
\end{itemize}
Recall that $S_z z = S_z \times^{S_z} \Spec(\kappa(z)) \simeq \Spec (\kappa(z))$ endowed with its original free $S_z$-action. The $S_z$-equivariant morphism $l \colon \tilde{X} \to X$ induces a $G$-equivariant morphism $l' \colon G \times^{S_z} \tilde{X} \to G \times^{S_z} X$ mapping $[g, \tilde{x}]$ to $[g, l(\tilde{x})]$.

\begin{claim}
	The composite $G \times^{S_z} \tilde{X} \xrightarrow{l'} G \times^{S_z} X \xrightarrow{e} X$ is a $G$-equivariant Nisnevich neighborhood of $Gz$.
\end{claim}
\begin{proof}
	We already know that $e \colon G \times^{S_z} X \to X$ is a $G$-equivariant Nisnevich neighborhood of $Gz$. We thus only need to show that $l' \colon G \times^{S_z} \tilde{X} \to G \times^{S_z} X$ is a $G$-equivariant Nisnevich neighborhood of $Gz$. We already know that $l'$ is a $G$-equivariant morphism. We show that $l'$ is étale as follows. The morphism $l \colon \tilde{X} \to X$ is étale, and thus also the pullback morphism $\textnormal{id}_G \times l \colon G \times_k \tilde{X} \to G \times_k X$ is étale. The $S_z$-action on the schemes $G \times_k \tilde{X}$ and $G \times_k X$ is free. Thus  we have a cartesian square of $S_z$-torsors
	% https://q.uiver.app/#q=WzAsNCxbMCwwLCJHIFxcdGltZXNfayBcXHRpbGRle1h9Il0sWzEsMCwiRyBcXHRpbWVzX2sgWCJdLFswLDEsIkcgXFx0aW1lc157U196fVxcdGlsZGV7WH0iXSxbMSwxLCJHIFxcdGltZXNee1Nfen1YIl0sWzEsM10sWzIsM10sWzAsMl0sWzAsMV1d
	\[\begin{tikzcd}
		{G \times_k \tilde{X}} & {G \times_k X} \\
		{G \times^{S_z}\tilde{X}} & {G \times^{S_z}X,}
		\arrow["\textnormal{id}_G \times l", from=1-1, to=1-2]
		\arrow[from=1-1, to=2-1]
		\arrow[from=1-2, to=2-2]
		\arrow["l'", from=2-1, to=2-2]
	\end{tikzcd}\]
	where the vertical arrows are faithfully flat and quasi-compact (even étale), because they are the quotient morphisms. Then the morphism $l'$ is étale, because étaleness is an fpqc local property on the base.
	
	It remains to show that there exists a $G$-equivariant section from $G \times^{S_z} \Spec(\kappa(z))$ to $G \times^{S_z} \tilde{X}$. By construction, there exists an $S_z$-equivariant section $\tilde{s}$ fitting into a commutative triangle as follows: 
	% https://q.uiver.app/#q=WzAsMyxbMCwxLCJcXFNwZWMoXFxrYXBwYSh6KSkiXSxbMSwwLCJcXHRpbGRle1h9Il0sWzEsMSwiWCJdLFswLDJdLFswLDEsIlxcdGlsZGV7c30iXSxbMSwyLCJsIl1d
	\[\begin{tikzcd}
		& {\tilde{X}} \\
		{\Spec(\kappa(z))} & X.
		\arrow["l", from=1-2, to=2-2]
		\arrow["{\tilde{s}}", from=2-1, to=1-2]
		\arrow[from=2-1, to=2-2]
	\end{tikzcd}\]
	Then we obtain a $G$-equivariant section on contracted products
	% https://q.uiver.app/#q=WzAsMyxbMCwxLCJHIFxcdGltZXNee1Nfen1cXFNwZWMoXFxrYXBwYSh6KSkiXSxbMSwwLCJHIFxcdGltZXNee1Nfen1cXHRpbGRle1h9Il0sWzEsMSwiRyBcXHRpbWVzXntTX3p9WCJdLFswLDJdLFswLDEsInMiXSxbMSwyLCJsJyJdXQ==
	\[\begin{tikzcd}
		& {G \times^{S_z}\tilde{X}} \\
		{G \times^{S_z}\Spec(\kappa(z))} & {G \times^{S_z}X,}
		\arrow["{l'}", from=1-2, to=2-2]
		\arrow["s", from=2-1, to=1-2]
		\arrow[from=2-1, to=2-2]
	\end{tikzcd}\]
	as desired.
\end{proof}

Set $X' \coloneq G \times^{S_z} \tilde{X}$ and $W \coloneq G \times^{S_z}\tilde{W}$. Then $X'$ and $W$ provide data $\Gnum{1}$ and $\Gnum{2}$. The next claim shows that we can also obtain data $\Gnum{3}$ and $\Gnum{4}$.

\begin{claim} \label{claim: conclusion case B}
	Let $V$ be the trivial one-dimensional $G$-representation over $k$. Then there exists a $G$-equivariant morphism $\varphi \colon X' \to W \times_k \AA(V)$ such that the quadruple $(X', W, V, \varphi)$ forms a $G$-equivariant Gabber's presentation of $(X,Z)$ with respect to the point $z$.
\end{claim}
\begin{proof}
	The $S_z$-equivariant morphism $\tilde{\varphi} = (\tilde{\psi}, \tilde{\nu}) \colon \tilde{X} \to \tilde{W} \times_k \AA(V)$ gives a $G$-equivariant morphism
	\[
	\varphi \colon X' = G \times^{S_z} \tilde{X} \to (G \times^{S_z} \tilde{W}) \times_k \AA(V) = W \times_k \AA(V)
	\]
	mapping $[g,x]$ to $([g,\tilde{\psi}(x)], \tilde{\nu}(x))$, which is well defined because $V$ has trivial $G$-action. The trivial $G$-representation $V$ provides $\Gnum{3}$ and the morphism $\varphi$ provides $\Gnum{4}$. We need to show that $\varphi$ satisfies the properties $\Pnum{1}, \Pnum{2}, \Pnum{3}$ and $\Pnum{4}$ of Definition \ref{definition: gabber presentation}. Since these properties are fpqc (even étale) local on the base, we can argue as follows. To simplify the notation, we write $\AA^1_k$ for $\AA(V)$. The map $\textnormal{id}_G \times \tilde{\varphi} \colon G \times_k \tilde{X} \to G \times_k \AA^1_{\tilde{W}}$ is étale, because $\tilde{\varphi}$ is étale. Passing to the quotient with respect to the $S_z$-action, the map remains étale because we have a cartesian diagram of $S_z$-torsors
	% https://q.uiver.app/#q=WzAsNCxbMCwwLCJHIFxcdGltZXNfayBcXHRpbGRle1h9Il0sWzEsMCwiRyBcXHRpbWVzX2sgXFxBQV4xX2siXSxbMCwxLCJHIFxcdGltZXNee1Nfen0gXFx0aWxkZXtYfSJdLFsxLDEsIiBHXFx0aW1lc157U196fSBcXEFBXjFfayJdLFsyLDNdLFsxLDNdLFswLDJdLFswLDFdXQ==
	\[\begin{tikzcd}
		{G \times_k \tilde{X}} & {G \times_k \AA^1_{\tilde{W}}} \\
		{X' =G \times^{S_z} \tilde{X}} & { G\times^{S_z} \AA^1_{\tilde{W}}= \AA^1_W}
		\arrow[from=1-1, to=1-2]
		\arrow[from=1-1, to=2-1]
		\arrow["\pi", from=1-2, to=2-2]
		\arrow["\varphi", from=2-1, to=2-2]
	\end{tikzcd}\]
	and étaleness is an fpqc local property on the base, giving $\Pnum{1}$. Let $\{g_i \}$ be a complete set of left coset representatives for the quotient $G/S_z$. Let $\tilde{Z} \coloneq \tilde{X} \times_X Z$. There is a cartesian diagram as follows:
	% https://q.uiver.app/#q=WzAsNixbMiwwLCJaIl0sWzIsMSwiWCJdLFsxLDEsIkcgXFx0aW1lc157U196fVggPVxcY29wcm9kX3tnX2l9WCJdLFsxLDAsIkcgXFx0aW1lc157U196fVogPVxcY29wcm9kX3tnX2l9WiJdLFswLDEsIkcgXFx0aW1lc157U196fSBcXHRpbGRle1h9PVxcY29wcm9kX3tnX2l9XFx0aWxkZXtYfSJdLFswLDAsIkcgXFx0aW1lc157U196fVxcdGlsZGV7Wn0gPSBcXGNvcHJvZF97Z19pfVxcdGlsZGV7Wn0iXSxbNCwyXSxbMiwxXSxbMywwXSxbNSw0XSxbNSwzXSxbMCwxXSxbMywyXV0=
	\[\begin{tikzcd}
		{G \times^{S_z}\tilde{Z} = \coprod_{g_i}\tilde{Z}} & {G \times^{S_z}Z =\coprod_{g_i}Z} & Z \\
		{G \times^{S_z} \tilde{X}=\coprod_{g_i}\tilde{X}} & {G \times^{S_z}X =\coprod_{g_i}X} & X,
		\arrow[from=1-1, to=1-2]
		\arrow[from=1-1, to=2-1]
		\arrow[from=1-2, to=1-3]
		\arrow[from=1-2, to=2-2]
		\arrow[from=1-3, to=2-3]
		\arrow[from=2-1, to=2-2]
		\arrow[from=2-2, to=2-3]
	\end{tikzcd}\]
	so that $Z_{X'} = G\times^{S_z} \tilde{Z}$.

	Next note that we also have a cartesian diagram of $S_z$-torsors as follows:
	
	% https://q.uiver.app/#q=WzAsNixbMCwwLCIgRyBcXHRpbWVzX2tcXHRpbGRle1p9Il0sWzEsMCwiRyBcXHRpbWVzX2tcXHRpbGRle1h9Il0sWzIsMCwiRyBcXHRpbWVzX2tcXEFBXjFfe1xcdGlsZGV7V319Il0sWzIsMSwiXFxBQV4xX3tXfSJdLFsxLDEsIlgnIl0sWzAsMSwiWiciXSxbMCw1XSxbNSw0XSxbNCwzLCJcXHZhcnBoaSJdLFsyLDNdLFsxLDRdLFsxLDIsIlxcdGV4dG5vcm1hbHtpZH1cXHRpbWVzIFxcdGlsZGV7XFx2YXJwaGl9Il0sWzAsMV1d
	\[\begin{tikzcd}
		{ G \times_k\tilde{Z}} & {G \times_k\tilde{X}} & {G \times_k\AA^1_{\tilde{W}}} \\
		{Z_{X'} = G \times^{S_z} \tilde{Z}} & {X'} & {\AA^1_{W}.}
		\arrow[from=1-1, to=1-2]
		\arrow[from=1-1, to=2-1]
		\arrow["{\textnormal{id}\times \tilde{\varphi}}", from=1-2, to=1-3]
		\arrow[from=1-2, to=2-2]
		\arrow[from=1-3, to=2-3]
		\arrow[from=2-1, to=2-2]
		\arrow["\varphi", from=2-2, to=2-3]
	\end{tikzcd}\]
	
	The properties $\Pnum{2}, \Pnum{3}$ and $\Pnum{4}$ are all fpqc local properties on the base. Since $\tilde{\varphi} \colon \tilde{X} \to \AA^1_{\tilde{W}}$ satisfies those properties, then also $\varphi \colon X' \to \AA^1_W$ satisfies the same properties, arguing the same way as we did for étaleness.
\end{proof}

\begin{remark} \label{remark: Case A and B give trivial one-dimensional representation}
	By Claim \ref{claim: produce actual model} and Claim \ref{claim: conclusion case B}, we see that both Case \ref{subsection: case A} and Case \ref{subsection: case B} produce a $G$-equivariant Gabber's presentation in which $\Gnum{3}$ is given by the trivial one-dimensional $G$-representation $V$ over $k$.
\end{remark}

\subsection{The case when $G_z \neq \left\{e\right\}$} \label{subsection: case C}

Suppose that the $G$-action on $z$ is not free. Set $H \coloneq G_z$ to be the geometric stabilizer of $z$. This is the largest subgroup of $G$ for which $z \in X^{H}$ holds. 
After replacing $X$ by the open $G$-invariant subscheme $X \setminus(\bigcup_{K \nsubseteq H} X^K)$ we may assume that $G_x \subseteq H$ for all $x \in X$. Let $z'$ be a closed specialization of $z$. Note that $z' \in X^{H}$. 

\begin{claim} \label{claim: if trivial on open neighborhood, conclusions hold}
	Suppose that there exists a $G$-invariant open neighborhood $U$ of $z$ in $X$ such that $U= U^H$ holds (i.e.\ the action of $H$ on $U$ given by the restriction of the $G$-action is trivial). Then there exists a $G$-equivariant Gabber's presentation $(X', W, V_e, \varphi)$ of $(X,Z)$ with respect to the point $z$, where $V_e$ is the trivial one-dimensional $G$-representation over $k$.
\end{claim}
\begin{proof}
	Let $G' \coloneq G/H$ be the quotient group.
	We claim that $(U,Z_U)$ is a $G'$-pair as in Definition \ref{definition: gabber presentation}. Since $U=U^H$, the quotient group $G' = G/H$ acts on $U$ as $\overline{g}.u \coloneq gu$. Indeed, the $G$-action on $U$ is given by a group homomorphism $\rho \colon G \to \Aut_{k}(U)$, and $\rho$ factors through the quotient group $G' = G/H$, because $U$ is fixed by $H$, giving as a result a group homomorphism $G/H \to \Aut_k(U)$, representing the $G'$-action on $U$. It follows that the scheme $U$ is a smooth $k$-scheme of positive dimension endowed with a $G'$-action, and $Z_U = Z \cap U$ is a closed $G'$-invariant subscheme of positive codimension, as desired.
	
	Under the $G'$-action on $U$, the set-theoretic stabilizer of $z$ is $S_z' = S_z/H$, while the geometric stabilizer becomes trivial.
	Moreover, $U \subset X$ is a $G$-equivariant Nisnevich neighborhood of $Gz$ by Remark \ref{rem: open neighborhoods are Nisnevich neighborhoods}.
	
	By Case \ref{subsection: case A} or Case \ref{subsection: case B}, depending on whether $S'_z = G'$ or $S_{z'}' \subsetneq G'$, there exists a $G'$-equivariant Gabber's presentation of the pair $(U, Z_U)$. More precisely the following data exist:
	\begin{enumerate}
		\item[$\Gnum{1}$] a $G'$-equivariant Nisnevich neighborhood $\alpha \colon X' \to U$ of $G'z$;
		\item[$\Gnum{2}$] a smooth $k$-scheme $W$ with a $G'$-action;
		\item[$\Gnum{3}$] a one-dimensional $G'$-representation $V_e$ over $k$ which is the trivial one-dimensional $G'$-representation by Remark \ref{remark: Case A and B give trivial one-dimensional representation};
		\item[$\Gnum{4}$] a $G'$-equivariant morphism $\varphi \colon X' \to W \times_k \AA(V_e)$ satisfying properties $\Pnum{1}, \Pnum{2}, \Pnum{3}$ and $\Pnum{4}$.
	\end{enumerate} 
	We will now show that these data determine a $G$-equivariant Gabber's presentation of the pair $(X, Z)$.
	We endow $X'$ with a $G$-action given by $gx \coloneqq \overline{g}x$ and we claim that $\alpha \colon X' \to U$ is also a $G$-equivariant Nisnevich neighborhood of $Gz$. First note that $\alpha$ is $G$-equivariant, because $\alpha(g x') = \alpha(\overline{g}x') = \overline{g}\alpha(x') = g \alpha(x')$ for all $g \in G$, $x' \in X'$. By definition of $G'$-equivariant Nisnevich neighborhood, there exists a $G'$-equivariant morphism $s'\colon G' \times^{S'_z} \Spec(\kappa(z)) \to X'$ fitting into a commutative diagram
	% https://q.uiver.app/#q=WzAsMyxbMCwxLCJHJyBcXHRpbWVzXntTX3p9XFxTcGVjKFxca2FwcGEoeikpIl0sWzEsMCwiWCciXSxbMSwxLCJYIl0sWzEsMiwiXFxhbHBoYSJdLFswLDJdLFswLDEsInMnIl1d
	\[\begin{tikzcd}
		& {X'} \\
		{G' \times^{S'_z}\Spec(\kappa(z))} & U.
		\arrow["\alpha", from=1-2, to=2-2]
		\arrow["{s'}", from=2-1, to=1-2]
		\arrow[from=2-1, to=2-2]
	\end{tikzcd}\]
	
	Note that there exists a well-defined morphism $r \colon G \times^{S_z} \Spec(\kappa(z)) \rightarrow G' \times^{S_z'} \Spec(\kappa(z))$
	mapping $[g,z] \mapsto [\overline{g}, z]$. Post-composing the morphism $r$ with $s'$ gives a morphism $s = s' \circ r \colon G \times^{S_z} \Spec(\kappa(z)) \to X'$ mapping $[g,z]$ to $s'([\overline{g}, z])$. 
	Note that $s$ is $G$-equivariant, because
	\[
	s(g_1[g_2,z]) = s ([g_1^{-1}g_2, g_1z]) = s'([\overline{g_1^{-1}g_2}, g_1z]) = s'(\overline{g_1}[\overline{g_2},z])= \overline{g_1}s'([\overline{g_2},z]) = g_1s([g_2, z]),
	\]
	where we used the fact that $s'$ is $G'$-equivariant and that $G$ is abelian.
	It follows that $\alpha \colon X' \to U$ is a $G$-equivariant Nisnevich neighborhood of $Gz$, as desired. In particular, $X' \to U \to X$ is a $G$-equivariant Nisnevich neighborhood of $Gz$, giving $\Gnum{1}$. We endow $W$ with a $G$-action by setting $gw \coloneq \overline{g}w$, giving $\Gnum{2}$ and we still denote by $V_e$ the trivial one-dimensional $G$-representation, giving $\Gnum{3}$. Then the $G'$-equivariant morphism $\varphi \colon X' \to W \times_k \AA(V_e)$ is also $G$-equivariant because
	\[
	\varphi(gx') = \varphi(\overline{g}x') = \overline{g}\varphi(x') =  g\varphi(x'),
	\]
	where we used the $G'$-equivariance of $\varphi$, giving $\Gnum{4}$. The morphism $\varphi$ satisfies properties $\Pnum{1}, \Pnum{2}, \Pnum{3}$ and $\Pnum{4}$ of Definition \ref{definition: gabber presentation} by construction.
\end{proof}

Note that $H$ is the geometric stabilizer of $z'$: indeed $H \subseteq G_{z'}$ because $z' \in X^H$, while $G_{z'} \subseteq H$ because, at the beginning of Case \ref{subsection: case C}, we removed the larger fixed-point loci. In particular, the $\kappa(z')$-vector space $T_{z'}X$ is an $H$-representation. If $T_{z'}X$ is the trivial $H$-representation, then by Lemma \ref{lemma: trivial action on tangent space} there exists a $G$-invariant open neighborhood $U$ of $z'$ with trivial $H$-action. Since $U$ would be also an open neighborhood of $z$, in that case we may apply Claim \ref{claim: if trivial on open neighborhood, conclusions hold} to conclude. Thus, from now on, we assume that $T_{z'}X$ is a nontrivial $H$-representation. After replacing $X$ by the union of the $G$-translates of an irreducible component containing $z$, we may assume that $X$ is equidimensional. Set $\dim X = n+1$ for some $n \geq 0$ and $X= \Spec(A)$ for a finitely generated $k$-algebra $A$.

\begin{claim}
	There exist an affine $G$-equivariant Nisnevich neighborhood $X_0 \to X$ of $Gz$, an $(n+1)$-dimensional $G$-representation $V''$ over $k$ and a $G$-equivariant morphism $\varphi_0 \colon X_0 \to \AA(V'')$ such that
	\begin{itemize}
		\item $\varphi_0$ is étale;
		\item there exists a decomposition of $G$-representations $V''= V' \oplus V$ with $\dim_k(V')=n$ and such that $\textnormal{Res}^G_H(V)$ is a one-dimensional nontrivial $H$-representation;
		\item the composite $Z_0 \to X_0 \xrightarrow{\varphi_0} \AA(V'') \to \AA(V')$ is quasi-finite.
	\end{itemize}
\end{claim}

\begin{proof}
	We show that the required properties can be achieved separately and then conclude by the openness of the properties. By Lemma \ref{lemma: exists representation 1} there exists a $G$-representation $V''$ over $k$ such that $\textnormal{Res}^{G}_{S_{z'}}(V'') \otimes_k \kappa(z') \simeq T_{z'}X$ in $\textnormal{Rep}^{\#}_{\kappa(z')}(S_{z'})$. Since $X$ is smooth over $k$, the dimension of $X$ equals $\dim_{\kappa(z')}T_{z'}X$, so that $\dim_k(V'')= n+1$ necessarily. By Proposition \ref{proposition: étale morphism as in 8.10 of equivariant cycles} there exists a $G$-equivariant morphism $f_1 \colon X \to \AA(V'')$ which is étale at $z'$. Moreover, the isomorphism $\textnormal{Res}^{G}_{S_{z'}}(V'') \otimes_k \kappa(z') \simeq T_{z'}X$ in $\textnormal{Rep}^{\#}_{\kappa(z')}(S_{z'})$ restricts to an isomorphism $\textnormal{Res}^{G}_{H}(V'') \otimes_k \kappa(z') \simeq T_{z'}X$ in $\textnormal{Rep}_{\kappa(z')}(H)$. Then by Proposition \ref{proposition: quasi-finite morphisms construction}, there exists a decomposition $V'' = V' \oplus V$ of $G$-representations over $k$ and a $G$-equivariant morphism $\tilde{f}_2 \colon X \to \AA(V')$ such that
	\begin{itemize}
		\item $\dim_k V' = n$;
		\item the composite $Z \to X \xrightarrow{\tilde{f}_2} \AA(V')$ is quasi-finite at $z'$;
		\item $\textnormal{Res}^G_H (V)$ is a one-dimensional nontrivial $H$-representation.
	\end{itemize}
	Let $V' \to V''$ be a $G$-equivariant linear splitting of the projection $V'' \to V'$. Then the morphism $f_2 \colon X \xrightarrow{\tilde{f}_2} \AA(V') \to \AA(V'')$ is $G$-equivariant and the composite $Z \to X \to \AA(V'') \to \AA(V')$ is quasi-finite at $z'$. 
	
	Let $\iota \colon X \hookrightarrow \AA(W)$ be a $G$-equivariant closed immersion of $X$ into a $G$-representation as in Remark \ref{remark: fix equivariant closed immersion}. Up to adding the equivariant elements defining the morphisms $f_1$ and $f_2$ to the generators of the closed immersion $\iota$, we may assume that they are among them.
	
	By Lemma \ref{lemma: openness of properties} and Remark \ref{remark: étaleness open condition}, the required étaleness and the required quasi-finiteness define two non-empty open subsets of the affine space parametrizing $G$-equivariant linear maps $\AA(W) \to \AA(V'')$, as witnessed by $f_1$ and $f_2$. Their intersection is non-empty because the parameter space is irreducible, and it has a $k$-rational point because $k$ is infinite. Hence there exists a morphism $\varphi \colon X \to \AA(V'')$ which is étale at $z'$, and such that the composite $Z \to X \xrightarrow{\varphi} \AA(V'') \to \AA(V')$ is quasi-finite at $z'$.

	By Lemma \ref{lemma: assume is étale and quasi-finite} we conclude that there exists an open $G$-invariant neighborhood $U \to X$ of $Gz=z$ such that $f_{|U} \colon U \to \AA(V'')$ is étale and $Z \cap U \to \AA(V'') \to \AA(V')$ is quasi-finite. Finally, by Lemma \ref{lemma: assume is affine} we know that there exists an affine $G$-equivariant Nisnevich neighborhood $X_0 \to U \subseteq X$ of $Gz$. The composite $\varphi_0 \colon X_0 \to U \to \AA(V'')$ is again étale, being a composition of étale morphisms and  $Z_{X_0} \to Z \cap U \to \AA(V')$ is again quasi-finite (indeed is quasi-compact and locally quasi-finite, because the property of being locally quasi-finite is étale local on the source-and-target, and $Z \cap U \to \AA(V')$ is quasi-finite), as desired.
\end{proof}

Let $\tilde{W} \subset X_0$ be the vanishing locus of the last coordinate of the map $\varphi_0 \colon X_0 \to \AA(V') \times_k \AA(V)$. In other words, $\tilde{W}$ is given by the cartesian product 
% https://q.uiver.app/#q=WzAsNCxbMCwwLCJXIl0sWzEsMCwiWCJdLFsxLDEsIlxcQUEoVl8xKSBcXHRpbWVzX2sgXFxBQShWXzIpIl0sWzAsMSwiXFxBQShWXzEpIl0sWzEsMiwiXFx2YXJwaGkiXSxbMCwxLCIiLDIseyJzdHlsZSI6eyJ0YWlsIjp7Im5hbWUiOiJob29rIiwic2lkZSI6InRvcCJ9fX1dLFszLDIsIiIsMCx7InN0eWxlIjp7InRhaWwiOnsibmFtZSI6Imhvb2siLCJzaWRlIjoidG9wIn19fV0sWzAsM11d
\[\begin{tikzcd}
	\tilde{W} & X_0 \\
	{\AA(V')} & {\AA(V'')}
	\arrow[hook, from=1-1, to=1-2]
	\arrow[from=1-1, to=2-1]
	\arrow["\varphi_0", from=1-2, to=2-2]
	\arrow[hook, from=2-1, to=2-2]
\end{tikzcd}\]
where the bottom horizontal map is the inclusion in all the coordinates but the last one. It follows that $\tilde{W}$ is étale over
$\AA(V')$, and hence $\tilde{W}$ is a smooth
$k$-scheme.
Moreover $X_0^{H} \subseteq \tilde{W}$, because $\textnormal{Res}^G_H(V)$ is a nontrivial $H$-representation, so that the missing coordinate carries a nontrivial $H$-action.
In particular, we have $Gz \in \tilde{W}$.

\begin{remark} \label{remark: why does not exists a section}
	Note that this definition of $\tilde{W}$ only works in the setting of $z$ with non-free action (Case \ref{subsection: case C}). Indeed, we need a coordinate with nontrivial $G_z$-action in order to be sure that $z$ belongs to the vanishing locus of that coordinate. This is the reason why we are able to produce a section only in the non-free action case.
\end{remark}

\begin{claim}
	There exist a $G$-equivariant Nisnevich neighborhood $X_1 \to X_0$ of $Gz$ and a $G$-equivariant morphism $a \colon X_1 \to \tilde{W}$ admitting a $G$-equivariant section $s \colon \tilde{W} \to X_1$.
\end{claim}
\begin{proof}
	Set $X_1 \coloneq X_0 \times_{\AA(V')} \tilde{W}$. We have the following cartesian diagram:

	% https://q.uiver.app/#q=WzAsOCxbMywwLCJXIl0sWzMsMSwiQSJdLFsyLDAsIlcgXFx0aW1lc19rIFxcQUFee1xcc2lnbWFfan1fayJdLFsyLDEsIkEgXFx0aW1lc19rIFxcQUFfa157XFxzaWdtYV9qfSJdLFsxLDAsIlhfMSJdLFsxLDEsIlgiXSxbMCwwLCJXIFxcdGltZXNfQSBXIl0sWzAsMSwiVyJdLFswLDFdLFsyLDNdLFszLDFdLFsyLDBdLFs1LDMsIlxcYWxwaGEiLDJdLFs0LDJdLFs0LDVdLFszLDQsIiIsMSx7InN0eWxlIjp7Im5hbWUiOiJjb3JuZXItaW52ZXJzZSJ9fV0sWzEsMiwiIiwxLHsic3R5bGUiOnsibmFtZSI6ImNvcm5lci1pbnZlcnNlIn19XSxbNyw1LCIiLDIseyJzdHlsZSI6eyJ0YWlsIjp7Im5hbWUiOiJob29rIiwic2lkZSI6InRvcCJ9fX1dLFs2LDddLFs2LDQsIiIsMix7InN0eWxlIjp7InRhaWwiOnsibmFtZSI6Imhvb2siLCJzaWRlIjoidG9wIn19fV0sWzUsNiwiIiwwLHsic3R5bGUiOnsibmFtZSI6ImNvcm5lci1pbnZlcnNlIn19XV0=
	\[\begin{tikzcd}
		{\tilde{W} \times_{\AA(V')} \tilde{W}} & {X_1} & {\tilde{W} \times_k \AA(V)} & \tilde{W} \\
		\tilde{W} & X_0 & {\AA(V') \times_k \AA(V)} & \AA(V').
		\arrow[hook, from=1-1, to=1-2]
		\arrow[from=1-1, to=2-1]
		\arrow[from=1-2, to=1-3]
		\arrow[from=1-2, to=2-2]
		\arrow[from=1-3, to=1-4]
		\arrow[from=1-3, to=2-3]
		\arrow[from=1-4, to=2-4]
		\arrow[hook, from=2-1, to=2-2]
		\arrow["\varphi_0", from=2-2, to=2-3]
		\arrow[from=2-3, to=2-4]
	\end{tikzcd}\]

	Each distinguished point $gz \colon \Spec(\kappa (z)) \to X_0$ of $Gz \subset X_0$ lifts to $X_1$, because $gz \in \tilde{W}$. It follows that the morphism $X_1 \to X_0$ is a $G$-equivariant Nisnevich neighborhood of $Gz$. More precisely, by universal property of the fiber product, there exists a unique morphism $G \times^{S_z} \Spec(\kappa(z)) \to X_1$ factoring $G \times^{S_z} \Spec(\kappa(z)) \to X_0$ and $G \times^{S_z} \Spec(\kappa(z)) \to \tilde{W}$. The composite $s \colon \tilde{W} \xrightarrow{\Delta_{\tilde{W}}} \tilde{W} \times_{\AA(V')} \tilde{W} \to X_1$ is a $G$-equivariant section of $a \colon X_1 \to \tilde{W}$. 
\end{proof}

We will now introduce a henselian scheme over which we will find the conclusions of Theorem  \ref{theorem: gabber}. Consider the semi-local ring $\cO_{\tilde{W},Gz}$, and its henselization at the ideal defining the points of the orbit $Gz$. By Lemma \ref{lemma: henselization}, we can write $\cO_{\tilde{W},Gz}^h = \prod_{gz \in Gz} \cO^h_{\tilde{W},gz}$.
Set $\tilde{W}_{Gz}^h \coloneq \Spec (\cO_{\tilde{W},Gz}^h) \simeq \coprod_{gz \in Gz} \Spec(\cO^h_{\tilde{W},gz})$. Note that $\tilde{W}_{Gz}^h$ is a $G$-scheme because morphisms between local rings lift to the henselization by \cite[\href{https://stacks.math.columbia.edu/tag/04GS}{Tag 04GS}]{stacks-project}. Moreover, $\tilde{W}_{Gz}^h$ coincides with the limit over Nisnevich étale neighborhoods of $Gz$ in $\tilde{W}$ by \cite[(2.21)]{MotivicHomotopyTheoryOfGroupSchemesActions}.

\begin{remark}
	As in Case \ref{subsection: case A} of $z$ with free action, starting from $X_1$ we define a sequence of schemes $X_2, X_3, X_4$, each obtained from the previous one by shrinking $X$ Nisnevich locally. In particular, every time that we construct a scheme $X_i$ with a morphism $X_i \to X$ such that there exist distinguished points $gz_i \in X_i$ mapping to $gz \in X$ and inducing an isomorphism on the residue fields, we will still denote by $gz$ the points $gz_i$. Moreover, we will denote by $Z_i $ the fiber product $Z_i \coloneq X_i \times_X Z$. At each step we will achieve one of the properties $\Pnum{1}, \Pnum{2}, \Pnum{3}, \Pnum{4}$ of Definition \ref{definition: gabber presentation} required in the conclusion of Theorem \ref{theorem: gabber}. Once all such properties have been established, we will have the conclusions of Theorem \ref{theorem: gabber} over the henselian scheme $\tilde{W}^{h}_{Gz}$ and then we apply a limit-descending argument to find conclusions over an actual Nisnevich neighborhood of $Gz$.
\end{remark}

\begin{claim} \label{claim: X_2 construction}
	There exist a $G$-scheme $X_2$ over $X_1$, a nontrivial one-dimensional $G$-representation $V$ and a $G$-equivariant étale morphism $\varphi_2 = (\psi_2, \nu_2) \colon X_2 \to \tilde{W}^h_{Gz} \times_k \AA(V)$ such that $\psi_2 \colon X_2 \to \tilde{W}^h_{Gz}$ admits a $G$-equivariant section $\tilde{W}_{Gz}^h \hookrightarrow X_2$ which is a closed immersion and the restriction $\psi_{2_{|Z_2}} \colon Z_2 \to \tilde{W}^h_{Gz}$ is quasi-finite.
\end{claim}
\begin{proof}
	The morphism $\varphi_1 \colon X_1 \to \tilde{W} \times_k \AA(V)$ induced by $\varphi_0$ and the fiber-product description of $X_1$ is étale, since its square with $\varphi_0 \colon X_0 \to \AA(V') \times_k \AA(V)$ is cartesian. 
	Set $X_2 \coloneq X_1 \times_{\tilde{W}} \tilde{W}^h_{Gz}$ and let $\varphi_2$ be its base change. The smooth $k$-scheme $\tilde{W}$ is a closed subscheme of $X_1$ and the composite $\tilde{W} \xrightarrow{s} X_1 \xrightarrow{a} \tilde{W}$ is the identity by construction. Moreover, the composite of $s$ with $X_1 \to \tilde{W} \times_k \AA(V)$ is the zero section. It follows that we have the following cartesian diagram:

	% https://q.uiver.app/#q=WzAsOCxbMiwwLCJXX3peaCBcXHRpbWVzX2sgXFxBQV57XFxzaWdtYV9qfSJdLFszLDAsIldfel5oIl0sWzMsMSwiVyJdLFsyLDEsIlcgXFx0aW1lc19rIFxcQUFee1xcc2lnbWFfan0iXSxbMSwxLCJYXzEnIl0sWzEsMCwiWF8yIl0sWzAsMCwiV196XmgiXSxbMCwxLCJXIl0sWzUsNF0sWzQsM10sWzMsMl0sWzEsMl0sWzAsMV0sWzAsM10sWzUsMF0sWzMsNSwiIiwyLHsic3R5bGUiOnsibmFtZSI6ImNvcm5lci1pbnZlcnNlIn19XSxbMiwwLCIiLDIseyJzdHlsZSI6eyJuYW1lIjoiY29ybmVyLWludmVyc2UifX1dLFs3LDQsIiIsMCx7InN0eWxlIjp7InRhaWwiOnsibmFtZSI6Imhvb2siLCJzaWRlIjoidG9wIn19fV0sWzYsN10sWzYsNSwiIiwwLHsic3R5bGUiOnsidGFpbCI6eyJuYW1lIjoiaG9vayIsInNpZGUiOiJ0b3AifX19XSxbNCw2LCIiLDAseyJzdHlsZSI6eyJuYW1lIjoiY29ybmVyLWludmVyc2UifX1dXQ==
	\[\begin{tikzcd}
		{\tilde{W}_{Gz}^h} & {X_2} & {\tilde{W}_{Gz}^h \times_k \AA(V)} & {\tilde{W}_{Gz}^h} \\
		\tilde{W} & {X_1} & {\tilde{W} \times_k \AA(V)} & \tilde{W},
		\arrow[hook, from=1-1, to=1-2]
		\arrow[from=1-1, to=2-1]
		\arrow["{\varphi_2}",from=1-2, to=1-3]
		\arrow[from=1-2, to=2-2]
		\arrow[from=1-3, to=1-4]
		\arrow[from=1-3, to=2-3]
		\arrow[from=1-4, to=2-4]
		\arrow[hook, from=2-1, to=2-2]
		%\arrow["\lrcorner"{anchor=center, pos=0.125}, draw=none, from=1-1, to=2-2]
		\arrow["{\varphi_1}",from=2-2, to=2-3]
		%\arrow["\lrcorner"{anchor=center, pos=0.125}, draw=none, from=1-2, to=2-3]
		\arrow[from=2-3, to=2-4]
		%\arrow["\lrcorner"{anchor=center, pos=0.125}, draw=none, from=1-3, to=2-4]
	\end{tikzcd}\]
	from which we see that $\tilde{W}^h_{Gz} \to X_2$ is a closed immersion and provides a $G$-equivariant section to the morphism $\psi_2 \colon X_2 \to \tilde{W}^h_{Gz}$. The morphism $\varphi_2 \colon X_2 \to \tilde{W}^h_{Gz} \times_k \AA(V)$ is étale because it is the base change of the étale morphism $\varphi_1$ just described; its section maps to the zero section. Note that $\dim_k(V) = 1$ and that it is a nontrivial $G$-representation by construction. The morphism $\psi_{2_{|Z_2}} \colon Z_{2} \to \tilde{W}^h_{Gz}$ is quasi-finite, because it is the base change of $Z_0 \to \AA(V')$. 
\end{proof}

By Claim \ref{claim: X_2 construction} we get property $\Pnum{1}$ of Definition \ref{definition: gabber presentation} over $\tilde{W}^h_{Gz}$. The distinguished points $Gz \in X_1$ lift to $X_2$.

\begin{claim} \label{claim: X_3 construction}
	There exists an open $G$-invariant subscheme $X_3 \subset X_2$ containing $\tilde{W}^h_{Gz}$ such that $\psi_{3_{|Z_3}} \colon Z_3 \to X_3 \xrightarrow{\psi_3} \tilde{W}^h_{Gz}$ is finite. Moreover, $Z_3$ decomposes as $Z_3 = \coprod_{gz \in Gz} \Spec(A_{1,gz})$, where each $A_{1,gz}$ is local and $\psi_{3_{|Z_3}}$ is the disjoint union of $|Gz|$ finite morphisms $\beta_{gz} \colon \Spec(A_{1,gz}) \to \tilde{W}^h_{gz}$.
\end{claim}

\begin{proof}
	It is sufficient to notice that Claim \ref{claim: X_2 construction} guarantees that hypotheses of Lemma \ref{lemma: pass from quasi-finite to finite} are satisfied for the schemes $X_2, \tilde{W}^h_{Gz}, Z_2$. Then by Lemma \ref{lemma: pass from quasi-finite to finite}.(1) we conclude that there exists an open $G$-invariant subscheme $X_3 \subset X_2$ such that $\psi_{3_{|Z_3}} \colon Z_3 \to \tilde{W}^h_{Gz}$ is finite and by Lemma \ref{lemma: pass from quasi-finite to finite}.(2) we have that $Z_3$ and $\psi_{3_{|Z_3}}$ decompose as desired. Finally, by Lemma \ref{lemma: pass from quasi-finite to finite}.(3) applied to the closed subscheme $K = \tilde{W}^h_{Gz}$ of $X_2$, we see that $\tilde{W}^h_{Gz} \subset X_3$.
\end{proof}

By Claim \ref{claim: X_3 construction} we get property $\Pnum{3}$ of Definition \ref{definition: gabber presentation} over $\tilde{W}^h_{Gz}$.

\begin{claim} \label{claim: obtain closed immersion}
	The morphism $\varphi_{3_{|Z_3}} \colon Z_3 \to \tilde{W}^h_{Gz} \times_k \AA(V)$ is a closed immersion.
\end{claim}	
\begin{proof}
	By Claim \ref{claim: X_3 construction}, the morphism $\varphi_{3_{|Z_3}}$ also decomposes as $|Gz|$ disjoint morphisms 
	\[
	\gamma_{gz} \colon \Spec(A_{1,gz}) \longrightarrow \tilde{W}^h_{gz} \times_k \AA(V).
	\]
	Hence it is sufficient to show that $\gamma_{gz}$ is a closed immersion for all $g \in G$. Fix $g \in G$. To show that $\gamma_{gz}$ is a closed immersion, we verify that the hypotheses of Lemma \ref{lemma: nakayama's lemma for closed immersion} are satisfied for $A=A_{1,gz}$, $B = \cO_{\tilde{W},gz}^h$ and $C = k[t]$. Indeed $\gamma_{gz}$ is an unramified morphism, because it is the composite of a closed immersion $\Spec(A_{1,gz}) \to X_3$ and an étale morphism $\varphi_3 \colon X_3 \subseteq X_2  \xrightarrow{\varphi_2} \tilde{W}^h_{Gz} \times_k \AA(V)$. The morphism $\beta_{gz} \colon \Spec(A_{1,gz}) \to \tilde{W}^h_{gz}$ is finite by Claim \ref{claim: X_3 construction}. Let $F$ be the fiber of $\beta_{gz}$ over the unique closed point $gz$. By Lemma \ref{lemma: closed fiber of a finite local morphism} we know that the underlying topological space of $F$ is made of the only closed point of $\Spec(A_{1,gz})$. Then the fiber morphism $h$ given by the cartesian square
	% https://q.uiver.app/#q=WzAsNCxbMCwwLCJcXFNwZWMoXFxrYXBwYSh6KSkiXSxbMCwxLCJcXEFBXjFfe1xca2FwcGEoeil9Il0sWzEsMSwiXFxBQV4xX3tcXHRpbGRle1d9Xmhfe2d6fX0iXSxbMSwwLCJcXFNwZWMoQV97MSxnen0pIl0sWzMsMiwiXFxnYW1tYV97Z3p9Il0sWzEsMl0sWzAsMSwiaCIsMl0sWzAsM11d
	\[\begin{tikzcd}
		{F} & {\Spec(A_{1,gz})} \\
		{\Spec(\kappa(gz)) \times_k \AA(V)} & {\tilde{W}^h_{gz} \times_k \AA(V)}
		\arrow[from=1-1, to=1-2]
		\arrow["h"', from=1-1, to=2-1]
		\arrow["{\gamma_{gz}}", from=1-2, to=2-2]
		\arrow[from=2-1, to=2-2]
	\end{tikzcd}\]
	 is radicial (indeed, it induces the identity on the residue field). Then by Lemma \ref{lemma: nakayama's lemma for closed immersion}  it follows that $\gamma_{gz}$ is a closed immersion as desired.
\end{proof}

By Claim \ref{claim: obtain closed immersion} we get property $\Pnum{2}$ of Definition \ref{definition: gabber presentation} over $\tilde{W}^h_{Gz}$.

\begin{claim} \label{claim: X_4}
	There exists an open $G$-invariant subscheme $X_4 \subset X_3$ containing $\tilde{W}^h_{Gz}$ such that the composite $\varphi_4 \colon X_4 \subseteq X_{3} \xrightarrow{\varphi_3} \tilde{W}^h_{Gz} \times_k \AA(V)$ induces an isomorphism $\varphi_4^{-1}(\varphi_4(Z_4)) \xrightarrow{\sim} \varphi_4(Z_4)$.
\end{claim}	

\begin{proof}
	By Lemma \ref{lemma: exist open with isomorphism in the image}, there exists an open $G$-invariant subscheme $X_4 \subset X_3$ such that $Z_4 = Z_3$ and the map $\varphi_4 \colon X_4 \to \tilde{W}^h_{Gz} \times_k \AA(V)$ induces an isomorphism $\varphi_4^{-1}(\varphi_4(Z_4)) \xrightarrow{\sim} \varphi_4(Z_4)$. Finally, for each point $gz$ of the orbit $Gz$ in $Z_3 = Z_4 \subset X_4$ we see that $\tilde{W}^h_{gz} \cap (X_3 \setminus X_4)$ is closed in $\tilde{W}^h_{gz}$ but it does not contain $gz$, hence it must be empty. Consequently we have $\tilde{W}^h_{Gz} \subseteq X_4$ necessarily.
\end{proof}

To summarize, we have obtained the properties $\Pnum{1}, \Pnum{2}, \Pnum{3}, \Pnum{4}$ over the henselian scheme $\tilde{W}^h_{Gz}$. Note that the schemes $\tilde{W}^h_{Gz} \times_k \AA(V), X_2, X_4, Z_{4}, \varphi_4(Z_{4}), \varphi_4^{-1}(\varphi_4(Z_{4}))$ are of finite presentation over $\tilde{W}^h_{Gz}$ (compare with Remark \ref{remark: schemes of finite presentation}). 

\begin{claim}
	There exists a $G$-equivariant Gabber's presentation $(X', W, V, \varphi)$ of the $G$-pair $(X,Z)$ with respect to the point $z$. Moreover, $V$ is a nontrivial $G$-representation, and $\psi \colon X' \to W$ admits a $G$-equivariant section.
\end{claim}

\begin{proof}
	We will use a limit-descending argument to obtain conclusions $\Pnum{1}, \Pnum{2}, \Pnum{3}, \Pnum{4}$ of Definition \ref{definition: gabber presentation}, that hold for $\varphi_4 \colon X_4 \to \tilde{W}^h_{Gz} \times_k \AA(V)$, over an actual model.
	Write $\tilde{W}^h_{Gz} = \lim_{ U \in N_{G}(Gz)}U$ as the limit over the filtering category of affine equivariant Nisnevich neighborhoods of $Gz$ as in \cite[\S 2]{MotivicHomotopyTheoryOfGroupSchemesActions}. Since $\tilde{W}$ is affine, there is an isomorphism $\tilde{W}^h_{Gz} \simeq \Spec(\cO_{\tilde{W},Gz}^h)$ by \cite[\S 2]{MotivicHomotopyTheoryOfGroupSchemesActions}. Since $\textnormal{Sch}^G_k$ is essentially small, the category
	$N_G(Gz)$ of equivariant Nisnevich neighborhoods is essentially small as well. Replacing it by a small skeleton, we may apply \cite[Theorem 1.5]{PresentableCategories} and write $\tilde{W}^h_{Gz}= \lim_{i \in I } U_i$ as the limit of an inverse system of schemes over a directed set $I$. All the required properties descend  through limits of morphisms of finite presentation. More precisely we do the following.
	
	Since $X_4$ is open in $X_2 = X_1 \times_{\tilde{W}} \tilde{W}^h_{Gz}$, by Lemma \ref{lemma: limit descending}, there exists an index $i \in I$, a scheme $S_i$ with a $G$-action and a $G$-equivariant morphism $a_i \colon S_i \to X_1 \times_{\tilde{W}} U_i$ which is an open immersion, fitting into a cartesian diagram
	% https://q.uiver.app/#q=WzAsNCxbMCwwLCJYXzQiXSxbMSwwLCJYX2oiXSxbMSwxLCJYIFxcdGltZXNfe1xcQUFebl9rfVVfaiJdLFswLDEsIlggXFx0aW1lc197XFxBQV5uX2t9IFdeaF93Il0sWzMsMl0sWzEsMiwiYV9qIl0sWzAsM10sWzAsMV1d
	\begin{equation}\label{diagram: first cartesian diagram}
	\begin{tikzcd}
		{X_4} & {S_i} \\
		{X_2} & {X_1 \times_{\tilde{W}} U_i.}
		\arrow[from=1-1, to=1-2]
		\arrow[from=1-1, to=2-1]
		\arrow["{a_i}", from=1-2, to=2-2]
		\arrow[from=2-1, to=2-2]
	\end{tikzcd}
	\end{equation}

	We first show how to obtain $\Pnum{1}$. Since $\varphi_4$ is étale, by Lemma \ref{lemma: limit descending}, there exists an index $j \in I$, a scheme $S_j$ with a $G$-action and a $G$-equivariant morphism $b_j \colon S_j \to U_j \times_k \AA(V)$ which is étale, fitting into a cartesian diagram
	% https://q.uiver.app/#q=WzAsNCxbMCwwLCJYXzQiXSxbMSwwLCJYX2kiXSxbMSwxLCJcXEFBXjFfe1VfaX0iXSxbMCwxLCJcXEFBXjFfe1deaF93fSJdLFszLDJdLFsxLDIsImEiXSxbMCwzLCJcXHZhcnBoaSIsMl0sWzAsMV1d
	\begin{equation} \label{diagram: second cartesian diagram}
	\begin{tikzcd}
		{X_4} & {S_j} \\
		{\tilde{W}^h_{Gz} \times_k \AA(V)} & {U_j \times_k \AA(V).}
		\arrow[from=1-1, to=1-2]
		\arrow["\varphi_4"', from=1-1, to=2-1]
		\arrow["b_j", from=1-2, to=2-2]
		\arrow[from=2-1, to=2-2]
	\end{tikzcd}
	\end{equation}
	
	After increasing the index, we may assume $j \geq i$. We rename $j = i$ and thus we have a scheme $S_i$ and morphisms $a_i, b_i$ as in diagram \ref{diagram: first cartesian diagram} and diagram \ref{diagram: second cartesian diagram}.
	
	We now show how to obtain property $\Pnum{2}$. Set $V_i \coloneq S_i \times_X Z$, so that we have a cartesian diagram as follows:
	% https://q.uiver.app/#q=WzAsMTQsWzAsMCwiWl80Il0sWzAsMSwiWF80Il0sWzEsMCwiVl9pIl0sWzEsMSwiU19pIl0sWzIsMCwiWl8xIFxcdGltZXNfe1xcdGlsZGV7V319VV9pIl0sWzIsMSwiWF8xIFxcdGltZXNfe1xcdGlsZGV7V319VWkiXSxbMywwLCJaXzEiXSxbNCwwLCJaIl0sWzQsMSwiWCJdLFszLDEsIlhfMSJdLFsxLDIsIlVfaSBcXHRpbWVzX2sgXFxBQShWKSJdLFswLDIsIlxcdGlsZGV7V31eaF97R3p9IFxcdGltZXNfayBcXEFBKFYpIl0sWzEsMywiVV9pIl0sWzAsMywiXFx0aWxkZXtXfV5oX3tHen0iXSxbMSwxMV0sWzExLDEwXSxbMywxMCwiYl9pIiwyXSxbMSwzXSxbMyw1LCJhX2kiLDJdLFsyLDMsImNfaSIsMl0sWzAsMV0sWzAsMl0sWzIsNF0sWzQsNV0sWzUsOV0sWzQsNl0sWzYsOV0sWzExLDEzXSxbMTMsMTJdLFsxMCwxMl0sWzksOF0sWzYsN10sWzcsOF1d
	\[\begin{tikzcd}
		{Z_4} & {V_i} & {Z_1 \times_{\tilde{W}}U_i} & {Z_1} & Z \\
		{X_4} & {S_i} & {X_1 \times_{\tilde{W}}U_i} & {X_1} & X \\
		{\tilde{W}^h_{Gz} \times_k \AA(V)} & {U_i \times_k \AA(V)} \\
		{\tilde{W}^h_{Gz}} & {U_i.}
		\arrow[from=1-1, to=1-2]
		\arrow[from=1-1, to=2-1]
		\arrow[from=1-2, to=1-3]
		\arrow["{c_i}"', from=1-2, to=2-2]
		\arrow[from=1-3, to=1-4]
		\arrow[from=1-3, to=2-3]
		\arrow[from=1-4, to=1-5]
		\arrow[from=1-4, to=2-4]
		\arrow[from=1-5, to=2-5]
		\arrow[from=2-1, to=2-2]
		\arrow[from=2-1, to=3-1]
		\arrow["{a_i}"', from=2-2, to=2-3]
		\arrow["{b_i}"', from=2-2, to=3-2]
		\arrow[from=2-3, to=2-4]
		\arrow[from=2-4, to=2-5]
		\arrow[from=3-1, to=3-2]
		\arrow[from=3-1, to=4-1]
		\arrow[from=3-2, to=4-2]
		\arrow[from=4-1, to=4-2]
	\end{tikzcd}\]
    Since $\varphi_{4_{|Z_4}}$ is a closed immersion, again by Lemma \ref{lemma: limit descending}, after increasing the index, we may assume that $b_i \circ c_i \colon V_i \to U_i \times_k \AA(V)$ is a closed immersion. We next show how to obtain property $\Pnum{3}$. Since $\psi_{4_{|Z_4}} \colon Z_4 \to \tilde{W}^h_{Gz}$ is finite, again by Lemma \ref{lemma: limit descending}, after increasing the index, we may assume that $V_i \to U_i$ is finite.
	
	We obtain property $\Pnum{4}$ as follows. The morphism $b_i \coloneq S_i \to U_i \times_k \AA(V)$ induces $b_i^{-1}(b_i(V_i)) \to b_i(V_i)$ which is an isomorphism on the base change over $\tilde{W}^h_{Gz}$ by Claim \ref{claim: X_4}. Then, after increasing the index $i$ again, we may assume by Lemma \ref{lemma: limit descending} that $b_i^{-1}(b_i(V_i)) \rightarrow b_i(V_i)$ is an isomorphism.
	
	Finally, since $\tilde{W}^h_{Gz} \to X_4 \to \tilde{W}^h_{Gz}$ is the identity, after increasing the index $i$ again, we may assume that the $G$-equivariant section $\tilde{W}^h_{Gz} \to X_4$ descends to an actual model. In other words, we may assume that there exists a $G$-equivariant morphism $U_i \to S_i$ such that the composite $U_i \to S_i \to U_i$ is the identity.
	
	Set $X' \coloneq S_i$, $W \coloneq U_i$ and $\varphi \coloneq b_i \colon X' \to W \times_k \AA(V)$. Then $Z_{X'} = X' \times_X Z= V_i$. We claim that $(X', W, V, \varphi)$ forms a $G$-equivariant Gabber's presentation of $(X,Z)$ with respect to the point $z$. We first show that $X' \to X$ is a $G$-equivariant Nisnevich neighborhood of $Gz$. In the following diagram
	% https://q.uiver.app/#q=WzAsNSxbMiwwLCJcXHRpbGRle1d9Il0sWzIsMSwiVyJdLFsxLDEsIlhfMSJdLFsxLDAsIlhfMSBcXHRpbWVzX1cgXFx0aWxkZXtXfSJdLFswLDAsIlhfNCJdLFsyLDFdLFswLDEsImIiXSxbMywyLCJjIiwyXSxbMywwXSxbNCwzXV0=
	\[\begin{tikzcd}
		X' & {X_1 \times_{\tilde{W}} W} & {W} \\
		& {X_1} & \tilde{W}
		\arrow[from=1-1, to=1-2]
		\arrow[from=1-2, to=1-3]
		\arrow["\alpha"', from=1-2, to=2-2]
		\arrow["\beta", from=1-3, to=2-3]
		\arrow[from=2-2, to=2-3]
	\end{tikzcd}\]
	the morphism $\beta$ is, by definition, a $G$-equivariant Nisnevich neighborhood of $Gz$; hence $\alpha$ is as well. Indeed the morphisms
	\[
	G \times^{S_z} \Spec(\kappa(z)) \longrightarrow X_1 \qquad \textnormal{and} \qquad  G \times^{S_z} \Spec(\kappa(z)) \longrightarrow  W
	\]
	give rise to the section $G \times^{S_z}\Spec(\kappa(z)) \to X_1 \times_{\tilde{W}} W$, because they form a commutative diagram over $\tilde{W}$.
	Since by construction $X_1 \to X$ is a $G$-equivariant Nisnevich neighborhood of $Gz$, it follows that $X_1 \times_{\tilde{W}} W \to X$ is also a $G$-equivariant Nisnevich neighborhood of $Gz$. Finally, $X' \to X$ is a $G$-equivariant Nisnevich neighborhood of $Gz$ because $X'$ is open in $X_1 \times_{\tilde{W}} W$ and it contains the orbit $Gz$. Then $X' \to X$ gives $\Gnum{1}$, the scheme $W$ gives $\Gnum{2}$, the $G$-representation $V$ gives $\Gnum{3}$ and $\varphi$ gives $\Gnum{4}$. The properties $\Pnum{1}, \Pnum{2}, \Pnum{3}, \Pnum{4}$ hold on $\varphi$ by construction, as well as the existence of the $G$-equivariant section $W \to X'$.
\end{proof}

This concludes the proof of Case \ref{subsection: case C} and thus also of Theorem \ref{theorem: gabber}.
\end{proof}

\vspace{0.5cm}
\subsection*{Acknowledgements} 
I thank Federico Scavia for his constant help, support, and insightful advice throughout the development and the writing of this work. I thank Jean Fasel for several helpful conversations and for insightful remarks that contributed to this work.

This project is co-funded by the European Union. Views and opinions expressed are however those of the author only and do not necessarily reflect those of the European Union. Neither the European Union nor the granting authority can be held responsible for them.

\vspace{0.5cm}

\begin{center}
\includegraphics[width=0.4\linewidth]{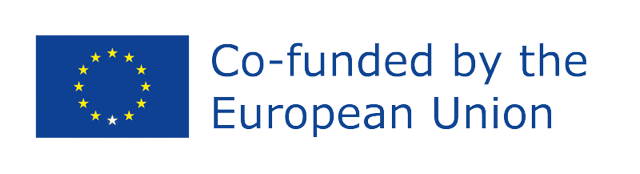}
\end{center}

\providecommand{\bysame}{\leavevmode\hbox to3em{\hrulefill}\thinspace}
\providecommand{\MR}{\relax\ifhmode\unskip\space\fi MR }
\providecommand{\MRhref}[2]{%
	\href{http://www.ams.org/mathscinet-getitem?mr=#1}{#2}
}
\providecommand{\href}[2]{#2}

\end{document}